\documentclass[11pt,reqno,xcolor=svgnames]{amsart}

\usepackage{amsmath, amsthm, amssymb, amsfonts, verbatim,tikz}

\usepackage{xspace,mathtools,bm}

\usepackage[pagebackref=true, colorlinks=true, citecolor=blue]{hyperref}

\usepackage{cleveref}

\def\0{{\bf 0}}

\def\R{{\mathbb R}}
\def\Z{{\mathbb Z}}

\def\N{{\mathbb N}}

\theoremstyle{plain}

\newtheorem{theorem}{Theorem}[section]
\newtheorem{cor}[theorem]{Corollary}
\newtheorem{prop}[theorem]{Proposition}

\newtheorem{lemma}[theorem]{Lemma}

\theoremstyle{definition}
\newtheorem{definition}[theorem]{Definition}

\newtheorem*{theorem*}{Theorem}

\newtheorem{remark}[theorem]{Remark}
\newtheorem*{remark*}{Remark}

\usepackage{tcolorbox}
\usepackage{tikz}
\tcbuselibrary{skins}
\usetikzlibrary{shadings}
\tcbset{
    myimage/.style={
        enhanced,
        overlay={
            \begin{scope}[shift={([xshift=1mm, yshift=7mm]frame.north west)}]
            \end{scope}}}}

\tcbset{
    skin=enhanced,
    fonttitle=\bfseries,
    interior style={white},
    segmentation style={black,solid,opacity=0.2,line width=1pt}}

\newtcolorbox{TitledBox}[2][]{
    myimage,              
    coltitle=black,       
    colbacktitle=white,   
    title=My title,
    attach boxed title to top center={
        yshift=-3mm,
        yshifttext=-1mm},
    attach boxed title to top left={
        xshift=1cm,
        yshift=-2mm},
    boxed title style={
        size=small},
    title={#2},#1}

\newcommand{\abs}[1]{\left\vert#1\right\vert}

\DeclareMathOperator{\dv}{div}

\DeclareMathOperator{\supp}{supp}
\newcommand{\iny}{\ensuremath{\infty}}
\newcommand{\grad}{\ensuremath{\nabla}}
\newcommand{\prt}{\ensuremath{\partial}}
\newcommand{\pdx}[2]{\frac{\prt #1}{\prt #2}}
\newcommand{\brac}[1]{\ensuremath{\left[ #1 \right]}}
\newcommand{\pr}[1]{\ensuremath{\left( #1 \right) }}

\newcommand{\norm}[1]{\ensuremath{\left\Vert #1 \right\Vert}}

\newcommand{\Cal}[1]{\ensuremath{\mathcal{#1}}}
\newcommand{\wh}{\widehat}
\newcommand{\diff}[2]{\frac{ d#1}{d#2}}
\newcommand{\al}{\alpha}

\newbool{HaveBBM}
\booltrue{HaveBBM}

\ifbool{HaveBBM}{
	\usepackage{bbm}
    }
    {
    }

\newcommand{\CharFunc}{
    \ifbool{HaveBBM}{
        \ensuremath{\mathbbm{1}}
        }
        {
        \ensuremath{\bm{1}}
        }
    }

\DeclareMathOperator{\PV}{p.v.} %

\allowdisplaybreaks

\crefname{cor}{Corollary}{Corollaries} 
									   
\crefname{lemma}{Lemma}{Lemmas}	       

\crefname{section}{Section}{Sections}
\Crefname{section}{Section}{Sections}

\crefname{appendix}{Appendix}{Appendices}
\Crefname{appendix}{Appendix}{Appendices}

\crefname{theorem}{Theorem}{Theorems}
\Crefname{theorem}{Theorem}{Theorems}

\crefname{prop}{Proposition}{Propositions}
\Crefname{prop}{Proposition}{Propositions}

\crefname{conj}{Conjecture}{Conjectures}
\Crefname{conj}{Conjecture}{Conjectures}

\crefname{definition}{Definition}{Definitions}
\Crefname{definition}{Definition}{Definitions}

\crefname{remark}{Remark}{Remarks}
\Crefname{remark}{Remark}{Remarks}

\crefname{assumption}{Assumption}{Assumptions}
\Crefname{assumption}{Assumption}{Assumptions}

\creflabelformat{enumi}{(#1)}
\creflabelformat{enumii}{(#1)}

\crefformat{equation}{(#2#1#3)}
\crefrangeformat{equation}{(#3#1#4) through (#5#2#6)}
\crefmultiformat{equation}
    {(#2#1#3)}%
    { and~(#2#1#3)}
    {, (#2#1#3)}
    { and~(#2#1#3)}

\newcommand{\NoteToSelf}[1]{
    }

\newcommand{\Holder}
    {H\"{o}lder }

\newcommand{\stardot}{\mathop{* \cdot}}

\renewcommand{\epsilon}{\varepsilon}
\newcommand{\eps}{\ensuremath{\varepsilon}}

\begin{document}

\raggedbottom

\numberwithin{equation}{section}

\title
    [Non-decaying solutions to dissipative SQG]
    {Non-decaying solutions to the 2D dissipative quasi-geostrophic equations}

\author{David M. Ambrose}
\address{Department of Mathematics, Drexel University}
\curraddr{}
\email{dma68@drexel.edu}

\author{Ryan Aschoff}
\address{Department of Mathematics, University of California, Riverside}
\curraddr{}
\email{ryan.aschoff@email.ucr.edu}

\author{Elaine Cozzi}
\address{Department of Mathematics, Oregon State University}
\curraddr{}
\email{cozzie@math.oregonstate.edu}

\author{James P. Kelliher}
\address{Department of Mathematics, University of California, Riverside}
\curraddr{}
\email{kelliher@math.ucr.edu}

\subjclass{Primary 76D03, 35Q86, 35A01, 35Q35} 
\date{August 13, 2025}


\keywords{Fluid mechanics, surface quasi-geostrophic equations, singular initial 
data, mild solutions, global solutions, non-decaying}

\begin{abstract}
We consider the surface quasi-geostrophic equation in two spatial dimensions, with
subcritical diffusion (i.e. with fractional diffusion of order $2\alpha$ for $\alpha>\frac{1}{2}$.)  We establish existence of solutions without
assuming either decay at spatial infinity or spatial periodicity.  
One obstacle is that for $L^{\infty}$ data, the constitutive law may not
be applicable, as Riesz transforms are unbounded.  However, for
$L^{\infty}$ initial data for which the constitutive law does converge, 
we demonstrate that there exists a unique solution
locally in time, and that the constitutive law continues to hold at positive
times. 
In the case that $\alpha\in(\frac{1}{2},1]$ and that the initial data has some 
smoothness (specifically, if the data is in $C^{2}$), 
we demonstrate a maximum principle and show that this
unique solution is actually classical and global in time.  
Then, a density argument allows us to show that mild solutions with only 
$L^{\infty}$ data are also global in time, and also possess this maximum 
principle.  Finally, we introduce a related problem 
in which we replace the usual constitutive law for the surface quasi-geostrophic
equation with a generalization of Sertfati type, and prove the same results for
this relaxed model.
\end{abstract}

\maketitle

\vspace{-2.5em}

\begin{center}

\end{center}

\tableofcontents

\section{Introduction}
The two-dimensional dissipative surface quasi-geostrophic equations (SQG) can be written, for $\nu > 0$, $\alpha>0$, and $\Lambda = (-\Delta)^{1/2}$, in strong form as,  
\begin{equation} \tag{$SQG$} \label{e:SQG}
	\begin{cases}
		\partial_t \theta + u \cdot \nabla \theta + \nu\Lambda^{2\alpha} \theta = 0
			& \text{in } [0, T] \times \mathbb{R}^2, \\
		u = -\nabla^\perp \Lambda^{-1} \theta
			& \text{in } [0, T] \times \mathbb{R}^2, \\
		\theta|_{t = 0} =  \theta_0
			& \text{in } \mathbb{R}^2.
	\end{cases}
\end{equation}
Without dissipation, this system was introduced by Constantin,
Majda, and Tabak to model atmospheric fluid flows and
as a two-dimensional analogy for the three-dimensional
Euler equations \cite{constantinMajdaTabak}.
 In the non-dissipative case, the existence of a smooth global solution (or singularity formation) remains an open question in general. We consider the question of local and global existence of solutions to the dissipative system \cref{e:SQG}, in the case that the data is non-decaying.
(We mention that authors differ on the choice of sign in the constitutive law, using $\pm \nabla^\perp \Lambda \theta$, but our choice agrees with that of  
\cite{constantinMajdaTabak}.)

The dissipative SQG system can be subcritical, critical,
or supercritical depending on the value of $\alpha.$ 
We consider the subcritical case, in which 
$\alpha>\frac{1}{2}.$  Other results for the subcritical case are 
\cite{constantinWu}, \cite{mayZahrouni},
\cite{Resnick1995}, \cite{Wu},  in which various local and global existence theorems are proved on the torus
or for decaying solutions in $\mathbb{R}^{2}.$  
Without attempting to provide an exhaustive list of
references, we mention that local and global existence results have also been proved in the critical 
($\alpha=\frac{1}{2}$) case
\cite{abidiHmidi},
\cite{kiselevEtAl}, 
and in the supercritical case ($0<\alpha<\frac{1}{2}$)
\cite{chaeLee}, \cite{chenMiaoZhang}, \cite{dongLi}.
None of these works focused on the question
considered here, which is existence theory in non-decaying function spaces such as 
$L^{\infty}.$

We can write the constitutive law, \cref{e:SQG}$_2$, as (see \cref{S:ConstLaw} for more details)
\begin{align}\label{e:ConstLawPV}
    u(t, x)
        &=  \PV K\ast \theta := \lim_{\substack{\epsilon\to 0 \\ R\to \infty}} \int_{\epsilon<|x-y|<R} K(x - y) \theta(t,y) \,
            dy,
\end{align}
where
\begin{align}\label{e:K}
    K(x)
        &:= - \frac{1}{2 \pi} \frac{x^\perp}{\abs{x}^3}
        = \grad^\perp \psi(x), \quad
        \psi(x) := -\frac{1}{2 \pi |x|}.
\end{align}

(For ease of notation, we will often abbreviate the principle value integral in \eqref{e:ConstLawPV}, writing $K\ast\theta$, only writing out the principal value integral when it is necessary to consider it carefully.)  We will be studying 
solutions $\theta\in L^{\infty}(\mathbb{R}^{2}),$ whereas SQG is more commonly
studied in $L^{2}(\mathbb{R}^{2})$ or similar function spaces.  A space
such as $L^{2}(\mathbb{R}^{2})$ has the advantage that $u$ is then clearly
defined; that is, the Riesz transforms in \cref{e:ConstLawPV} are well-defined
for $\theta\in L^{2}(\mathbb{R}^{2}),$ and return $u\in (L^{2}(\mathbb{R}^{2}))^{2}.$  A fundamental difficulty to overcome in our setting is
that Riesz transforms are unbounded on $L^{\infty},$ as is well-known.

The closest works in the literature to the present are the papers \cite{lazar13}, \cite{lazar15}.
In these works, Lazar studied dissipative SQG in the critical
case $\alpha=1/2,$ proving existence of local and global weak solutions.  The Lazar solutions
start from data in the space 
$L^{\infty}(\mathbb{R}^{2})\cap\Lambda^{s}(\dot{H}^{s}_{ul}(\mathbb{R}^{2})),$
i.e., the data is in $L^{\infty}$ but is also the $s^{\mathrm{th}}$ derivative of a function with
$s$ derivatives in the uniformly local $L^{2}$ space.  
This additional assumption on the data is made to 
induce oscillations, which allow the Riesz transforms to
converge.

We take two alternative approaches to make sense of the constitutive law
for non-decaying solutions.  First, while convolution with $K$ does not 
converge for many elements of $L^{\infty},$ we proceed for those elements
of $L^{\infty}$ for which the convolution does make sense.
That is, in our first approach, we take initial data
$\theta_{0}\in L^{\infty}(\mathbb{R}^{2})$ for which convolution with
$K$ yields a result that is also in $L^\iny(\R^2)$.  (We in fact need slightly 
more than this, in that we also ask that the convolution converge uniformly,
in a sense to be made precise in \cref{preliminary} below.)  We give 
several examples of such $\theta_{0}$ in \cref{preliminary}.
Our second approach will be to introduce a relaxation of the constitutive law.

We define a new notion of mild solution for \cref{e:SQG} which 
allows us to solve for $\theta$ and $u$ simultaneously, and in a sense, 
no longer requires us to reconstruct $u$ from $\theta$ at every instant.
Our starting point in making this mild formulation
is the work of Marchand \cite{Marchand} and Marchand and Lemari\'{e}-Rieusset \cite{Marchand2005}, who
write a mild formulation of \cref{e:SQG} with a 
single integral equation,
\begin{align*}
    \theta(t,x)
        = (G_\alpha(t)\,\theta_0)(x) \;-\;\int_0^t \nabla G_\alpha(t-s)\,\cdot\bigl(\theta\,(K*\theta)\bigr)(s,x)\,ds,
\end{align*}
where $G_\alpha(t)$ is the fractional heat semigroup defined in (\ref{def:heatsemigroup}). We replace $K*\theta$ with $u$ and couple this equation to a second integral relation for $u,$ as in the next definition.

\begin{definition}\label{D:mildsolution}
Let $T\in(0,\infty)$ and let $(\theta_0, u_0) \in (L^\iny(\R^2))^{3}$.
A pair
\[
(\theta,u)\;\in\;L^\infty\bigl([0,T]\times\R^2\bigr)\;\times\;\bigl(L^\infty([0,T]\times\R^2)\bigr)^2
\]
is called a \emph{mild solution} to \cref{e:SQG} on $[0,T]$ if for each $t\in(0,T]$ one has
\begin{equation}\label{SQGintegral}
\begin{split}
    \theta(t,x)
        &= (G_\alpha(t)\,\theta_0)(x)
          \;-\;\int_0^t \nabla G_\alpha(t-s)\,\cdot\bigl(\theta u\bigr)(s,x)\,ds,\\
    u(t,x)
        &= (G_\alpha(t)\,u_0)(x)
          \;-\;\int_0^t (K*\nabla G_\alpha(t-s))\,\cdot\bigl(\theta u\bigr)(s,x)\,ds.
\end{split}
\end{equation}
\end{definition}

Through \cref{D:mildsolution} we have circumvented the question of whether we can convolve the kernel $K$ with an $L^{\infty}$ function.  
Given bounded $\theta_{0}$ and $u_{0},$
we will show that there exists $T>0$ and $\theta$ and
$u$ such that \cref{SQGintegral} holds.
This definition does not enforce any relationship between $\theta_{0}$ and 
$u_{0}$, allowing us to utilize both of our approaches to this 
question without changing the definition of a mild solution.
As we have said, our first approach is to consider $\theta_{0}$ for which 
there is a $u_{0}$ such that $u_{0}=K*\theta_{0},$ with an additional assumption
of uniform convergence.  For the resulting solutions $(\theta,u),$ we can
then show that $u=K*\theta$ at positive times, as one would desire.  The 
following is this local existence theorem.
\begin{theorem} \label{thm:existence-of-solutions}
Let $\theta_0\in L^\infty (\R^2)$, $u_0\in (L^\infty(\R^2))^2$, and fix $\alpha>\frac{1}{2}$. For some $T>0$ there exists a unique mild solution $(\theta, u) \in (C((0, T]; L^{\infty}(\R^2))^3$ of \cref{e:SQG} with $\theta|_{t=0} = \theta_0$ and $u|_{t=0} = u_0$. 
Moreover, if $u_0 = \PV K \ast \theta_0$ with the principal value integral converging uniformly in the sense of \cref{e:InitCondUnif}, the solution $(\theta,u)$ satisfies \eqref{e:SQG}$_2$. 
\end{theorem}
\begin{proof}
    See \cref{existence}.
\end{proof}

We also study solutions with higher regularity.  For notational clarity, we introduce three classical spaces. Let $k\in \N$ and denote by $C^k$ the space of $k$ times differentiable functions. 
Let $C^k_b(\R^2)$ denote the Banach space of $k$-times continuously differentiable 
bounded functions with norm
$$\|f\|_{C^k_b} := \sum_{\substack{\beta\in \N^2, |\beta|\leq k}} \|D^\beta f\|_{L^\infty} <\infty.$$ 
For $k=0$, $C^0_b(\R^2)$ denotes the space of bounded continuous functions. For $0<\gamma<1$, we denote the $\gamma$-H\"older continuous functions $C^\gamma(\R^2)$ as the subspace of $L^\infty(\R^2)$ bounded by the norm 
$$\|f\|_{C^\gamma} := \|f\|_{L^\infty} + \sup_{x\neq y} \frac{|f(x)-f(y)|}{|x-y|^\gamma}.$$ 
 
We are able to show that if the initial data is
in the $C^k_b$ spaces for $k\geq 1$, then the mild solutions preserve this regularity.
 \begin{theorem} \label{thm:regularity-of-theta-and-u}
    Let
    $\alpha>\frac{1}{2}$
    and $k\geq 1$. Select $\theta_0 \in C_b^k(\R^2)$, $u_0 \in   C_b^k(\R^2))^2$ satisfying $u_0=K\ast \theta_0$. Let $(\theta, u)$ be the mild solution given by \cref{thm:existence-of-solutions} which exists up to time $T$. For all $t < T$, we have $D^\beta \theta(t) \in L^\infty(\R^2)$ and $D^\beta u(t) \in (L^\infty(\R^2))^2$ for any multi-index $\beta \in \N^2$ such that $|\beta|\leq k$.
\end{theorem}
\begin{proof}
    See \cref{proof:regularity-of-theta-and-u}.
\end{proof}

Our main results, however, are the following two global existence theorems.
The first of these theorems states that if the initial data is at least twice continuously differentiable, the solution can be extended for all time.  We now restrict to 
$\alpha\in\left(\frac{1}{2},1\right]$ so that we can use
maximum principles.
\begin{theorem} \label{thm:global-in-time-solution}  
    Let $\alpha\in\left(\frac{1}{2},1\right]$ be given.
    Suppose $k\geq2$ and $(\theta_0,u_0)\in (C^k_b(\R^2))^3$. If one has $u_0 = \PV K\ast \theta_0$, then for all $T>0$, there exists a classical solution (i.e., pointwise solution) to \eqref{e:SQG} on $[0,T]$ with $(\theta(t),u(t))\in (C_b^2(\R^2))^3$ for all $t \in [0,T]$. Further, $\theta$ is uniformly bounded by its initial data, i.e. $\|\theta\|_{L^\infty_{t,x}}\leq \|\theta_0\|_{L^\infty_x}.$ 
\end{theorem}
\begin{proof}
    See \cref{proof-of-global-in-time-solution}.
\end{proof}

Exploiting the $C^k_b$ solutions of \cref{thm:global-in-time-solution} and
a density argument, we then obtain our second main result, 
the extension of the solutions with $L^\infty$ data to an arbitrary time.
\begin{theorem} \label{thm:L-infty-global-solution}
    Let $\alpha\in\left(\frac{1}{2},1\right]$ be given.
    Suppose that $(\theta_0,u_0)\in (L^\infty(\mathbb{R}^{2}))^{3}$. If one has $u_0 = \PV K\ast \theta_0$, then for arbitrary time $T>0$, there exists a mild solution $(\theta,u)\in (L^\infty([0,T]\times\R^2))^3$ to \eqref{e:SQG} on $[0,T]$. Further, $\theta$ is uniformly bounded by its initial data, i.e. $\|\theta\|_{L^\infty_{t,x}}\leq \|\theta_0\|_{L^\infty_x}.$  
\end{theorem}
\begin{proof}
    See \cref{proof:L-infty-global-solution}.
\end{proof}

Finally, we describe
our second approach to making sense of the constitutive law
for non-decaying solutions, which is to introduce a 
version of \cref{e:SQG} in which the constitutive law is relaxed; we call this
a Serfati-type constitutive law.
For the two-dimensional Euler equations, Serfati proved the existence and uniqueness
of solutions with velocity and vorticity in $L^{\infty}(\mathbb{R}^{2})$ 
\cite{serfati} (see also 
\cite{AKLN} for further exposition on Serfati's work).  
With vorticity in $L^{\infty},$ the constitutive law (which, for the Euler equations, is the Biot-Savart law) 
cannot be used to obtain the velocity from vorticity; in its place, Serfati used an integral identity relating the velocity and vorticity for a solution to the Euler equations that applies in the case of bounded vorticity. 
Three of the present authors and Erickson have used an analogue of the Serfati identity for 
inviscid SQG \cite{ACEK} to prove local existence of solutions of SQG
in uniformly local Sobolev spaces and H\"{o}lder spaces.
In addition to this analogue of the Serfati integral identity for SQG, the work \cite{ACEK} also 
uses a related relaxation of the 
constitutive law involving the 
Littlewood-Paley operators $\dot{\Delta}_{j}$ (see \cref{preliminary} 
below for details
on the Littlewood-Paley blocks). In the present work, we use the Littlewood-Paley relaxation of the
constitutive law.
In the dissipative case, the new system is
\begin{align}\label{ssqg}
\tag{$SSQG$}
\begin{cases}
\prt_t \theta + u\cdot\grad\theta + \nu\,\Lambda^{2\alpha}\theta = 0
  &\text{in }[0,T]\times\R^2,\\
\dot\Delta_j u = (\dot\Delta_j K)*\theta
  &\text{in }[0,T]\times\R^2,\ \forall\,j\in\Z,\\
\theta|_{t=0} = \theta_0
  &\text{in }\R^2,\\
\dot\Delta_j u_0 = (\dot\Delta_j K)*\theta_0
  &\forall\,j\in\Z,\\
\dv u_0 = 0 &\text{in }\R^2.
\end{cases}
\end{align}
If the initial data $(\theta_{0},u_{0})$ satisfies \cref{ssqg}$_{4},$
then a pair $(\theta,u)$ satisfying \eqref{SQGintegral} is called a \emph{mild solution} to the Serfati‐type surface quasi‐geostrophic system \eqref{ssqg}.  
Of course, this relaxation allows additional data 
$\theta_{0}\in L^{\infty}(\mathbb{R}^{2})$ to be treated.  
That is, if $\theta_{0}\in L^{\infty}(\mathbb{R}^{2})$ is such that there 
exists $u_{0}\in(L^{\infty}(\mathbb{R}^{2}))^{2}$ for which
\eqref{ssqg}$_{2}$ holds, then we can prove the same
results as for \eqref{e:SQG}.  This is the content of the next theorem, which is treated only briefly in the remaining text, as the proof 
follows immediately from the proofs of the above theorems.
\begin{theorem} \label{cor:solution-to-sqg}
Let $\theta_0\in L^\infty (\R^2)$ and $u_0\in (L^\infty(\R^2))^2$ satisfy the initial data condition $\dot{\Delta}_j u_0 = (\dot{\Delta}_j K) * \theta_0$ for all $j\in \Z$.  Then the conclusions of \cref{thm:existence-of-solutions}, \cref{thm:regularity-of-theta-and-u}, \cref{thm:global-in-time-solution}, and \cref{thm:L-infty-global-solution}
hold true for \eqref{ssqg}.
\end{theorem}
\begin{proof}
    See \cref{proof:solution-ssqg}.
\end{proof}

The organization of the paper is as follows.
In \cref{preliminary}, 
we give various definitions and 
estimates related to the fractional Laplacian 
$\Lambda^{2\alpha}$, the fractional heat kernel,
and the kernel $K$.
In \cref{sec:mild-solution}, we establish various
{\it a priori} continuity and differentiability properties
of mild solutions of \eqref{e:SQG} and \eqref{ssqg}.
In \cref{existence}, we prove the existence of a 
finite-in-time mild solution with $L^{\infty}$ data; this 
is the proof of \cref{thm:existence-of-solutions}.
In \cref{sec:regularity}, 
we prove that $C^{k}$-regularity of 
the initial data is propagated in time by the mild 
solution; this is the proof of 
\cref{thm:regularity-of-theta-and-u}.
In \cref{sec:globalSection}, we prove that given $C^{2}$ initial data, the solution can be shown to exist for all time and corresponds to a classical solution; this is the proof
of \cref{thm:global-in-time-solution}. After proving the existence of 
global-in-time solutions with regular data, we conclude the section by extending 
solutions with $L^\infty$ initial data to be global in time,
proving \cref{thm:L-infty-global-solution}.
Finally, in \cref{sec:appendix}, we prove a technical property of the fractional Laplacian.

\section{ $L^1(\R^2)$ estimates for various operators}\label{preliminary}

\subsection{Littlewood-Paley operators}\label{S:LPOperators}
We begin with an overview of the Littlewood-Paley operators and some of their properties.  There exist two functions ${\chi}, {\varphi} \in \mathcal{S}(\R^d)$ with supp $\wh{\chi}\subset \{\xi\in \R^d: |\xi |\leq 1\}$ and supp $\wh{\varphi}\subset \{\xi\in \R^d: \frac{1}{2} \leq|\xi |\leq \frac{3}{2} \}$, such that, setting $\varphi_j(x) = 2^{jd} \varphi(2^j x)$ for all $j \in \Z$,
\begin{equation*}
\begin{split}
	&\wh{\chi}+ \sum_{j\geq 0} \wh{\varphi}_j
		= \wh{\chi} + \sum_{j\geq 0} \wh{\varphi}(2^{-j} \cdot) 
		\equiv 1.
\end{split}
\end{equation*}

For $n\in\Z$, define ${\chi}_n \in \mathcal{S}(\R^d)$ in terms of its Fourier transform ${\wh{\chi}}_n$, where ${\wh{\chi}}_n$ satisfies 
\begin{equation*}
{\wh{\chi}}_n (\xi) =   \wh{\chi}(\xi) + \sum_{j=0}^n \wh{\varphi}_j(\xi)
\end{equation*}
for all $\xi\in\R^d$.  For $f\in \mathcal{S}'(\R^d)$, define the operator $S_n$ by  
\begin{equation*}
S_n f = {{\chi}}_n \ast f.
\end{equation*}
Finally, for $f\in \mathcal{S}'(\R^d)$ and $j\in\Z$, define the inhomogeneous Littlewood-Paley operators ${\Delta}_j$ by
\begin{align*}
    \begin{matrix}
        &\Delta_j f  = \left\{
            \begin{matrix}
                     0, \qquad j<-1\\
                 \chi\ast f,  \qquad j=-1\\
                \varphi_j\ast f, \qquad j\geq 0,
            \end{matrix}
            \right.
    \end{matrix}
\end{align*}
and, for all $j\in\Z$, define the homogeneous Littlewood-Paley operators $\dot{\Delta}_j$ by
\begin{equation*}
    \dot{\Delta}_j f = {\varphi}_j \ast f.
\end{equation*}  
Note that $\dot{\Delta}_j f = {\Delta}_j f$ when $j\geq 0$.

\subsection{Kernels for the fractional heat equation}

We first introduce the fractional Laplacian, $\Lambda^{2\alpha}$, defined in $\Cal{S}'(\R^2)$ via its Fourier transform as follows:
\begin{equation} \label{def:frac-laplacian-fourier}
\mathcal{F} (\Lambda^{2\alpha} f) (\xi) = -|\xi|^{\alpha} \mathcal{F} f(\xi).
\end{equation}
In most of this work, we will use the operator $\Lambda^{2\alpha}$ with the above representation.  In Section \ref{sec:globalSection}, however, we apply an alternative definition of the fractional Laplacian. Specifically, we let \( \Lambda^{2\alpha}_I f(x) \) denote the singular integral operator defined formally by
\begin{equation}\label{def:frac-laplacian-sio}
    \Lambda^{2\alpha}_I f(x) := c_{\alpha} \lim_{\epsilon \to 0^+} \int_{|x - y| > \epsilon} \frac{f(x) - f(y)}{|x - y|^{2 + 2\alpha}} \, dy,
\end{equation}
where \( c_\alpha > 0 \) is a normalization constant,  $$c_\alpha = \frac{4^\alpha \Gamma(1+\alpha)}{\pi |\Gamma(-\alpha)|}.$$
We define the domain \( \operatorname{Dom}(\Lambda^{2\alpha}_I, L^\infty) \) to be the set of functions \( f \in L^\infty(\mathbb{R}^2) \) for which the above limit exists and is finite for almost every \( x \in \mathbb{R}^2 \). As a consequence of \cref{lem:fractional-laplacian-agreeing}, the definitions of $\Lambda^{2\alpha}$ and $\Lambda^{2\alpha}_I$ coincide on $C^2_b(\R^2)$.
 
We now recall that the fractional heat kernel $g_\alpha(t,x)$ is the solution to 
\begin{align}\label{ln:frac-heat-property}
    (\partial_t +\nu \Lambda^{2\alpha}) g_{\alpha} = 0
\end{align} 
 on $\R^2$ subject to the initial condition $g_\alpha(0,x) = \delta(x)$. It is easily seen that the Fourier transform of $g_\alpha(t,x)$ is given by
\begin{equation}\label{def:frac-g}
    \wh{g}_\alpha(t, \xi)
        = \int_{\R^2}
            g_\alpha (t, x) e^{-i \xi \cdot x} \, dx
        = e^{-\nu|\xi|^{2\alpha} t}.
\end{equation}
Often in what follows, we omit the spatial argument for notational convenience and write $g_\alpha(t) := g_\alpha(t, \cdot)$. We denote the fractional heat semigroup acting on $f\in L^{\infty}(\R^2)$ by
\begin{align} \label{def:heatsemigroup}
    G_\alpha (t) f = g_\alpha (t)\ast f.  
\end{align}
While it is known that $g_\alpha(t,x)$ cannot be written in terms of an elementary function, for $\alpha \in [0,1]$ and $t>0$, $g_\alpha(t,x)$ is a nonnegative and non-increasing radially-symmetric smooth function which satisfies the dilation relation
\begin{equation} \label{ln:dilation-relation-frac-g}
    g_\alpha(t,x) = (\nu t)^{-\frac{1}{\al}} g_\alpha\left(1,x(\nu t)^{-\frac{1}{2\alpha}}\right).
\end{equation}

\subsection{The constitutive law}\label{S:ConstLaw} For $\theta \in L^\iny(\R^2)$, $K * \theta $ is not well-defined as a convolution.  Hence, we cannot obtain a constitutive law in the form $u = K * \theta$. To evade this restriction, we consider two modifications of the constitutive law, which we now describe.

For the first modification, we let $A_{\eps, R}(0)$ denote the annulus centered at the origin with inner radius $\eps$ and outer radius $R$.

Given $\theta \in L^\iny(\R^2)$,  by \cref{e:ConstLawPV} we mean, more precisely,
\begin{align}\label{e:PVKDetail}
    \begin{split}
    \PV K * \theta(x)
        &:= \lim_{(\eps, R) \to (0, \iny)}
            (\CharFunc_{A_{\eps, R}(0)} K) * \theta(x) \\
        &= \lim_{(\eps, R) \to (0, \iny)}
        \int_{\epsilon<|x-y|<R} K(x - y) \theta(y)
            \, dy
    \end{split}
\end{align}
    for any $x \in \R^2$.
    
\begin{definition}\label{def:convergence-uniformly-over-annuli}
    We say that $\PV K * \theta(x)$ converges \textit{uniformly over annuli} if
    \begin{align}\label{e:InitCondUnif}
        \sup_{(r, R) \in (0, \iny)^2}
            \norm{(\CharFunc_{A_{r, R}(0)} K) * f}_{L^\iny}
            < \iny.
    \end{align}    
\end{definition}

To motivate the well-definedness of the above constitutive law, we provide a few classes of initial data as examples.

\begin{remark} Define the homogeneous Besov space $\dot{B}_{p,q}^s(\R^2)$ as in Definition 2.15 of \cite{BahouriCheminDanchin2011}. One has $\dot{B}_{\infty,1}^{0}(\R^2)\subset L^\infty(\R^2)$ (see Proposition 2.39 of \cite{BahouriCheminDanchin2011}). Moreover, as a consequence of the proof of Theorem 1.3 of \cite{Sawano2020} applied to $\dot{B}_{\infty,1}^{0}(\R^2)$, one can show that $\|R_i \theta_0\|_{\dot{B}_{\infty,1}^{0}(\R^2)}\leq \|\theta_0\|_{\dot{B}_{\infty,1}^{0}(\R^2)}$ for $\theta_0\in \dot{B}_{\infty,1}^0(\R^2)$. A simple application of \cref{L:KgradgalUnifBound} below to the Littlewood Paley expansion of $\theta_0$ reveals that the convergence is uniform over annuli. Hence, defining $u_0 = \PV K\ast \theta_0$, the pair $(\theta_0,u_0)$ satisfies the constitutive law.
\end{remark}
\begin{remark}
If $f(x) \in C_c^\infty(\R^2)$ is any bump function, then the function $$\theta_0(x) = \sum_{j=1}^\infty \sum_{i=1}^\infty f(x_1-2^i,x_2-2^j)$$ has a bounded Riesz transform which converges uniformly over annuli. Hence, setting $u_0 = \PV K\ast \theta_0$, the pair $(\theta_0,u_0)$ satisfies the constitutive law. Any variation of this function involving a sufficiently sparse sum of uniformly bounded compactly supported functions also suffices.
\end{remark}
\begin{remark}
    For $r > 0$ a non-integer, let $\psi \in C^{r+1}_0(\R^2)$ (that is, the space of functions with $r+1$ bounded derivatives vanishing at infinity) and set $u_0 = \nabla^\perp \psi$ and $\theta_0 = \Lambda \psi$. Then $(\theta,u)$ will be a pair satisfying the constitutive law.
\end{remark}

In \cref{L:KUniformConv}, we show that $K$ and $g_\alpha$ commute, but only if we assume a kind of uniform convergence over annuli of the principal value integral.
\begin{lemma}\label{L:KUniformConv}
    Let $f \in L^\iny(\R^2)$ and suppose that $\PV K * f$ exists, converges uniformly over annuli (see \cref{def:convergence-uniformly-over-annuli}), and lies in $ L^\iny(\R^2)$. Then
    \begin{align}\label{e:NiceInitConst}
        g_\al(t) * (\PV K * f)(x)
            = \PV K * (g_\al(t) * f)(x).
    \end{align}
\end{lemma}
\begin{proof}
    Fix $x \in \R^2$ and write $\lim_{r, R}$ for the limit in \cref{e:PVKDetail}.
    Then,
    \begin{align*}
         g_\al(t) * &(\PV K * f)(x)
         	= \int_{\R^2}
            	\brac{g_\al(t, x - y)
                \lim_{r, R}  ((\CharFunc_{A_{r, R}(0)} K)
                * f(y))} \, dy \\
         	&= \int_{\R^2}
            	\lim_{r, R} \brac{g_\al(t, x - y)
                ((\CharFunc_{A_{r, R}(0)} K) * f)(y)} \, dy \\
            &= \lim_{r, R} 
             	\int_{\R^2}
            	\brac{g_\al(t, x - y)
                ((\CharFunc_{A_{r, R}(0)} K) * f)(y)} \, dy  \\
            &= \lim_{r, R} \,
				(g_\al(t) * ((\CharFunc_{A_{r, R}(0)} K)
					* f))(x) \\
            &= \lim_{r, R} \,
				((\CharFunc_{A_{r, R}(0)} K) * (g_\al(t)
					* f))(x) \\
			&= \PV K * (g_\al(t) * f).
    \end{align*}
    Above, we used \cref{e:InitCondUnif} with
    $g_\al(t) \in L^1(\R^2)$ to take the limit outside of
    the integral via the Dominated Convergence Theorem.  We were able to commute the order of convolution before taking the limit,
    using that $g_\al(t)$ and $\CharFunc_{r, R}(0) K$ are
    in $L^1$ while $f \in L^\iny$.
\end{proof}

In \cref{cor:solution-to-sqg}, we take a different approach to the constitutive law that avoids using the principal value integral of $K$. We consider, instead, a constitutive law in the form $\dot{\Delta}_j u = (\dot{\Delta}_j K) * \theta$, taking advantage of the following simple lemma:
\begin{lemma}\label{L:DeltaK}
    For all $j \in \Z$, $\dot{\Delta}_j K \in \Cal{S}(\R^2)$.
\end{lemma}
\begin{proof}
    Taking the Fourier transform,
    \begin{align*}
        \Cal{F} (\dot{\Delta}_j K)(\xi)
            &= \Cal{F}(\varphi_j * K)(\xi)
            = \wh{\varphi}_j(\xi) \wh{K}(\xi)
            = -i\wh{\varphi}_j(\xi) \frac{\xi^\perp}{\abs{\xi}}.
    \end{align*}
    Because $\wh{\varphi}_j$ is in $C^\iny(\R^2)$ and supported in an annulus, $\Cal{F} (\dot{\Delta}_j K)$ belongs to $\Cal{S}(\R^2)$, and hence so does $\dot{\Delta}_j K$.
\end{proof}

\subsection{Convolution bounds}
We will use \cref{lem:operator-L1-estimates} below to bound the integrands appearing in our formulation of the mild solution to \cref{e:SQG}.  The proof of \cref{lem:operator-L1-estimates} relies on the following lemma from \cite{WCU}:

\begin{lemma}[\cite{WCU}]\label{Ward}
Fix $\epsilon \in [0,1)$ and an integer $N \ge 1$, and assume $f$ is a differentiable function on $\R^d$ which satisfies
\begin{itemize}
    \setlength{\itemsep}{0.35em} 
    \item[(1)]
        $|f(x)| \leq C(1+|x|)^{-d-N+\epsilon}$,
    \item[(2)]
        $|D^{\beta}f(x)|\leq C(1+|x|)^{-d -N -1 +\epsilon}$ for all $|\beta|=1$,
    \item[(3)]
        $\int x^{\beta} f(x) \, dx = 0$ for all $|\beta|<N$.
\end{itemize}
Then for each $i$, $1\leq i\leq d$, 
\begin{align*}
    \abs
        {\PV \int_{\R^2} K^i(x - y) f(y) \, dy}
            &\le C(1+|x|)^{-d-N +\epsilon + \delta}
\end{align*}
for every $\delta$ satisfying $0 < \delta < 1-\epsilon$. If, in addition, $f \in \Cal{S}(\R^2)$ then the same bound applies to $K * f$.
\end{lemma}
\begin{proof}
This follows from Theorem 3.2 of \cite{WCU}.
\end{proof}
We will utilize estimates on the fractional heat kernel to bound the mild solution to \cref{e:SQG}. In order to derive estimates on $\nabla G_\alpha$ as in \eqref{SQGintegral}, we first examine the derivatives of $g_\alpha$ at $t=1$.
\begin{lemma} \label{lem:L1-boundedness-of-frac-heat-kernel}
Let $\alpha>0$ then the $k^{th}$-order derivatives of the fractional heat kernel are in $L^1_x(\R^2)$ for all $k \in \N$. Specifically,
\begin{align*}
    \|\nabla^k g_{\alpha}(1,x) \|_{L^1_x} < \infty \quad \text{and} \quad   \|x\cdot\nabla^2 g_{\alpha}(1,x) \|_{L^1_x} < \infty.
\end{align*}
\begin{proof}
     Using the Bochner's Relation (see Corollary on page 72 of \cite{stein1970singular}), we have
 \begin{equation} \label{ln:derivative-dimension-relation}
     \frac{\partial}{\partial x_j} g_\alpha^d (1,x) = -\frac{x_j}{2\pi} g_\alpha ^{d+2} (1, \tilde{x}),
 \end{equation}
 where $g_\alpha^d(t,x)$ is the heat kernel in $d$ dimensions and $\tilde{x} = (x_1,x_2,\ldots, x_d, 0,0)$. We also take note of the well-known estimate (see \cite{bogdan2010heat}),
 \begin{equation} \label{ln:frac-g-pointwise-bound}
     g_\alpha^d (1,x) \lesssim \min \{1,|x|^{-d-2\alpha}\}.
 \end{equation}
 It then follows,
 \begin{equation}\label{e:gradgalBound}
 \begin{split}
     \left |\nabla^k g_\alpha (1,x) \right| \leq \frac{|x|^k}{2\pi} g_\alpha^{d+2k}(1,\tilde{x}) \lesssim |x|^k \min\{1, |x|^{-d-2k-2\alpha}\} = \min\{|x|^k, |x|^{-d-k-2\alpha}\},     
 \end{split}
 \end{equation}
 which is clearly integrable. By the same argument, the second estimate follows easily.
\end{proof}
\end{lemma}

We are now in a position to show that the $L^1_x$-norms of spatial derivatives of $g_\alpha$ decay in time, and that convolution with $K$ preserves this decay.
\begin{lemma} \label{lem:operator-L1-estimates}
	For $\alpha>0$ and $\beta$ a multi-index with $k:=|\beta|$ and $j = 1, 2$, we have
	\begin{align*}
		\begin{cases}
			\norm{D^\beta g_\alpha(t)}_{L^1}
				\le C(k,\nu,\alpha ) t^{-k/(2\alpha)}
				&\text{ for all } \beta, \\
			\norm{K^j * D^\beta g_\alpha(t)}_{L^1}
				\le C(k,\nu, \alpha) t^{-k/(2\alpha)}
				&\text{ for all } k \text{ odd}.
		\end{cases}
	\end{align*}
\end{lemma}

\begin{proof} 
We will use the dilation relation \eqref{ln:dilation-relation-frac-g} to manipulate the norm of $D^\beta g_\alpha(t)$ and apply the chain rule (see also Lemma 6 of \cite{schonbek_qg_farfield} for an alternate proof).  We write
\begin{equation*}
\begin{split}
    &\norm{D^\beta g_\alpha(t)}_{L^1_x}
        = \int_{\R^2} \left | D^\beta g_\alpha(t,x) \right| dx
        = \int_{\R^2} \left | D^\beta \left(t^{-1/\alpha} g_\alpha\left(1, x t^{-\frac{1}{2\alpha}}\right)\right) \right|\, dx \\
        &\qquad\qquad = \int_{\R^2}  \left | t^{-\left(\frac{1}{\alpha} +\frac{k}{2\alpha}\right)} (D^\beta g_\alpha)(1,xt^{-\frac{1}{2\alpha}}) \right| \, dx.
\end{split}
\end{equation*}
Finally, the substitution $u = xt^{-\frac{1}{2\alpha}}$ yields
\begin{align*}
    \norm{D^\beta g_\alpha(t)}_{L^1_x} = t^{-k/(2\alpha)} \|D^\beta g_\alpha(1) \|_{L^1_x}.
\end{align*}
     Invoking \cref{lem:L1-boundedness-of-frac-heat-kernel}, we have $D^\beta g_\alpha(1)\in L^1_x(\R^2)$. To prove the second bound in Lemma \ref{lem:operator-L1-estimates}, we use the dilation relation \eqref{ln:dilation-relation-frac-g} to write
    \begin{equation}\label{ln:K-D-beta-frac-g}
    \begin{split}
                K^j * D^\beta g_\alpha(t, x)
            &= K^j * D^\beta \left(t^{-1/\alpha}g_\alpha(1,xt^{-\frac{1}{2\alpha}})\right) \\
            &= t^{-\left( \frac{1}{\alpha} + \frac{k}{2\alpha}\right)}
                \PV\int_{\R^2} K^j(z) (D^\beta g_\alpha(1))
                    \pr{t^{-\frac{1}{2\alpha}}(x - z)} \, dz\\
                    &= t^{-\left( \frac{1}{\alpha} + \frac{k}{2\alpha}\right)}
                (K^j * D^\beta g_{\alpha}(1)) \pr{xt^{-\frac{1}{2\alpha}}}.
    \end{split}
    \end{equation}

    To get the last equality, we used the scaling property $K(z) = b^2 K(bz)$ for any $b > 0$ and the integral substitution $u = t^{-\frac{1}{2\alpha}}z$.
    Applying the substitution $u = t^{-\frac{1}{2\alpha}}x$ again then gives
    \begin{align}\label{ln:norm-k-D-beta-frac-g}
        \norm{K^j * D^\beta g_\alpha(t)}_{L^1_x}
            &= t^{-k/(2\alpha)} \norm{K^j * D^\beta g_\alpha(1)}_{L^1_x}.
    \end{align}

    By symmetry of $g_\alpha$ and the fact that symmetry is preserved under the Fourier transform, when $k = |\beta| $ is odd, we see that $D^\beta g_\al(1)$ satisfies the hypotheses of \cref{Ward}, including the critical moment condition in (3) for $N=1$, i.e.,
    \begin{equation}
        \int_{\R^2} D^\beta g_\alpha(1) \, dx = 0.
    \end{equation}
    The bound for $K^j * D^\beta g_\alpha(t)$ follows.
\end{proof}
\begin{remark} \label{rmk:abs-nabla-frac-g}
    Notice that for $\Lambda g_\alpha(t,x)$, the above argument holds to produce the bound
    \begin{equation*}
        \left \|\Lambda g_\alpha(t,x) \right \|_{L^1_x} \leq C(\nu,\alpha) t^{-\frac{1}{2\alpha}}.
    \end{equation*}
\end{remark}

\begin{remark}
    In fact, the estimate \eqref{lem:operator-L1-estimates} holds for $K^j \ast D^\beta g_\alpha(t)$ for even $k:=|\beta|$ as a result of \cref{C:g1KInBesov}. 
\end{remark}
When invoking the fractional heat equation property \eqref{ln:frac-heat-property}, it will also be necessary to have estimates available for the fractional Laplacian applied to $g_\alpha$.
\begin{lemma} \label{lem:lambda-nabla-frac-g-estimate} For $\alpha>0$, there exists a constant $C(\nu,\alpha)>0$ such that
	\begin{align*}
		\begin{cases}
			\|\Lambda^{2\alpha} \nabla g_\alpha(t)  \|_{L^1_x} \leq  C(\nu, \alpha) t^{-\left (1+ \frac{1}{2\alpha}\right )}, \\
			\|K\ast \Lambda^{2\alpha} \nabla g_\alpha(t)  \|_{L^1_x} \leq  C(\nu, \alpha) t^{-\left (1+ \frac{1}{2\alpha}\right )}.
		\end{cases}
	\end{align*}
\begin{proof}
    First, we devise a dilation law for the tensor of $k^\text{th}$-derivatives. Observe for $k\in \N$,
    \begin{equation} \label{ln:dilation-relation-k-deriv-frac-g}
    \begin{split}
        &\nabla^k g_\alpha(t,x) = \nabla^k \left [t^{-1/\alpha} g_\alpha (1,xt^{-\frac{1}{2\alpha}}) \right]
        = t^{-1/\alpha}\nabla^k g_\alpha (1,xt^{-\frac{1}{2\alpha}}) \\
        &\qquad\qquad = t^{-\frac{2+k}{2\alpha}}(\nabla^k g_\alpha) (1,xt^{-\frac{1}{2\alpha}}).
    \end{split}
    \end{equation}
    We evaluate the norm of $\Lambda^{2\alpha} \nabla g_\alpha(t)$ using the fractional heat property \eqref{ln:frac-heat-property} and the dilation relation \eqref{ln:dilation-relation-k-deriv-frac-g} for $k=1$.  We write
    \begin{equation} \label{ln:lambda-2-alpha-frac-g}
    \begin{split}
        &\Lambda^{2\alpha} \nabla g_\alpha(t,x) = -\frac{1}{\nu} \frac{\partial}{\partial t} \nabla g_\alpha(t,x) = -\frac{1}{\nu} \frac{\partial}{\partial t} \left[ t^{-\frac{3}{2\alpha}} (\nabla g_\alpha)(1, x t^{-\frac{1}{2\alpha}}) \right] \\
        &= -\frac{1}{\nu} \left[ -\frac{3}{2\alpha} t^{-\left(\frac{3}{2\alpha}+1\right)} (\nabla g_\alpha)(1, x t^{-\frac{1}{2\alpha}}) 
        - \frac{1}{2\alpha} t^{-\left(\frac{3}{2\alpha}+\frac{1}{2\alpha}+1\right)} x \cdot (\nabla^2 g_\alpha)(1, x t^{-\frac{1}{2\alpha}}) \right].
    \end{split}
\end{equation}
    We use (\ref{ln:lambda-2-alpha-frac-g}) to compute the $L^1_x$-norm of $\Lambda^{2\alpha} \nabla g_\alpha(t)$, giving
    \begin{equation} \label{ln:int-lambda-2-alpha-frac-g}
        \begin{split}
            \int_{\R^2} &\left |\Lambda^{2\alpha} \nabla g_\alpha(t,x) \right| \, dx \\
            &=  \frac{1}{\nu}\int_{\R^2} \left |
                -\frac{3 t^{-\left(\frac{3}{2\alpha}+1\right)}}{2\alpha}  \nabla g_\alpha(1,xt^{-\frac{1}{2\alpha}})
                -\frac{t^{-\left(\frac{3}{2\alpha}+\frac{1}{2\alpha}+1\right)}}{2\alpha}x\cdot(\nabla^2 g_\alpha)(1,xt^{-\frac{1}{2\alpha}}) \right | dx \\
            &\leq \frac{3}{2\alpha\nu}t^{-1}\int_{\R^2} \left | \nabla g_\alpha(t,x) \right | dx + \frac{1}{2\alpha \nu} t^{-\left(\frac{3}{2\alpha}+\frac{1}{2\alpha}+1\right)}\int_{\R^2} \left| x\cdot (\nabla^2 g_\alpha)(1,xt^{-\frac{1}{2\alpha}})\right | \,dx.
        \end{split}
    \end{equation}
    We apply the integral substitution $u = xt^{-\frac{1}{2\alpha}}$ to the second term of \eqref{ln:int-lambda-2-alpha-frac-g}, which yields 
    \begin{equation} \label{ln:lambda-nabla-g-alpha-second-term}
        \begin{split}    
           &\frac{1}{2\alpha \nu} t^{-\left(\frac{3}{2\alpha}+\frac{1}{2\alpha}+1\right)}\int_{\R^2} \left | x\cdot (\nabla^2 g_\alpha)(1,xt^{-\frac{1}{2\alpha}})\right| \,dx  \\
           &\qquad\qquad =\frac{1}{2\alpha \nu } t^{-\left (\frac{1}{2\alpha} +1\right )}\int_{\R^2}\left | x\cdot (\nabla^2 g_\alpha)(1,x)\right |\,dx.
        \end{split}
    \end{equation}
   The integral in \eqref{ln:lambda-nabla-g-alpha-second-term} is bounded by \cref{lem:L1-boundedness-of-frac-heat-kernel}.  Using \cref{lem:operator-L1-estimates}, we can further simplify the first term of \eqref{ln:lambda-2-alpha-frac-g}. Indeed,
    \begin{equation} \label{ln:lambda-nabla-g-alpha-first-term}
        \frac{3}{2\alpha\nu}t^{-1}\int_{\R^2} \left | \nabla g_\alpha(t,x) \right | dx \leq \frac{3}{2\alpha\nu} C  t^{-\left (\frac{1}{2\alpha} +1\right )}.
    \end{equation}
    One can then combine \eqref{ln:lambda-nabla-g-alpha-second-term} and \eqref{ln:lambda-nabla-g-alpha-first-term} to conclude that
    \begin{equation*}
        \int_{\R^2} \left |\Lambda^{2\alpha} \nabla g_\alpha(t,x) \right| dx \leq C(\nu,\alpha) t^{-\left (1+\frac{1}{2\alpha}\right )}.
    \end{equation*}
    We now prove the second inequality of Lemma \ref{lem:lambda-nabla-frac-g-estimate}. For $j=1,2$, we expand $K^j\ast \Lambda^{2\alpha} \nabla g_\alpha(t)$ using \eqref{ln:dilation-relation-frac-g}.  We write
    \begin{equation*}
    \begin{split}
        K^j \ast \Lambda^{2\alpha} \nabla g_\alpha (t,x)
        &= -\frac{1}{\nu} K^j \ast\bigg[ -\frac{3}{2\alpha} t^{-\left(\frac{3}{2\alpha}+1\right)} \nabla g_\alpha(1, x t^{-\frac{1}{2\alpha}}) 
        \\
        &\qquad- \frac{1}{2\alpha} t^{-\left(\frac{3}{2\alpha}+\frac{1}{2\alpha}+1\right)} x \cdot (\nabla^2 g_\alpha)(1, x t^{-\frac{1}{2\alpha}}) \bigg].
    \end{split}    
    \end{equation*}
    Therefore,
    \begin{equation} \label{ln:K-lambda-nabla-frac-g}
    \begin{split}
        \|K^j \ast \Lambda^{2\alpha} \nabla g_\alpha (t,x) \|_{L^1_x} &\leq \frac{3}{2\alpha \nu} t^{-\left(\frac{3}{2\alpha}+1\right)} \bigg \|\underbrace{K^j \ast  \nabla g_\alpha(1, x t^{-\frac{1}{2\alpha}}) }_{=:[I]} \bigg\|_{L^1_x} \\ 
        &+ \frac{1}{2\alpha\nu} t^{-\left(\frac{3}{2\alpha}+\frac{1}{2\alpha}+1\right)} \bigg \| \underbrace{K^j \ast \left(x \cdot (\nabla^2 g_\alpha)(1, x t^{-\frac{1}{2\alpha}})\right)}_{=:[II]} \bigg\|_{L^1_x}.
    \end{split}
    \end{equation}
    For the first integral of \eqref{ln:K-lambda-nabla-frac-g}, using $K(x) = b^2K(bx)$ for $b>0$ and the integral substitution $u =zt^{-\frac{1}{2\alpha}}$, we have
    \begin{equation*}
    \begin{split}
        [I]  &= \PV \int_{\R^2} K^j(z) \nabla g_\alpha\left(1,(x-z)t^{-\frac{1}{2\alpha}}\right) dz\\
        &\qquad = t^{-1/\alpha} \PV \int_{\R^2}  K^j(zt^{-1/(\alpha)}) \nabla g_\alpha\left(1,(x-z)t^{-\frac{1}{2\alpha}}\right)dz \\
        & \qquad = \PV \int_{\R^2}  K^j(z) \nabla g_\alpha\left(1,xt^{-\frac{1}{2\alpha}}-z\right)dz \\
        & \qquad = \left(K^j \ast \nabla g_\alpha(1)\right)(xt^{-\frac{1}{2\alpha}})
    \end{split}
    \end{equation*}
     Using the integral substitution $u= xt^{-\frac{1}{2\alpha}}$, we arrive at
    \begin{equation}\label{ln:K-nabla-frac-g-first-term}
    \begin{split}
        \frac{3}{2\alpha \nu} t^{-\left(\frac{3}{2\alpha}+1\right)} \left \|[I]  \right\|_{L^1_x} &= \frac{3}{2\alpha \nu} t^{-\left(\frac{3}{2\alpha}+1\right)}\left\|\left(K^j \ast \nabla g_\alpha(1)\right)(xt^{-\frac{1}{2\alpha}})\right\|_{L^1_x} \\
        &\leq \frac{3}{2\alpha \nu}t^{-\left( \frac{1}{2\alpha}+1\right)}  \left\| K\ast \nabla g_\alpha(1,x)\right\|_{L^1_x}.
    \end{split}
    \end{equation}
    By \cref{lem:operator-L1-estimates}, $\|[I](x)\|_{L^1_x}$ is bounded.
    For $[II]$,
    as in the proof of \cref{lem:operator-L1-estimates}, we use the scaling property $K(x) = b^2 K(bx)$ for $b=t^{-\frac{1}{2\alpha}}$ and the integral substitution $u = zt^{-\frac{1}{2\alpha}}$ to write
    \begin{equation*}
    \begin{split}
        [II] &= \PV \int_{\R^2} K^j (z) \left[(x-z)\cdot(\nabla^2 g_\alpha) \left(1,(x-z)t^{-\frac{1}{2\alpha}} \right) \right] dz \\
        &= t^{-\frac{1}{2\alpha}} \PV\int_{\R^2} K^j (t^{-\frac{1}{2\alpha}}) \left[ t^{-\frac{1}{2\alpha}} \cdot (\nabla^2 g_\alpha) \left( 1,(x-z)t^{-\frac{1}{2\alpha}} \right)\right] dz \\
        & = t^{\frac{1}{2\alpha}} \PV\int_{\R^2} K^j(z)\left[ (xt^{-\frac{1}{2\alpha}} - z) \cdot (\nabla^2 g_\alpha)\left( 1,(xt^{-\frac{1}{2\alpha}}-z) \right) \right] dz \\
        & = t^{\frac{1}{2\alpha}} K^j \ast \left(x\cdot\nabla^2 g_\alpha(1)\right)(t^{-\frac{1}{2\alpha}}x).
    \end{split}
    \end{equation*}
    Observe that $x\cdot (\nabla^2 g_\alpha)$ is an odd function via a standard symmetric argument involving the Fourier transform, and thus satisfies the conditions of \cref{Ward}. Hence,
    \begin{equation}\label{ln:K-x-nabla-2-frac-g}
    \begin{split}
     \| [II]\|_{L^1_x} &= t^{\frac{1}{2\alpha}}\left \|  K^j \ast \left(x\cdot\nabla^2 g_\alpha(1)\right)(t^{-\frac{1}{2\alpha}}x)  \right \|_{L^1_x}  \\
     & = t^{\frac{3}{2\alpha}} \left \|  K^j \ast \left(x\cdot\nabla^2 g_\alpha(1)\right)(x)  \right \|_{L^1_x} <\infty.
    \end{split}
    \end{equation}  In the final line above, we use the $L^1_x$ boundedness of $K^j \ast \left(x\cdot\nabla^2 g_\alpha(1)\right)(x)$ from \cref{Ward}. 
    Thus, we can substitute the terms \eqref{ln:K-x-nabla-2-frac-g} and \eqref{ln:K-nabla-frac-g-first-term} into \eqref{ln:K-lambda-nabla-frac-g}. We conclude
    \begin{align*}
        \|K^j \ast \Lambda^{2\alpha} \nabla g_\alpha (t,x) \|_{L^1_x} &\leq C(\alpha,\nu) t^{-\left(1+ \frac{1}{2\alpha}\right)} \left ( \|[I]\|_{L^1_x} + \|[II]\|_{L^1_x}\right) <\infty,
    \end{align*}
    as desired.
\end{proof}
\end{lemma}
Having proven that the expression $K\ast\nabla G_\alpha$ is well-defined, we are now in a position to clarify the interpretation of the mild formulation of a solution. Recalling \cref{SQGintegral}, we defined a mild solution as a pair of functions $(\theta,u)$ which satisfy the coupled equations:
\begin{equation}\label{SQGintegral-repeat}
\begin{split}
    \theta(t,x)
        &= G_\alpha (t)\theta_0(x)-\int_0^t \nabla G_\alpha (t-s) \cdot (\theta u)(s, x) \, ds, \\
    u(t,x)
        &= G_\alpha(t)u_0(x)  - \int_0^t (K \ast\nabla G_\alpha(t-s)) \cdot (\theta u)(s, x) \, ds.
\end{split}
\end{equation}
\begin{definition}\label{D:KGpsi}
    We interpret the integrand in \cref{SQGintegral-repeat}$_2$ as follows. For any $t > 0$ and vector function $p:\R^2 \to \R^2$, define
    \begin{align*}
        (K \ast\nabla G_\alpha(t)) \cdot p
            &:= (K * \grad g_\alpha(t)) \stardot p,
    \end{align*}
    where $K * \grad g_\alpha(t)$ is the $2$-tensor, or $2 \times 2$ matrix,
    \begin{align} \label{def:K-ast-nabla-g-2-tensor}
        \begin{bmatrix}
            K^1 \ast \partial_{x_1} g_\alpha(t)
            & K^2 \ast \partial_{x_1} g_\alpha(t) \\
            K^1 \ast \partial_{x_2} g_\alpha(t)
            & K^2 \ast \partial_{x_2} g_\alpha(t)
        \end{bmatrix},
    \end{align}
    and the $\stardot$ product is given by
    $$ (K * \grad g_\alpha(t)) \stardot p(s) = \begin{bmatrix}
            K^1 \ast \partial_{x_1} g_\alpha(t) \ast p_1 + K^2 \ast \partial_{x_1} g_\alpha(t)\ast p_2 \\
            K^1 \ast \partial_{x_2} g_\alpha(t)\ast p_1+  K^2 \ast \partial_{x_2} g_\alpha(t)\ast p_2
        \end{bmatrix}.$$
    Each component of the matrix in \eqref{def:K-ast-nabla-g-2-tensor}, by \cref{lem:operator-L1-estimates}, lies in $L^1(\R^2)$. For any $\varphi \in (L^\iny(\R^2))^2$, we then define $(K * \grad g_\alpha(t)) \stardot \varphi(x)$ as the convolution of an $L^1$ matrix-field with an $L^\iny$-vector field, resulting in an $L^\iny$-vector field, whose $i^{th}$ component, $i = 1, 2$, is given by $ \sum_{j=1,2}  (K^j *\prt_i g_\alpha(t)) * \varphi_j$.
\end{definition}

It will be necessary to compute the time derivative of $K\ast \nabla G_\alpha(t)$ to prove properties about the time-regularity of solutions to \eqref{e:SQG}, as we do in \cref{L:K-heat-equation-property}.
\begin{lemma}\label{L:K-heat-equation-property}
Let $\alpha>0$. $K$ and $\partial_t$ commute, in the sense that, for all $t>0$,
    \begin{align*}
        \frac{\partial}{\partial t} \left(K * \nabla g_\alpha(t)\right)
            &= -\nu K\ast \Lambda^{2\alpha} \nabla g_\alpha (t) \in L^1(\R^2).
    \end{align*}
\end{lemma}

    First, we prove a lemma, which we will use in the proof of \cref{L:K-heat-equation-property}.

    \begin{lemma}\label{L:NeededLemma}
        Let $f, \grad f \in C^\iny(\R^2) \cap L^1(\R^2) \cap L^\iny(\R^2)$. Then, with $\psi$ as in (\ref{e:K}),
        \begin{align}\label{e:NeededEq}
        	\begin{split}
	        	\PV K * f &= \psi * \grad^\perp f, \\
    	        \Cal{F} \pr{\PV K * f} &= \wh{K} \wh{f}.
            \end{split}
        \end{align}
    \end{lemma}
    \begin{proof}
        First observe that $\psi$ is locally integrable. By the assumption on $f$, we have $\psi\ast \partial_j f$ exists, $j = 1, 2$, because
        \begin{align}\label{e:psistarprtjf}
            \psi * \prt_j f
                &= (\CharFunc_{B_1(0)} \psi) * \prt_j f
                    + (\CharFunc_{B_1(0)^C} \psi) * \prt_j f,
        \end{align}
        and the first term is an $L^1$ function convolved with an $L^\iny$ function, while the second part is an $L^\iny$ function convolved with an $L^1$ function.
        Hence,
        \begin{align*}
        	\psi * \prt_j f(x)
				&= \lim_{\eps \to 0} I_{\eps}(x), \quad
            I_{\eps}(x)
                := \int_{B_\eps(x)^C}
                    \psi(x - y) \prt_j f(y)
                        \, dy.
        \end{align*}
        Integrating by parts, using that
        $\prt_{y_j} \psi(x - y) = - \prt_j \psi(x - y)$,
        and considering the orientation of the boundary,
        \begin{align*}
            I_{\eps}(x)
                &= \int_{B_\eps(x)^C}
                    \prt_j \psi(x - y) f(y)
                        \, dy
                    - \int_{\prt B_\eps(x)}
                        f(y) \psi(x - y) n^i_y
                        \, d s(y)
                =: I_1(x) + I_2(x).
        \end{align*}
        On the boundary, $\psi(x - y) = - (1/2 \pi) \eps^{-1}$, so
        \begin{align*}
            I_2(x)
                &= \frac{1}{2 \pi \eps} \int_{\prt B_\eps(x)}
                        ((f(y) - f(x)) n^i_y
                        \, d s(y)
                    + \frac{1}{2 \pi \eps}
                        f(x) 
                        \int_{\prt B_\eps(x)}
                        n^i_y
                        \, d s(y).
        \end{align*}
        The second term integrates to zero, so
        \begin{align*}
            \abs{I_2(x)}
                &\le \frac{1}{2 \pi \eps}
                    2 \pi \eps
                    \sup_{y\in \prt B_\eps(x)}
                        \abs{f(x) - f(y)}
                \to 0
                    \text{ as } \eps \to 0,
        \end{align*}
        because $f$ is continuous.
        Hence, $I_{\eps} \to \PV \prt_j \psi * f$ as $\eps \to 0$,
        giving \cref{e:NeededEq}$_1$.
        
        From \cref{e:NeededEq}$_1$ with \cref{e:psistarprtjf}, $K = (-\prt_2 \psi, \prt_1 \psi) = \grad^\perp \psi$, and  using the linearity of the Fourier transform,
        \begin{align*}
        	\Cal{F} \pr{\PV K * f}
				&= \wh{\psi} \wh{\grad^\perp f}
				= i \xi^\perp \wh{\psi} \wh{f}
				= \wh{\grad^\perp \psi} \wh{f}
                    = \wh{K} \wh{f},
        \end{align*}
 	giving \cref{e:NeededEq}$_2$.
     \end{proof}

\begin{remark}\label{R:NeededLemma}
    We can also write the first conclusion of \cref{L:NeededLemma} as
    $\PV K^j * f = (-1)^j \psi * \prt_{3 - j} \psi*f$.
\end{remark}
\begin{proof}[\textbf{Proof of \cref{L:K-heat-equation-property}}]
    
    Write $\langle \cdot, \cdot \rangle$ for the pairing between $\Cal{D}'((0, T) \times \R^2)$ and $\Cal{D} ((0, T) \times \R^2)$. Then
    \begin{align*}
        \left \langle\pdx{}{t} (K * \grad g_\al), \varphi\right \rangle
            &= - \left \langle K * \grad g_\al, \pdx{}{t} \varphi\right \rangle.
    \end{align*}
    By \cref{lem:operator-L1-estimates}, $K * \grad g_\al \in L^1(\R^2)$. Also, $g_\al, \grad g_\al \in C^\iny(\R^2) \cap L^1(\R^2) \cap L^\iny(\R^2)$, since the Fourier transform of $g_\al$ has spatial exponential decay (see \cref{def:frac-g}). Thus 
    the pairing above is an actual integral of continuous functions, and we have, using \cref{ln:frac-heat-property,L:NeededLemma},
    \begin{align*}
        \left \langle\pdx{}{t} (K* \grad g_\al), \varphi\right\rangle
             &= - \int_0^T \int_{\R^2}
                    (K * \grad g_\al)(t, x)
                    \pdx{}{t} \varphi(t, x) \, dt \, dx \\
            &= - \int_0^T
                \left(K * \grad g_\al, \pdx{\varphi}{t} \right) \, dt
            = - \int_0^T
                \left( \Cal{F}(K * \grad g_\al),
                    \Cal{F}\pr{\pdx{\varphi}{t}} \right) \, dt \\
            &= - \int_0^T
                \left( \wh{K} \wh{\grad g_\al},
                    \pdx{}{t} \wh{\varphi}(t) \right) \, dt
            = - \int_{\R^2} \int_0^T
                \wh{K}(\xi) \wh{\grad g_\al}(t, \xi)
                    \pdx{}{t} \wh{\varphi}(t, \xi)
                    \, dt \, d \xi \\
            &= \int_{\R^2} \int_0^T
                \wh{K}(\xi) \pdx{}{t} \wh{\grad g_\al}(t, \xi)
                    \wh{\varphi}(t, \xi)
                    \, dt \, d \xi \\
            &= \int_{\R^2} \int_0^T
                \wh{K}(\xi) \Cal{F} \pr{\grad \pdx{}{t} g_\al}(t, \xi)
                    \wh{\varphi}(t, \xi)
                    \, dt \, d \xi \\
            &= -\nu \int_{\R^2} \int_0^T
                \wh{K}(\xi) \Cal{F} \pr{\Lambda^{2 \al} \grad g_\al}(t, \xi)
                    \wh{\varphi}(t, \xi)
                    \, dt \, d \xi \\
            &= -\nu \int_{\R^2} \int_0^T
                \Cal{F}
                    \pr{K * \Lambda^{2 \al} \grad g_\al}(t, \xi)
                    \wh{\varphi}(t, \xi)
                    \, dt \, d \xi \\
            &= -\nu \int_0^T
                    \pr{K * \Lambda^{2 \al} 
                        \grad g_\al(t, \xi),
                    \varphi(t, \xi)}
                    \, dt
            = \left \langle-\nu K * \Lambda^{2 \al}  \grad g_\al,
                \varphi\right \rangle.
    \end{align*}
    Above, we applied the Fubini-Tonelli theorem to interchange the order of integration.
\end{proof}

When showing that the constitutive law holds for the mild formulation, we will use the limited form of commutativity of $\PV K$ and $\grad g_\alpha$ given in \cref{L:KUniformConv}. We will also need a form of associativity, as we will see in the proof of \cref{prop:constitutive-law-holds}; for this, we rely on a lemma similar to \cref{Ward} to provide a dominating $L^1_x$ bound.

\begin{lemma}[\cite{WCU}]\label{L:WardSimple} For any function $f$ satisfying the conditions of \cref{Ward} there exists a nonnegative function $F \in L^1(\R^d)$ such that for all $x\in \R^2$,
\begin{align*}
    \abs
        {(\CharFunc_{A_{r, R}(0)} K^j) * f(x)}
            &\le F(x), \quad 1 \le j \le d,
\end{align*}
holds uniformly over $0 < r \le R < \iny$.
\end{lemma}
\begin{proof}
This follows from a straightforward adaptation of the proof of Theorem 3.2 of \cite{WCU}, which gives the result (as an explicit decay bound) for $K^j$.

\end{proof}

\begin{lemma} \label{L:KgradgalUnifBound}
    Let $\alpha>0$. There exists a constant $C = C(\nu, \al)$ such that for all $0 < r < R < \iny$, 
	\begin{align*}
			\norm{(\CharFunc_{A_{r, R}(0)} K^j) * \prt_i g_\alpha(t)}_{L^1}
				\le Ct^{-1/(2\alpha)}.
	\end{align*}
\end{lemma}
\begin{proof}
    The proof is a simple adaptation of \cref{lem:operator-L1-estimates}, using the equality $\CharFunc_{A_{r, R}(0)}(az) = \CharFunc_{A_{r/a, R/a}(0)}(z)$ for any $a > 0$, and using the uniform bound in \cref{L:WardSimple} in place of \cref{Ward}.
\end{proof}

\section{Properties of mild solutions to \cref{e:SQG} and \cref{ssqg}}\label{sec:mild-solution}
\noindent
In this section, we establish some properties of mild solutions to \cref{e:SQG} and \cref{ssqg} as in \cref{D:mildsolution}. 

Our formulation of a mild solution does not fully incorporate a form of the constitutive law; however, in \cref{P:time-regularity}, we show that if the constitutive law holds initially, then it will hold for all time. In \cref{P:SQGMotivation}, we motivate \cref{D:mildsolution} more fully by showing that a sufficiently regular mild solution is, in fact, a classical solution. 
    
When proving the pointwise regularity of the solution, we must establish that the divergence-free condition on $u$ holds for all time if $\dv u_0 = 0$. To that end, we need the following technical result. 
\begin{lemma}\label{L:divuZero}
    Suppose that $f \in (L^1([0, T] \times \R^2))^2$ with $\dv f(t) = 0$ in $\Cal{S}'(\R^2)$ for all $t \in [0, T]$. Then for all $\psi \in L^\iny([0, T] \times \R^2)$ and $t \in [0, T]$,
    \begin{align*}
        \dv \int_0^t (f * \psi)(s, x) \, ds
            = 0
            \text{ in } \Cal{S}'(\R^2).
    \end{align*}
\end{lemma}
\begin{proof}
    For fixed $t\in [0,T]$, let $I(x) := \int_0^t (f *\psi)(s, x) \, ds$. Note that $I$ is in $(L^\infty(\R^2))^2$. Then for any $\varphi \in \Cal{S}(\R^2)$ we can apply the Tonelli-Fubini theorem to give,
    \begin{align*}
        (\dv I, \varphi)
            &= -(I, \grad \varphi)
            = -\int_0^t \int_{\R^2} (f * \psi)(s, x)
                \cdot \grad \varphi(x)
                \, dx \, ds \\
            &= -\int_0^t \int_{\R^2} \int_{\R^2}
                f(s, x - y) \psi(s, y)
                \cdot \grad \varphi(x)
                \, dy \, dx \, ds \\
            &= -\int_0^t \int_{\R^2} \int_{\R^2}
                f(s, x - y) \cdot \grad \varphi(x) \, \psi(s, y)
                \, dx \, dy \, ds \\
            &= \int_0^t \int_{\R^2}\int_{\R^2} (\dv f)(s,x-y)       \varphi(x) \psi(s,y) dx\, dy\,  ds  \\
            &= \int_0^t (\dv f(s) * \tilde{\varphi}, \tilde{\psi}(s)) \, ds
            = 0.
            \qedhere
    \end{align*}
\end{proof}
Here, $\tilde{\varphi}$ and $\tilde{\psi}$ are reflected versions of $\phi$ and 
$\psi,$ namely $\tilde{\varphi}(x)=\phi(-x)$ and $\tilde{\psi}(y)=\psi(-y).$

\subsection{Continuity in space}
We turn now, in \cref{P:gamma-holder-regularity}, to showing that mild solutions gain some \Holder regularity immediately after time zero. 
\begin{prop}\label{P:gamma-holder-regularity}
    Suppose that $(\theta,u)$ is a mild solution to \cref{e:SQG} on $[0,T]$. If $\alpha>1/2$ then for all $t\in (0,T]$ and $0<\gamma<2\alpha -1$, $(\theta(t,\cdot), u(t,\cdot))\in (C^{\gamma}(\R^2))^3$.
\end{prop}

We first establish a series of lemmas.
\begin{lemma}\label{L:fgConvBound}
    Let $f \in L^1(\R^2) \cap C^\gamma(\R^2)$ for some $\gamma \in (0, 1)$ and $g \in L^\iny(\R^2)$. Then
    \begin{align*}
        \norm{f * g}_{C^\gamma}
            &\le \norm{f}_{\widetilde{C}^\gamma} \norm{g}_{L^\iny},
    \end{align*}
    where
    \begin{align*}
        \norm{f}_{\widetilde{C}^\gamma}
            &:= \sup_{x\neq y}\int_{\R^2}
                \frac{\abs{f(x - z) - f(y - z)}}
                    {\abs{x - y}^\gamma}              
                    \, dz.
    \end{align*}
\end{lemma}
\begin{proof}
    Let $x, y \in \R^2$. Then
    \begin{align*}
        &\frac{\abs{f * g(x) - f * g(y)}}
                    {\abs{x - y}^\gamma}
            = \abs{\int_{\R^2}
                \frac{f(x - z) - f(y - z)}
                    {\abs{x - y}^\gamma}              
                    g(z) \, dz
                }
            \le
            \norm{f}_{\widetilde{C}^\gamma} \norm{g}_{L^\iny}.
        \qedhere
    \end{align*}
\end{proof}

Applying \cref{L:fgConvBound} to obtain \cref{P:gamma-holder-regularity} comes down to showing that $g_\al(t)$ and $\PV K * \grad g_\al(t)$ have a $\widetilde{C}^\gamma$ seminorm that scales sufficiently well in time.

\begin{lemma}\label{L:WideCScaling}
    Let $F \in \widetilde{C}^\gamma$, $\gamma \in (0, 1)$, and for any $r > 0$ define $F_r(\cdot) = r^{2 a} F(r^a \cdot)$ for some fixed $a > 0$. Then $F_r \in \widetilde{C}^\gamma$ with
    \begin{align*}
        \norm{F_r}_{\widetilde{C}^\gamma}
            \le r^{a \gamma} \norm{F}_{\widetilde{C}^\gamma}.
    \end{align*}
\end{lemma}
\begin{proof}
     For any distinct $x, y \in \R^2$,
    \begin{align*}
        \int_{\R^2}
            &\frac{\abs{F_r(x - z) - F_r(y - z)}}
                    {\abs{x - y}^\gamma}              
                    \, dz
            = r^{2a} \int_{\R^2}
                \frac{\abs{F(r^a(x - z)) - F(r^a(y - z))}}
                    {\abs{r^a x - r^a y}^\gamma}
                    r^{a \gamma}             
                    \, dz \\
            &= r^{2a + a \gamma} r^{-2a} \int_{\R^2}
                \frac{\abs{F(r^a x - w) - F(r^a y - w)}}
                    {\abs{r^a x - r^a y}^\gamma}             
                    \, dw
            \le r^{a \gamma} \norm{F}_{\widetilde{C}^\gamma}.
            \qedhere
    \end{align*}
\end{proof}

\begin{lemma}\label{L:CgammaW11Embedding}
    If $f \in W^{1, 1}(\R^2)$ then for all $\gamma\in(0,1),$ we have
    $f \in \widetilde{C}^\gamma$ with
    $\norm{f}_{\widetilde{C}^\gamma} \le 2 \norm{f}_{W^{1, 1}}$.
\end{lemma}
\begin{proof}
    Assume first that $\abs{x - y} \le 1$ with $x \ne y$, and let $h = (y - x)/\abs{x - y}$.  Then
    \begin{align*}
        \int_{\R^2} &\frac{\abs{f(x - z) - f(y - z)}}
                {\abs{x - y}^\gamma}              
                    \, dz
            \le
            \int_{\R^2} \frac{\abs{f(x - z) - f(y - z)}}
                {\abs{x - y}}              
                    \, dz \\
            &= \frac{1}{\abs{x - y}}
                \int_{\R^2}
                \abs{\int_0^{\abs{x - y}}
                \grad f(x + sh - z) \cdot h
                    \, ds} \, dz
                    \\
            &\le \frac{1}{\abs{x - y}}
                \int_0^{\abs{x - y}}
                \int_{\R^2}
                \abs{\grad f(x + sh - z)}
                    \, dz \, ds \\
            &= \frac{1}{\abs{x - y}}
                \int_0^{\abs{x - y}}
                \norm{\grad f}_{L^1}
                    \, ds
            =\norm{\grad f}_{L^1},
    \end{align*}
    where we used the translation invariance of the $L^1$ norm.
    If $\abs{x - y} > 1$ then
     \begin{align*}
        \int_{\R^2} \frac{\abs{f(x - z) - f(y - z)}}
                    {\abs{x - y}^\gamma}              
                    \, dz
            &\le \int_{\R^2}
                \abs{f(x - z) - f(y - z)}             
                    \, dz
            \le 2 \norm{f}_{L^1},
    \end{align*}
    which completes the proof.
\end{proof}

\begin{cor}\label{C:g1KInBesov}
    Each of $g_\al(1)$, $\grad g_\al(1)$, and $\PV K * \grad g_\al(1)$ have finite $\widetilde{C}^\gamma$ norm.
\end{cor}
\begin{proof}
    For $g_\al(1)$, $\grad g_\al(1)$, apply
    \cref{L:CgammaW11Embedding} to \cref{lem:L1-boundedness-of-frac-heat-kernel}.
    
    We know that $\PV K * \grad g_\al(1) \in L^1$ by \cref{lem:operator-L1-estimates}. To bound its derivatives, let $\varphi \in C^\iny_c(\R^2)$  take values in $[0, 1]$ with $\varphi \equiv 1$ on $B_1(0)$ and
    supported on $B_2(0)$. Then, much as in \cref{e:psistarprtjf} of \cref{L:NeededLemma} (and see \cref{R:NeededLemma}), for $j = 1, 2$ setting $k = 3 - j$,
    \begin{align*}
        \prt_j (\PV K * \grad g_\al(1))
            &= \abs{\psi * \prt_k \grad g_\al(1)}
            = \abs{(\varphi \psi) * \prt_k \grad g_\al(1)
                    + ((1 - \varphi) \psi) * \prt_k \grad g_\al(1)} \\
            &= \abs{(\varphi \psi) * \prt_k \grad g_\al(1)
                    + (\prt_k (1 - \varphi) \psi) * \grad g_\al(1)},
    \end{align*}
    so
    \begin{align*}
        \norm{\prt_j (\PV K * \grad g_\al(1))}_{L^1}
            &\le \norm{\varphi \psi}_{L^1} \norm{\prt_k \grad g_\al(1)}_{L^1}
                +  \norm{\prt_k ((1 - \varphi) \psi)}_{L^1}
                    \norm{\grad g_\al(1)}_{L^1},
    \end{align*}
    which is finite by \cref{lem:L1-boundedness-of-frac-heat-kernel}. Hence, we can apply \cref{L:CgammaW11Embedding} again to give that $\PV K * \grad g_\al(1)$ has a finite $\widetilde{C}^\gamma$ norm.
\end{proof}

\begin{lemma}\label{L:galKBesovBounds}
    For any $t > 0$,
    \begin{align*}
        \norm{g_\al(t)}_{\widetilde{C}^\gamma}
            &\le C t^{-\frac{\gamma}{2 \al}}, \\ 
        \norm{\grad g_\al(t)}_{\widetilde{C}^\gamma}
            &\le C t^{-\frac{1 + \gamma}{2 \al}}, \\ 
        \norm{\PV K * \grad g_\al(t)}_{\widetilde{C}^\gamma}
            &\le C t^{-\frac{1 + \gamma}{2 \al}}.
    \end{align*}
\end{lemma}
\begin{proof}
    The bounds on $\norm{g_\al(t)}_{\widetilde{C}^\gamma}$ and $\norm{\grad g_\al(t)}_{\widetilde{C}^\gamma}$ follow from applying \cref{L:WideCScaling,C:g1KInBesov} with the scaling given by \cref{ln:dilation-relation-frac-g}. The bound on $\norm{\PV K * \grad g_\al(t)}_{\widetilde{C}^\gamma}$ follows from applying \cref{L:WideCScaling,C:g1KInBesov} with the scaling given by \cref{ln:K-D-beta-frac-g}, noting the additional factor in that equation of $t^{-\frac{1}{2 \al}}$.
\end{proof}

\begin{proof}[\textbf{Proof of \cref{P:gamma-holder-regularity}}]
    Applying \cref{L:fgConvBound,L:galKBesovBounds} to \cref{SQGintegral}, we have
    \begin{align*}
        \norm{\theta(t)}_{C^\gamma}
            &\le C t^{-\frac{\gamma}{2 \al}} \norm{\theta_0}_{L^\iny}
                + C \int_0^t
                    (t - s)^{-\frac{1 + \gamma}{2 \al}}
                     \norm{(\theta u)(s)}_{L^\iny} \,ds,
    \end{align*}
    which is finite since $\gamma\in(0,2\alpha-1).$ 
    The bound on $u(t)$ is obtained the same way.
\end{proof}

\subsection{Continuity in time}

With the above result, we can now establish the regularity in time of the mild solution.
\begin{prop}\label{P:time-regularity}
    Suppose that $(\theta,u)$ is a mild solution to \cref{e:SQG} on $[0,T]$. Then we have the following:
    \begin{enumerate}
        \item  If $\dot{\Delta}_j u_0 = (\dot{\Delta}_j K) * \theta_0$ for all $j\in \Z$, then $\dot{\Delta}_j u(t) = (\dot{\Delta}_j K) * \theta(t)$ for all $t \in [0, T]$ and all $j\in\Z$.

        \item\label{i:uthetaCont}
        Let $\alpha>1/2$, then $(\theta, u)$ belongs to $(C((0,T]; L^\infty(\R^2)))^3$.
        \item We have, $\theta(t, x)\to \theta_0(x)$ and $u(t, x) \to u_0(x)$ a.e. $x\in \R^2$ as $t\rightarrow 0$.
        \item  If $\dv u_0 = 0$ then $\dv u = 0$ on $[0, T] \times \R^2$.
    \end{enumerate}
\end{prop}
\begin{proof}

    \textbf{(1):} Suppose that $\dot{\Delta}_j u_0 = (\dot{\Delta}_j K) * \theta_0$. Applying $(\dot{\Delta}_j K) *$ to both sides of the expression for $\theta(t, x)$ in \cref{SQGintegral}, we have
    \begin{align*}
        (\dot{\Delta}_j K) * \theta(t,x)
            &= (\dot{\Delta}_j K) * (G_\alpha(t) \theta_0)(x)
                    - (\dot{\Delta}_j K) * \int_0^t (\nabla G_\alpha(t-s) \cdot (\theta u)(s))(x) \, ds \\
            &= (g_\alpha(t) \stardot (\dot{\Delta}_j K) * \theta_0)(x)
                    - \int_0^t (\dot{\Delta}_j K) * 
                        (\nabla g_\alpha(t-s) \stardot (\theta u)(s))(x) \, ds \\
            &= G_\alpha(t) (\dot{\Delta}_j u_0)(x)
                    - \int_0^t \dot{\Delta}_j ((K * 
                        \nabla g_\alpha(t-s)) \stardot (\theta u)(s))(x) \, ds \\
            &= \dot{\Delta}_j (G_\alpha(t) u_0)(x)
                    - \dot{\Delta}_j \int_0^t (K * 
                        \nabla g_\alpha(t-s)) \stardot (\theta u)(s) \, ds \\
                        &= \dot{\Delta}_j  u (t, x).
    \end{align*}
    Here, we used that $\dot{\Delta}_j K \in \Cal{S}(\R^2)$ by \cref{L:DeltaK} to move $\dot{\Delta}_j K$ inside the integral and to commute $\dot{\Delta}_j K$ with the convolution at time zero. We brought $\dot{\Delta}_j = \varphi_j *$ outside the integral similarly. Finally, we used that
    \begin{align*}
        (\dot{\Delta}_j K) * 
                        \nabla g_\alpha(t-s) \stardot (\theta u)(s)
            = \dot{\Delta}_j ((K * 
                        \nabla g_\alpha(t-s)) \stardot (\theta u)(s)),
    \end{align*}
    as the Fourier transforms of the two expressions coincide.

    \noindent \textbf{(2):} We first aim to bound the size of $u(b,x) - u(a,x)$ uniformly in $x$ for $a,b>0$.
    We can write
    \begin{equation}\label{ln:continuity-in-time-theta-difference}
    \begin{split}
        &\|u(b,x)-u(a,x)\|_{L^\infty_x} \leq \left \| \left(G_\alpha(b)-G_\alpha(a)\right) u_0\right \|_{L^\infty_x}\\
        &\qquad + \left\| \int_0^b K\ast\nabla G_\alpha(b-s)\cdot (\theta u)(s,x)\, ds -  \int_0^a K\ast\nabla G_\alpha(a-s)\cdot (\theta u)(s,x) \, ds\right\|_{L^\infty_x}\\
        &\leq \left \| \left(G_\alpha(b)-G_\alpha(a)\right) u_0(x)\right \|_{L^\infty_x}+\left\| \int_a^b K\ast \nabla G_\alpha(b-s)\cdot (\theta u) (s,x) \, ds\right\|_{L^\infty_x} \\
        &\qquad\qquad
        + \left\| \int_0^a (K\ast \nabla G_\alpha(b-s)- K\ast \nabla G_\alpha(a-s))\cdot (\theta u)(s,x) ds\right\|_{L^\infty_x}.
    \end{split}
    \end{equation}
    For the second term on the right hand side of (\ref{ln:continuity-in-time-theta-difference}), applying Young's inequality and \cref{lem:operator-L1-estimates} for $k=1$ gives the bound,
    \begin{equation} \label{ln:continuity-in-time-theta-b-minus-a}
    \begin{split}
                &\left\| \int_a^b K\ast \nabla G_\alpha(b-s)\cdot (\theta u)(s,x) ds\right\|_{L^\infty_x}
                \leq C\int_a^b (b-s)^{-\frac{1}{2\alpha}} \|(\theta u)(s,x)\|_{L^\infty_{x}} ds\\
                &\qquad \leq  \frac{2\alpha}{2\alpha-1}C\|\theta\|_{L^\infty_{t,x}} \| u\|_{L^\infty_{t,x}} (b-a)^{1-\frac{1}{2\alpha}}.
    \end{split}
    \end{equation}
    Moreover, applying Young's inequality to the third term on the right hand side of (\ref{ln:continuity-in-time-theta-difference}), we have
    \begin{equation} \label{ln:continuity-in-time-theta-uniformly-cont}
    \begin{split}
        \bigg\| \int_0^a &(K\ast \nabla G_\alpha(b-s)- K\ast \nabla G_\alpha(a-s))\cdot (\theta u)(s,x) \, ds\bigg\|_{L^\infty_x} \\
        &\leq  \|\theta\|_{L^\infty_{t,x}} \|u\|_{L^\infty_{t,x}}  \int_0^a \|K\ast \nabla g_\alpha(b-s) - K\ast \nabla g_\alpha(a-s)\|_{L^1_{x}} \, ds.
    \end{split}
    \end{equation}
    For the above integral term, by the Fundamental Theorem of Calculus and \cref{L:K-heat-equation-property}, we can write
    \begin{equation} \label{ln:continuity-in-time-theta-integral-nabla-g}
    \begin{split}
        \int_0^a &\|K\ast \nabla g_{\alpha}(b-s) - K\ast \nabla g_{\alpha}(a-s)\|_{L^1_{x}} \, ds
        = \int_0^a \left\|\int_{a-s}^{b-s} \frac{\partial}{\partial \rho} K\ast \nabla g_{\alpha}(\rho)\, d\rho \right\|_{L^1_{x}}\,  ds \\
        & \leq \int_0^a \int_{a-s}^{b-s} \left \|\frac{\partial}{\partial \rho} K\ast \nabla g_{\alpha}(\rho)\right \|_{L^1_{x}}  d\rho \, ds
        \leq \nu \int_0^a \int_{a-s}^{b-s} \left \|K\ast \Lambda^{2\alpha}\nabla g_{\alpha}(\rho)\right \|_{L^1_{x}}  d\rho \, ds.
    \end{split}
    \end{equation}
    Subsequently, by \cref{lem:lambda-nabla-frac-g-estimate}, one has
     \begin{equation} \label{ln:K-nabla-3-g-est}
         \|K\ast \Lambda^{2\alpha}\nabla g_{\alpha}(\rho)\|_{L^1_{x}} \leq C\rho^{-\left(1+\frac{1}{2\alpha}\right)}.
     \end{equation}
     Substituting (\ref{ln:K-nabla-3-g-est}) into (\ref{ln:continuity-in-time-theta-integral-nabla-g}) and integrating, one finds that,
     \begin{equation} \label{ln:holder-continuity-theta-estimate}
         \int_0^a \|K\ast \nabla g_{\alpha}(b-s) - K\ast \nabla g_{\alpha}(a-s)\|_{L^1_{x}} ds \lesssim \frac{4\alpha^2}{2\alpha-1}\left [(b-a)^{1-\frac{1}{2\alpha}}+b^{1-\frac{1}{2\alpha}} -a^{1-\frac{1}{2\alpha}}\right].
     \end{equation}
    
To estimate the first term on the right hand side of \eqref{ln:continuity-in-time-theta-difference}, we note that by Lemmas \ref{L:fgConvBound} and \ref{L:galKBesovBounds},   $G_\alpha(a) u_0$ belongs to $C^{\gamma}$ for all $\gamma>0$ and is therefore uniformly continuous.  We can then apply an approximation to the identity argument to conclude that, as $b-a\rightarrow 0$,  
\begin{equation}\label{uniformapprox}
\|G_\alpha(b) u_0 - G_\alpha(a) u_0\|_{L^{\infty}_x} = \|G_\alpha(b-a)G_\alpha(a) u_0 - G_\alpha(a) u_0\|_{L^{\infty}_x}\rightarrow 0.
\end{equation}

    Gathering (\ref{ln:continuity-in-time-theta-difference}),(\ref{ln:continuity-in-time-theta-b-minus-a}), (\ref{ln:continuity-in-time-theta-uniformly-cont}), (\ref{ln:holder-continuity-theta-estimate}), and (\ref{uniformapprox}) and taking the limit of \eqref{ln:continuity-in-time-theta-difference} as $a\to b$, the continuity of $u$ is proved. 
    
    For $\theta$, we proceed with a series of estimates analogous to those of $u$ to obtain the following:
    \begin{equation*}
    \begin{split}
         \|\theta(b,x)-\theta(a,x)\|_{L^\infty_x} &\leq  \left\| \left(G_\alpha(b) - G_\alpha(a)\right) \theta_0 \right\|_{L^\infty_x} \\
    &\quad+ C\|\theta\|_{L^\infty_{t,x}} \|u\|_{L^\infty_{t,x}} \left[(b-a)^{1-\frac{1}{2\alpha}}+b^{1-\frac{1}{2\alpha}} -a^{1-\frac{1}{2\alpha}}\right].
    \end{split}
    \end{equation*}
Taking the limit as $a\to b$, the desired continuity of $\theta$ is achieved.

    \noindent\textbf{(3):} The proof is similar to that of (2), but with $b = 0$.  Indeed, the proofs of (2) and (3) differ only in that the first term on the right side of \cref{ln:continuity-in-time-theta-difference}, given by
    \begin{equation*}\label{b0case}
        \norm{u_0 - G_\al(a) u_0}_{L^\iny}
            = \norm{u_0 - g_\al(a) * u_0}_{L^\iny},
    \end{equation*}
    need not vanish as $a \to 0$, since $u_0$ is not necessarily uniformly continuous.  Rather, in this case, we use that $(g_\alpha(t, \cdot)_{t > 0})$ is an approximation to the identity to conclude that $g_\al(a) * u_0(x) \to  u_0(x)$ at every Lebesgue point of $u_0$ (see Theorem 8.15 of \cite{Folland}) and hence a.e.. From this, (3) follows. 

    \noindent \textbf{(4):} We apply \cref{L:divuZero} on $[0,T]$ with $f = K\ast \nabla g_\alpha$  and $\psi = (\theta u)$. The choice of $f$ satisfies the hypotheses of the lemma, as for all $t\in [0,T]$, $\dv f(t) = 0$ in $\mathcal{S}'(\R^2)$  and $f (t) \in L^1(\R^2)$ by \cref{lem:operator-L1-estimates}. Thus,
    \begin{align*}
    	\dv u(t, x)
    		&= \dv (g_\alpha(t) * u_0(x))
    			- \dv \int_0^t ((K * \grad g_\alpha(t - s))\stardot
    					(\theta u)(s))(x) \, ds \\
    		&= g_\alpha(t) * \dv u_0(x).
    \end{align*}
    Therefore, $\dv u(t) = 0$ for all $t \in [0,T]$ if $\dv u_0 = 0$. 
\end{proof}

\subsection{Preservation of the constitutive law}
Having shown the time and spatial regularity of $(\theta,u)$, we are now in a position to prove that (1) If $(\theta_0, u_0)$ satisfy the constitutive law $\cref{SQG}_2$ in the form $u_0 = \PV K * \theta_0$ uniformly then $u(t) = \PV K * \theta(t)$ for $t > 0$; (2) With sufficient regularity of the initial conditions, a solution to the mild formulation satisfies \eqref{e:SQG} pointwise. We prove (1) in this subsection, (2) in the next.

\begin{prop}\label{prop:constitutive-law-holds}
 Suppose that $(\theta,u)$ is a mild solution to \cref{e:SQG} on $[0,T]$ for which $(\theta, u) \in L^\infty([0,T]\times\R^2))^3$. If $u_0 = \PV K * \theta_0$, converging uniformly over annuli, as in \cref{e:InitCondUnif}, then \cref{e:SQG}$_2$ holds.
\end{prop}
\begin{proof}
    
We show that \eqref{e:SQG}$_2$ is satisfied componentwise. Pick $j=1,2$, and convolve the solution $\theta$ with $\PV K^j$,
\begin{equation} \label{ln:K-convolve-mild-solution}
\begin{split}
        &(\PV K^j \ast \theta)(t,x)= \bigg(\PV K^j \ast \bigg(G_\alpha(t) \theta(0,\cdot) -\int_0^t \nabla G_\alpha(t-s)\cdot(\theta u)(s,\cdot)\, ds \bigg) \bigg)(x) \\
        &\qquad = (\PV K^j \ast (G_\alpha(t) \theta(0,\cdot)))(x) - \left(\PV K^j \ast \int_0^t \nabla G_\alpha(t-s)\cdot(\theta u)(s,\cdot)\,ds\right)(x).
\end{split}
\end{equation}

As for the second term, because $K(x - y) \in L^\iny(A_{r, R}(x))$ and $\grad G_{\alpha}(t - s) \cdot (\theta u)(s, \cdot) \in L^1(\R^2)$, we can invoke the Fubini-Tonelli theorem to give
\begin{equation}\label{ln:contentious-use-of-Fubini-theoremAlt}
\begin{split}
\bigg( \PV K \ast\!&\int_{0}^{t} 
    \nabla G_{\alpha}(t-s)\cdot(\theta u)(s, x)\,ds\;\bigg ) (x) \\
   &= \lim_{r, R}\int_{A_{r, R}(x)}\int_{0}^{t}
        K(x-y)\,
        \bigl(\nabla G_{\alpha}(t-s)\cdot(\theta u)(s,y)\bigr)\,ds\,dy \\[4pt]
   &= \lim_{r, R}\int_{0}^{t} \int_{A_{r, R}(x)}
        K(x-y)\,
        \bigl(\nabla G_{\alpha}(t-s)\cdot(\theta u)(s,y)\bigr)\,dy\,ds \\
   &= \lim_{r, R} \int_0^t
   		(\CharFunc_{A_{r, R}(0)} K) *
			\brac{\grad g_\al(t - s) * (\theta u)}(s, x) \, ds \\
	&= \lim_{r, R} \int_0^t 
   		\brac{(\CharFunc_{A_{r, R}(0)} K) * \grad g_\al(t - s)}
			* (\theta u)(s, x) \, ds.
\end{split}
\end{equation}
Because the integrand was the convolution of two $L^1$ functions and an $L^\iny$ function, we were able to use the associativity of the convolutions. By \cref{L:KgradgalUnifBound},
\begin{align*}
	&\abs{\brac{(\CharFunc_{A_{r, R}(0)} K) * \grad g_\al(t - s)}
			* (\theta u)(s, x)} \\
		&\qquad
		\le \norm{(\CharFunc_{A_{r, R}(0)} K)
			* \grad g_\al(t - s)}_{L^1(\R^2)}
			\norm{(\theta u)(s)}_{L^\iny(\R^2)} \\
		&\qquad
		\le C \norm{\theta u}_{L^\iny((0, T) \times \R^2)} 
			(t - s) ^{-1/(2\alpha)},
\end{align*}
which is in $L^1((0, t))$. Hence, we can apply the dominated convergence theorem to give,
\begin{align*}
	\bigg( \PV K \ast\!&\int_{0}^{t} 
    \nabla G_{\alpha}(t-s)\cdot(\theta u)(s, x)\,ds
    			\;\bigg )(x) \\
    	&= \int_0^t \lim_{r, R} 
   		\brac{(\CharFunc_{A_{r, R}(0)} K) * \grad g_\al(t - s)}
			* (\theta u)(s, x) \, ds \\
		&= \int_0^t (K * \grad G_\al(t - s))
			\cdot (\theta u)(s, x).
\end{align*}

Similarly, with the equality \( K \ast \theta(0,x) = u(0,x) \), by applying \cref{L:KUniformConv}, we have
\begin{equation} \label{ln:initial-data-term-u}
    \operatorname{p.v.} \, K \ast (G_\alpha(t) \theta(0,x)) 
    = G_\alpha(t) (\operatorname{p.v.} \, K \ast \theta(0,x)) 
    = G_\alpha(t) u(0,x).
\end{equation}

Using \eqref{ln:initial-data-term-u} in \eqref{ln:K-convolve-mild-solution}, we obtain
\begin{equation*}
\operatorname{p.v.} \, (K \ast \theta)(t,x) 
= G_\alpha(t) u(0,x) 
- \int_0^t K \ast \nabla G_\alpha(t - s) \cdot (\theta u)(s,x) \, ds 
= u(t,x).
\end{equation*}
From this, we conclude that \eqref{e:SQG}$_2$ holds for all $t \in [0, T]$; that is,
$
	u = \PV K * \theta
$ in $[0,T]\times\R^2$.

\end{proof}

\subsection{Classical \cref{e:SQG} for smooth initial data}

Having shown that \cref{e:SQG}$_2$ holds for mild solutions, we now show that sufficiently smooth mild solutions also satisfy \cref{e:SQG}$_1$. In the following proposition $f(x) = \lceil x \rceil$ denotes the ceiling function.
\begin{prop}\label{P:SQGMotivation}
    Suppose that $(\theta,u)$ is a mild solution to \cref{e:SQG} on $[0,T]$ for which $(\theta, u) \in (L^\infty([0,T]; C^{\lceil 2\alpha \rceil }_b(\R^2)))^3$ and $u_0= \PV K\ast \theta_0$.
    Then $\theta$ and $u$ are once differentiable in time and \cref{e:SQG}$_1$ holds in the classical sense for a.e $x\in \R^2$.
\end{prop}
\begin{proof}
To avoid the singularity at $G_\alpha(0)$ and $\nabla G_\alpha(0)$, we first use the Dominated Convergence Theorem to rewrite the mild formulation as 
\begin{equation*}
\begin{split}
     \theta(t,x) &= G_\alpha(t) \theta(0, x) - \int_0^{t} \nabla G_\alpha(t - s) \cdot (u \theta)(s,x) \, ds\\
     &=G_\alpha(t) \theta(0, x) - \lim_{\varepsilon\to 0}\int_0^{t-\varepsilon} \nabla G_\alpha(t - s) \cdot (u \theta)(s,x) \, ds.
     \end{split}
\end{equation*}
Taking a time derivative of our expression for $\theta$ above, we make use of $g_\alpha(t)$ as the fundamental solution of the fractional heat equation given by \eqref{ln:frac-heat-property} and write
\begin{equation}\label{partialtheta}
\begin{split}
    \frac{\partial}{\partial t} \theta(t,x)
    &= -\nu \Lambda^{2\alpha} G_\alpha(t) \theta(0, x)  -\frac{\partial}{\partial t} \lim_{\varepsilon\to 0} \int_0^{t-\varepsilon} \nabla G_\alpha(t - s) \cdot (u \theta)(s,x)\, ds.
    \end{split}
\end{equation}
We wish to show that the derivative and limit can be swapped. Equivalently (see Theorem 7.17 of \cite{Rudin}), fix $t$, let $(\varepsilon_n)_{n=1}^\infty$ be any sequence such that $\varepsilon_n\to 0$ as $n\to \infty$, and define the functions 
\begin{equation*}
    f_n(t,x) := \int_0^{t-\varepsilon_n} \nabla G_\alpha(t - s) \cdot (\theta u)(s,x) \, ds,
\end{equation*}
and $$f(t,x) := \int_0^{t} \nabla G_{\alpha}(t-s)\cdot (\theta u)(s,x) \, ds= \lim_{n\to \infty} f_n(t,x).$$  We will show that for each $x\in\R^2$,
\begin{enumerate}
    \item $f_n(t,x)$ converges uniformly to $f$ in time and
    \item $\partial_t f_n(t,x)$ converges uniformly in time.
\end{enumerate}
 From this, we will conclude that $\partial_t f$ exists and
\begin{equation*}
    \lim_{n\to \infty} \frac{\partial}{\partial t} f_n(t,x) = \frac{\partial}{\partial t} f(t,x).
\end{equation*}
For (1): Uniform convergence in time follows by invoking \cref{lem:operator-L1-estimates} and Young's inequality.  For each $x\in\R^2$, write
\begin{equation*}
\begin{split}
    &\|f(t,x) - f_n(t,x)\|_{L^\infty_t}
        = \left\|\int_{t-\varepsilon_n}^t \nabla G_\alpha(t-s) \cdot (\theta u)(s,x) \, ds \right\|_{L^\infty_t} \\
    &\leq \left\|\int_{t-\varepsilon_n}^t C (t-s)^{-\frac{1}{2\alpha}} \left\|(\theta u) (s,x) \right\|_{L^\infty_x} \, ds \right \|_{L^\infty_t} 
    \leq \frac{2\alpha}{2\alpha-1} C \left\|u \right\|_{L^\infty_{t,x}} \left\|\theta \right\|_{L^\infty_{t,x}} \varepsilon_n^{1-\frac{1}{2\alpha}}.
\end{split}
\end{equation*}
For (2): We note that for each $n$, $\partial_t f_n$ exists using the Leibniz integral rule,
\begin{align}\label{ln:liebniz-rule-applied-to-fn}
    \begin{split}
    \frac{\partial}{\partial t} f_n(t,x) &= \frac{\partial}{\partial t} \int_0^{t-\varepsilon_n}
    \nabla G_\alpha(t - s) \cdot (\theta u) (s,x) \, ds \\
    &= (\nabla G_\alpha(\varepsilon_n)\cdot(\theta u))(t-\varepsilon_n,x) + \int_0^{t-\varepsilon_n} \frac{\partial}{\partial t} \nabla G_\alpha(t-s) \cdot (\theta u)(s,x) \, ds.
    \end{split}
\end{align} 

To justify the use of Leibniz rule above, we note that $\|\frac{\partial}{\partial t} \nabla g_\alpha(t-s,x)\|_{L^1_x}$ is continuous in $s$ on $[0,t-\varepsilon_n]$, and $\|\theta u(s,x)\|_{L^{\infty}_x}$ is bounded on $[0,t-\varepsilon_n]$ by assumption.  Thus, an application of Young's inequality implies that the expression 
\[
\frac{\partial}{\partial t} \left[\nabla G_\alpha(t - s) \cdot (u \theta)(s,x)\right]
\]
exists and is absolutely integrable in $s$ on $[0,t-\varepsilon_n]$.

Now we show that $(\partial_t f_n)_{n=1}^\infty$ is Cauchy in $L^{\infty}_t$.  First note that by integration by parts and an application of \eqref{ln:frac-heat-property},
\begin{equation*}
     \frac{\partial}{\partial t} \left[\nabla G_\alpha(t - s) \cdot (u \theta)(s,x)\right]= \frac{\partial}{\partial t}  G_\alpha(t - s) \dv (u \theta)(s,x) = \Lambda  G_\alpha(t - s) \Lambda^{2\alpha-1}\dv (u \theta)(s,x).
\end{equation*}
Substituting this equality into (\ref{ln:liebniz-rule-applied-to-fn}) gives
\begin{align}\label{partialtfnformula}
    \begin{split}
    \frac{\partial}{\partial t} f_n(t,x) &= \frac{\partial}{\partial t} \int_0^{t-\varepsilon_n}
    \nabla G_\alpha(t - s) \cdot (\theta u) (s,x) \, ds \\
    &= (\nabla G_\alpha(\varepsilon_n)\cdot(\theta u))(t-\varepsilon_n,x) + \int_0^{t-\varepsilon_n} \Lambda  G_\alpha(t - s) \Lambda^{2\alpha-1}\dv (u \theta)(s,x) \, ds.
    \end{split}
\end{align} 
 Thus, we have for  $n>m>0$ and for each $x\in\R^2$,
\begin{equation}\label{ln:uniform-convergence-of-time-derivative-theta}
\begin{split}
    &\|\partial_t f_n(x) - \partial_t f_m(x)\|_{L^\infty_{t}} 
    \leq \left\| 
    (\nabla G_\alpha(\varepsilon_n) - \nabla G_\alpha(\varepsilon_m)) \cdot (\theta u)(t-\epsilon_n,x)
    \right\|_{L^\infty_{t}} \\
    &\qquad\qquad + \left\| 
    \nabla G_\alpha(\varepsilon_m)((\theta u)(t-\epsilon_n,x)-(\theta u)(t-\epsilon_m,x) )
    \right\|_{L^\infty_{t}}\\
    &\qquad\qquad + \nu \left\| 
    \int_{t - \varepsilon_m}^{t - \varepsilon_n} \Lambda G_\alpha(t - s) 
    \Lambda^{2\alpha - 1} \operatorname{div}((u \theta)(s,x))\, ds 
    \right\|_{L^\infty_{t}}.
\end{split}
\end{equation}
For the first term of (\ref{ln:uniform-convergence-of-time-derivative-theta}), we use the Fundamental Theorem of Calculus and the fractional heat kernel property \eqref{ln:frac-heat-property} to write
\begin{align*}
\begin{split}
    &\left \|  (\nabla G_\alpha(\varepsilon_n) - \nabla G_\alpha(\varepsilon_m)) \cdot (\theta u)(t-\varepsilon_n,x)\right \|_{L^\infty_{t}}
    = \left \|  \int_{\varepsilon_n}^{\varepsilon_m}\frac{\partial}{\partial \rho} \nabla G_\alpha(\rho) \cdot (\theta u)(t-\varepsilon_n,x)\, d\rho \right \|_{L^\infty_{t}} \\
    &\qquad\qquad
    \leq \nu \left \|  \int_{\varepsilon_n}^{\varepsilon_m}\nabla G_\alpha(\rho) \cdot \Lambda^{2\alpha}(u \theta) (t-\varepsilon_n,x) \, d\rho \right \|_{L^\infty_{t}}\\
    &\qquad\qquad \leq \nu \left \|  \int_{\varepsilon_n}^{\varepsilon_m} \left\|\nabla G_\alpha(\rho) \cdot \Lambda^{2\alpha} (u \theta)(t-\varepsilon_n,x)  \right\|_{L^\infty_x}\, d\rho \right \|_{L^\infty_{t}}.
\end{split}
\end{align*}
Applying Young's inequality and \cref{lem:operator-L1-estimates}, and integrating with respect to $\rho$, yields
\begin{equation}\label{ln:cauchy-seq-theta-1}
\begin{split}
    &\left \|  \int_{\varepsilon_m}^{\varepsilon_n} \left\|\nabla G_\alpha(\rho) \cdot \Lambda^{2\alpha} (u \theta)(t-\varepsilon_n,x)  \right\|_{L^\infty_x} d\rho \right \|_{L^\infty_{t}} \\
    &\qquad \leq \frac{2\alpha}{2\alpha-1}C\left\|\Lambda^{2\alpha} (u \theta) \right\|_{L^\infty_{t,x}}\left(\varepsilon_m^{1-\frac{1}{2\alpha}} - \varepsilon_n^{1-\frac{1}{2\alpha}}\right).
\end{split}
\end{equation}
We expand the second term of \eqref{ln:uniform-convergence-of-time-derivative-theta} by using integration by parts, Young's convolution inequality, and the divergence-free condition on $u=\nabla^\perp \Lambda \theta$, which follows from \cref{prop:constitutive-law-holds}.  We write    
\begin{equation}\label{UCintime}
\begin{split}
    \sup_{t\in [0,T]} &\left| \nabla G_\alpha(\epsilon_n) \cdot ((\theta u)(t-\epsilon_n,x)-(\theta u)(t-\epsilon_m,x) ) \right| \\
    &= \sup_{t\in [0,T]} \left| G_\alpha(\epsilon_n)\dv ((\theta u)(t-\epsilon_n,x)-(\theta u)(t-\epsilon_m,x) ) \right| \\
    &\leq \sup_{t\in [0,T]} \| g_\alpha(\epsilon_n,x) \|_{L^1_x} \|u\cdot\nabla \theta(t-\epsilon_n,x)-u\cdot\nabla \theta(t-\epsilon_m,x) \|_{L^\infty_x} \\
    &= \sup_{t\in [0,T]}  \|(u\cdot\nabla \theta)(t-\epsilon_n,x)-(u\cdot\nabla \theta)(t-\epsilon_m,x) \|_{L^\infty_x}.
\end{split}
\end{equation}
The above term vanishes as a consequence of the continuity in time of $u$ from \cref{P:time-regularity}(2) and the continuity in time of $\nabla \theta$ from \cref{lem:continuity-in-time-of-derivative-of-theta} below.
Turning to the third term in (\ref{ln:uniform-convergence-of-time-derivative-theta}), we apply \cref{lem:operator-L1-estimates} and invoke the $C^{\lceil 2\alpha \rceil}_b$ regularity of $(\theta, u)$, to write
\begin{equation} \label{ln:cauchy-seq-theta-2}
\begin{split}
    &\left\| \int_{t-\varepsilon_m}^{t-\varepsilon_n} \Lambda  G_\alpha(t - s)  \Lambda^{2\alpha-1}\dv ((u \theta)(s,x)) \, ds \right\|_{L^\infty_t} \\
    &\qquad
    \leq \int_{t-\varepsilon_m}^{t-\varepsilon_n} C (t-s)^{-\frac{1}{2\alpha}} \left\|\Lambda^{2\alpha-1} \operatorname{div} ((u \theta)(s,x)) \right\|_{L^\infty_x} \, ds \\
    &\qquad
    \leq \frac{2\alpha}{2\alpha-1}C \left\|\Lambda^{2\alpha-1} \operatorname{div} (u \theta) \right\|_{L^\infty_{t,x}} \left(\varepsilon_m^{1-\frac{1}{2\alpha}} - \varepsilon_n^{1-\frac{1}{2\alpha}}\right).
\end{split}
\end{equation}
Combined, (\ref{ln:cauchy-seq-theta-1}), (\ref{UCintime}), and (\ref{ln:cauchy-seq-theta-2}) imply that $\partial_t f_n$ is Cauchy in $L^\infty_t$. In addition, $\partial_t f$ exists and $\partial_t f_n \to \partial_t f$ as $n\to \infty$. We utilize this convergence and (\ref{partialtfnformula}) to rewrite (\ref{partialtheta}) as 

\begin{equation}
\label{ln:set-up-pointwise-formulation}
\begin{split}
\frac{\partial}{\partial t}\theta(t,x)
&=
\frac{\partial}{\partial t}\Bigl[G_\alpha(t)\ast \theta(0,\cdot)\Bigr](x)
\;-\;
\lim_{n\to\infty}
\frac{\partial}{\partial t}\,f_n(t,x)
\\
&=
\frac{\partial}{\partial t}\Bigl[G_\alpha(t)\ast \theta(0,\cdot)\Bigr](x)
\;-\;
\lim_{n\to \infty}
\Biggl[
\,\nabla G_\alpha\bigl(\varepsilon_n\bigr)\cdot\bigl(\theta u\bigr)(t-\varepsilon_n,x)
\\
&\qquad\qquad
-\;\nu
\int_{\,0}^{\,t-\varepsilon_n}
 \Lambda  G_\alpha(t - s) \Lambda^{2\alpha-1}\dv (u \theta)(s,x)
\,ds
\Biggr].
\end{split}
\end{equation}
For the first term on the right-hand side of (\ref{ln:set-up-pointwise-formulation}), we apply (\ref{ln:frac-heat-property}), and for the second term, we use integration by parts. Finally, for the third term, we apply the Dominated Convergence Theorem, noting that by a calculation similar to that in (\ref{ln:cauchy-seq-theta-2}), $\Lambda  G_\alpha(t - \cdot) \Lambda^{2\alpha-1}\dv (u \theta)(\cdot,x)$ belongs to $L^1([0,t])$ for each $x$. We conclude that

\begin{equation} \label{ln:pointwise-formulation-of-sqg-1}
\begin{split}
    \frac{\partial}{\partial t} \theta &= -\nu  \Lambda^{2\alpha} G_\alpha(t) \theta(0,\cdot) - \lim_{n \to \infty}  [G_\alpha (\varepsilon_n)\dv(\theta u) (t-\varepsilon_n, \cdot)]\\
    &\qquad +  \nu\int_0^{t} \Lambda^{2\alpha}\nabla G_\alpha (t - s) \cdot (\theta u)(s,\cdot) \, ds.
\end{split}
\end{equation}
For the second term of (\ref{ln:pointwise-formulation-of-sqg-1}), write it as
\begin{equation}\label{ln:divergence-term-classical-solution}
\begin{split}
    G_\alpha(\epsilon_n) &\dv(\theta u)(t-\epsilon_n,\cdot) \\
    &=  G_\alpha(\epsilon_n) (\dv(\theta u)(t-\epsilon_n,\cdot) - \dv(\theta u)(t,\cdot)) + G_\alpha(\epsilon_n) \dv(\theta u)(t,\cdot)    
\end{split}
\end{equation}
Applying the divergence free condition on $u$, the first term of \eqref{ln:divergence-term-classical-solution} vanishes as a consequence of the continuity in time of $\nabla \theta$ from \cref{lem:continuity-in-time-of-derivative-of-theta} and the continuity in time of $u$ from \cref{P:time-regularity}(2).

For the second term of \eqref{ln:divergence-term-classical-solution}, we apply Theorem 8.15 of \cite{Folland} to assert the a.e. convergence of $G_{\alpha}(\varepsilon_n)\dv(\theta u)(t,x)$ to $\dv(u\theta)(t,x)$. Hence,
\begin{equation}\label{ln:pointwise-sqg-term-2}
\begin{split}
         \lim_{n\to \infty} G_\alpha(\varepsilon_n)\dv(\theta u)    &= \dv(\theta u)
            =u\cdot\nabla \theta,
\end{split}
\end{equation}
where we used the divergence free condition on $u$ to get the second equality above.
Notice that for the third term of (\ref{ln:pointwise-formulation-of-sqg-1}), we can interchange the order of time integration and $\Lambda^{2\alpha}$ as a consequence of the Leibniz integral rule, as both
\begin{equation*}
    \nabla G_\alpha(t-s) \cdot (\theta u) \text{  and  } \Lambda^{2\alpha} \nabla G_\alpha(t-s) \cdot (u\theta)
\end{equation*}
are integrable in time by \cref{lem:operator-L1-estimates} and \cref{lem:lambda-nabla-frac-g-estimate}, respectively. Indeed, for 
$\Lambda^{2\alpha} \nabla G_\alpha(t-s) \cdot (u\theta)$, we apply $\Lambda^{2\alpha}$ to the product $\theta u$ and use the boundedness of the derivatives of $\theta u$ to reach the conclusion. Thus, we have
\begin{equation} \label{ln:pointwise-sqg-term-3}
\begin{split}
     \nu\int_0^{t} \Lambda^{2\alpha} \nabla G_\alpha(t - s) \cdot (u \theta)(s,x) \, ds &= \nu\Lambda^{2\alpha} \int_0^{t}\nabla G_\alpha(t - s) \cdot (u \theta)(s,x) \, ds. 
\end{split}
\end{equation}
Collecting terms (\ref{ln:pointwise-sqg-term-2}) and (\ref{ln:pointwise-sqg-term-3}), we can rewrite  (\ref{ln:pointwise-formulation-of-sqg-1}) using the definition of $\theta$ in (\ref{SQGintegral})$_1$ and deduce
\begin{equation*}
\begin{split}
     \frac{\partial}{\partial t} \theta(t,x) &= -\nu \Lambda^{2\alpha} G_\alpha(t) \theta(0,x) -  (u\cdot \nabla \theta)(t,x)
     + \nu\Lambda^{2\alpha}\int_0^{t}\nabla G_\alpha(t - s) \cdot (u \theta)(s,x) \, ds \\
     &= -\nu \Lambda^{2\alpha} \theta (t,x) - (u\cdot \grad \theta)(t,x).
\end{split}
\end{equation*}
Hence, we see that \eqref{e:SQG}$_1$ is satisfied. 
From here, we can estimate the size of $\frac{\partial}{\partial t} \theta$ by
\begin{equation}
    \left\| \frac{\partial}{\partial t} \theta \right\|_{L^\infty_x} \leq \nu \|\Lambda^{2\alpha} \theta\|_{L^\infty_x} + \|u\|_{L^\infty_x} \|\nabla \theta \|_{L^\infty_x} < \infty.
\end{equation}
Thereby, we conclude that  $(\theta,u)$ is a classical solution.
\end{proof}

\subsection{Solutions to \cref{ssqg}}
Reusing the first parts of the argument of \cref{P:SQGMotivation}, by modifying the assumption on the initial data, we can show mild solutions of \eqref{ssqg} also satisfy the equation pointwise.
\begin{prop}\label{P:lp-SQGMotivation}
    Suppose that $(\theta,u) \in L^\infty([0,T]; C^2_b(\R^2))\times (L^\infty([0,T]; C^2_b(\R^2))^2$ and satisfy \eqref{D:mildsolution} with $\dot{\Delta}_j u_0 = (\dot{\Delta}_j K) * \theta_0$ for all $j\in \Z$ and $\dv u_0 = 0$.
    Then $(\theta, u)$ are once differentiable in time and satisfy \cref{ssqg}.
    \begin{proof}
        By \cref{P:time-regularity} (4) one has $\dv u = 0$, Therefore, the argument in \cref{P:SQGMotivation} proceeds identically up to line \eqref{ln:pointwise-formulation-of-sqg-1}. More specifically, we have that $\theta$ is differentiable in time. For any sequence $(\varepsilon_n)_{n=1}^\infty$ with $\varepsilon_n \to 0$, the time derivative of $\theta$ is given by:
    \begin{equation}\label{ln:lp-pointwise-formulation-of-sqg-1}
    \begin{split}
    \frac{\partial}{\partial t} \theta &= -\nu  \Lambda^{2\alpha} G_\alpha(t) \theta(0,x) - \lim_{n \to \infty}  G_\alpha (\varepsilon_n)\dv(\theta u) \\
    &\qquad +  \nu\int_0^{t} \Lambda^{2\alpha}\nabla G_\alpha (t - s) \cdot (u \theta) \, ds.
    \end{split}
    \end{equation}
    For the first term of \eqref{ln:lp-pointwise-formulation-of-sqg-1}, we can successively apply part (4) and Theorem 8.15 of \cite{Folland},
    \begin{equation} \label{ln:lp-pointwise-formulation-of-sqg-2}
    \begin{split}
        \lim_{n\to\infty} G_\alpha(\epsilon_n) \dv (\theta u) &= \lim_{n\to\infty} G_\alpha(\epsilon_n) (u\cdot \nabla \theta)
        = u\cdot \nabla \theta.
    \end{split}
    \end{equation}
    Substituting \eqref{ln:pointwise-sqg-term-3} and \eqref{ln:lp-pointwise-formulation-of-sqg-2} into \eqref{ln:lp-pointwise-formulation-of-sqg-1} yields the desired result for $\theta$,
\begin{equation*}
\begin{split}
     \frac{\partial}{\partial t} \theta
     &= -\nu \Lambda^{2\alpha} \theta - u\cdot \grad \theta.
\end{split}
\end{equation*}
We invoke the same estimate as in the previous proposition for $\frac{\partial}{\partial t} \theta$,
\begin{equation}
    \left\| \frac{\partial}{\partial t} \theta \right\|_{L^\infty_x} \leq \nu \|\Lambda^{2\alpha} \theta\|_{L^\infty_x} + \|u\|_{L^\infty_x} \|\nabla \theta \|_{L^\infty_x} < \infty,
\end{equation}
by the hypotheses. The constitutive law \eqref{ssqg}$_2$ is recovered directly from \cref{P:time-regularity} (1).
\end{proof}
\end{prop}

\section{Existence of a Finite Time Mild Solution}\label{existence}
In this section, we prove \cref{thm:existence-of-solutions}. 
We let $\tau > 0$, choosing a precise value of $\tau$ later. For $p\in L^\infty([0,\tau]\times\R^2)$ and $\omega \in  (L^\infty([0,\tau]\times\R^2))^2$, we define the two maps,
\begin{align}\label{integralform}
    \begin{split}
        T_\omega p(t, \cdot)
            := G_\alpha(t)\theta_0
                - \int_0^t \nabla G_\alpha(t-s)\cdot (p\hspace{0.1em}\omega)(s) \, ds, \\
        U_p \omega(t, \cdot)
            := G_\alpha(t)u_0
                -\int_0^t (K\ast \nabla G_\alpha(t-s))\cdot(p\hspace{0.1em}\omega) \, ds.
    \end{split}
\end{align}
Our proof of existence will follow an iterative scheme, setting
\begin{equation}\label{scheme}
    \theta^1(t,x) := \theta_0(x), \,
    u^1(t,x) := u_0(x)
    \text{ for all } t \ge 0,
\end{equation}
while for $n\geq 1$,
\begin{align}\label{scheme2}
    \begin{split}
        \theta^{n+1}(t,x)
            &:= T_{u^n} \theta^{n + 1}(t,x)
            = G_\alpha(t)\theta_0-\int_0^t \nabla G_\alpha(t-s) \cdot (u^n\theta^{n+1})(s) \, ds,\\
        u^{n+1}(t,x)
            &:= U_{\theta^{n + 1}} u^n(t,x)
            = G_\alpha(t)u_0  - \int_0^t (K \ast\nabla G_\alpha(t-s)) \cdot (u^n\theta^{n+1})(s) \, ds.
    \end{split}
\end{align}

We iterate over $n = 0, 1, \dots$ as follows:
\begin{enumerate}
    \item
        Setting $\omega = u^n$, we use the Banach contraction mapping theorem to obtain $\theta^{n + 1}$ as the fixed point of the operator $T_\omega$. This fixed point exists on a time interval $\tau > 0$ that depends on the initial data, but is independent of $n$.
        \\[-5pt]
    \item
        For a fixed $p = \theta^{n + 1}$, we set $u^{n + 1} = U_p u^n$.
\end{enumerate}
We then prove the convergence of the sequences $(\theta^n)$ and $(u^n)$ to $\theta$ and $u$, respectively, and we show that $(\theta, u)$ is a mild solution as in \cref{D:mildsolution} up to time $\tau$.

We begin with estimates on the integrals appearing in $T$ and $U$.

From the definition of $G_\alpha(t)$ and \cref{lem:operator-L1-estimates} for $k=1$,
\begin{equation*}
\begin{split} 
&\left\| \int_0^t \nabla G_\alpha(t-s) \cdot(p\hspace{0.1em}\omega) ds \right\|_{L^{\infty}_{x}} \leq  \int_0^t  \left \|\nabla g_\alpha(t-s, y) \ast \cdot (p\hspace{0.1em}\omega) \right\|_{L^{\infty}_{x}} \, ds \\
&\leq C\int_0^t (t-s)^{-\frac{1}{2\alpha}}\|p\hspace{0.1em}\omega\|_{L^\infty_x} ds \leq C\int_0^t (t-s)^{-\frac{1}{2\alpha}}  \|p\|_{L^{\infty}_{x}} \|\omega\|_{L^{\infty}_{x}} \, ds. 
\end{split}
\end{equation*}
Integrating the above in time and applying the $L^\infty_t \equiv L^\infty([0,\tau])$ norm to both sides of the resulting inequality gives
\begin{equation} \label{ln:size-of-integral-nabla-G}
\begin{split} 
&\left\| \int_0^t \nabla G_\alpha(t-s) \cdot(p\hspace{0.1em}\omega) ds \right\|_{L^{\infty}_{t,x}}
\leq \frac{2\alpha}{2\alpha-1}C\tau^{1-\frac{1}{2\alpha}}\|p\|_{L^{\infty}_{t,x}} \|\omega\|_{L^{\infty}_{t,x}}. \end{split}
\end{equation}
Again, using \cref{lem:operator-L1-estimates} for $k=1$,
\begin{equation*}
\begin{split}
&\left \|\int_0^t K\ast \nabla G_\alpha(t-s) \cdot (p\hspace{0.1em}\omega) (s) ds \right\|_{L_{x}^\infty} \leq \int_0^t \left \|(K\ast \nabla G_\alpha(t-s)) \cdot (p\hspace{0.1em}\omega) (s) \right\|_{L_{x}^\infty} ds \\
&\qquad\qquad 
\leq \int_0^t C (t-s)^{-\frac{1}{2\alpha}} \left\|p\right\|_{L_x^\infty}  \left \|\omega \right\|_{L_x^\infty} ds.
\end{split}
\end{equation*}
As before, integrating in time and applying the $L^\infty_t$ norm to both sides gives
\begin{equation} \label{ln:size-of-integral-K-nabla-G}
\begin{split} 
&\left\| \int_0^t K\ast \nabla G_\alpha(t-s) \cdot(p\hspace{0.1em}\omega) ds \right\|_{L^{\infty}_{t,x}} 
 \leq \frac{2\alpha}{2\alpha-1}C\tau^{1-\frac{1}{2\alpha}}\|p\|_{L^{\infty}_{t,x}} \|\omega\|_{L^{\infty}_{t,x}}. \end{split}
\end{equation}

Now, we choose $\tau>0$ to satisfy the following size condition,
\begin{equation}\label{Tcondition}
\frac{2\alpha}{2\alpha-1}C\tau^{1-\frac{1}{2\alpha}} ( \| {\theta}_0 \|_{L^{\infty}_x}+ \| u_0 \|_{L^{\infty}_x}) \leq \tfrac{1}{8}.
\end{equation}

\textbf{Convergence of Approximating Sequence.} We begin the iterative process by realizing $\theta^2$ as the fixed point of $T_{u_1}$ using the Banach contraction mapping theorem (Step 1). From there we can quickly obtain $u^2$ and then for illustrative purposes, we proceed by generating $\theta^3$ with a similar argument, as the fixed point of $T_{u_2}$ (Step 2). Finally, we consider the general case of $\theta^{n+1}$ as the fixed point of $T_{u^n}$ (Step 3). 
\medskip

\noindent\textbf{Step 1 (Fixed point of $T_{u^1}$)}: Set $R=2\|\theta_0\|_{L_x^{\infty}}$ and define $B_R$ as the ball of radius $R$ centered at the origin in $L^\infty([0,\tau]\times \R^2)$. Let $p$ and $\bar{p}$ be two elements of $B_R$.  Then, using estimate (\ref{ln:size-of-integral-nabla-G}) and the equality $\|u^1\|_{L^\infty_{t,x}} = \|u_0\|_{L^\infty_{x}}$ from (\ref{scheme}), we have
\begin{equation*}
\begin{split} 
&\|T_{u^1} p - T_{u^1} \bar{p} \|_{L^{\infty}_{t, x}} = \left \|\int_{0}^t \nabla G_\alpha(t-s) \cdot(u^1(p-\bar{p})) \, ds \right\|_{L^\infty_{t,x}} \\
&\leq \frac{2\alpha}{2\alpha-1}C\tau^{1-\frac{1}{2\alpha}} \|u^1\|_{L^{\infty}_{t,x}} \|p - \bar{p} \|_{L^{\infty}_{t,x}} =  \frac{2\alpha}{2\alpha-1}C \tau^{1-\frac{1}{2\alpha}} \|u_0\|_{L^{\infty}_{x}} \|p - \bar{p} \|_{L^{\infty}_{t,x}}.
\end{split}
\end{equation*}
Invoking the size condition on $\tau$ in (\ref{Tcondition}), we conclude that 
\begin{equation} \label{ln:estimate-on-contraction-map-with-u-1}
    \| T_{u^1}p - T_{u^1}\bar{p} \|_{L^\infty_{t,x}} \leq \frac{1}{8} \| p - \bar{p} \|_{L^\infty_{t,x}}.
\end{equation}

To see that $T_{u^1}$ maps $B_R$ into $B_R$, select $p\in B_R$. Applying the estimate (\ref{ln:estimate-on-contraction-map-with-u-1}) and Young's inequality, we write
\begin{equation}
\begin{split} \label{ln:size-of-theta-1}
&\| T_{u^1} p \|_{L^{\infty}_{t,x}} \leq \| T_{u^1} p - T_{u^1} 0 \|_{L^{\infty}_{t,x}} + \| T_{u^1} 0 \|_{L^{\infty}_{t,x}} \leq \frac{1}{8} \|p - 0 \|_{L^{\infty}_{t,x}} + \| G_\alpha(t) \theta_0 \|_{L^{\infty}_{t,x}} \\
&\leq \frac{1}{8} \|p \|_{L^{\infty}_{t,x}} + \sup_{t\in(0,\tau]}\| g_\alpha(t) \|_{L^1_x} \|\theta_0 \|_{L^{\infty}_{x}} \leq \frac{1}{4} \|\theta_0\|_{L^\infty_{x}} + \|\theta_0\|_{L^\infty_x} = \frac{5}{4}\|\theta_0\|_{L^\infty_x} \leq R.
\end{split}
\end{equation}

It follows that $T_{u^1}$ is a strict contraction from $B_R$ into $B_R$.  Thus, there exists a fixed point of $T_{u^1}$, call it $\theta^{1+1} = \theta^2$. 

To proceed, we establish the existence of $u^2 := U_{\theta^2} u^1$ by proving the boundedness of $U_{\theta^2}$. By $(\ref{scheme2})_2$ and (\ref{ln:size-of-integral-K-nabla-G}),
\begin{equation*}
\begin{split}
\| u^2(t) \|_{L^{\infty}_{t,x}} &=  \left\| G_\alpha(t)u_0(x)  - \int_0^t (K \ast\nabla G_\alpha(t-s)) \cdot (\theta^{2} u^1)(s) \, ds \right\|_{L^\infty_{t,x}} \\ 
&\leq  \left\| G_\alpha(t)u_0(x)\right\|_{L^\infty_{t,x}}  + \left \|\int_0^t (K \ast\nabla G_\alpha(t-s)) \cdot (\theta^{2} u^1)(s) \, ds \right\|_{L^\infty_{t,x}} \\
& \leq \| u_0 \|_{L^{\infty}_x} + \frac{2\alpha}{2\alpha-1}C\tau^{1-\frac{1}{2\alpha}}\| \theta^2 \|_{L^\infty_{t,x}} \| u^1 \|_{L^{\infty}_{t,x}}.
\end{split}
\end{equation*}

We then substitute in the estimate $\|\theta^2\|_{L_{t,x}^\infty} \leq \frac{5}{4}\|\theta_0\|_{L_x^\infty}$ from (\ref{ln:size-of-theta-1}) and apply our condition on $\tau$ in (\ref{Tcondition}).  This gives 
\begin{equation} \label{velA22}
\begin{split}
\| u^2(t) \|_{L^{\infty}_{t,x}} &\leq \| u_0 \|_{L^{\infty}_{x}} + \frac{5}{4} \frac{2\alpha}{2\alpha-1}C\tau^{1-\frac{1}{2\alpha}} \| \theta_0 \|_{L^\infty_{t,x}} \| u_0 \|_{L^{\infty}_{x}} \\
&\leq \| u_0 \|_{L^{\infty}_{x}} + \frac{5}{32} \| u_0 \|_{L^{\infty}_{x}} \leq 2 \|u_0\|_{L^\infty_x} = R.
\end{split}
\end{equation}

 With the boundedness of $u^2\in (L^\infty([0,\tau]\times\R^2))^2$, we continue by showing the existence of the fixed point of $T_{u^2}$.

\medskip

\noindent\textbf{Step 2 (Fixed point of $T_{u^2}$)}: For $p \in L^\infty([0,\tau]\times \R^2)$, consider
\begin{equation}\label{integralform2}
T_{u^2}p(t,x) = G_\alpha(t)\theta_0 - \int_0^t \nabla G_\alpha(t-s)\cdot (pu^2)(s) \, ds.
\end{equation}
Using an argument similar to the work in Step 1, we find that for $p$ and $\bar{p}$ in $B_R$,
\begin{equation}\label{A2}
\begin{split}
\|T_{u^2} p - T_{u^2} \bar{p} \|_{L^{\infty}_{t,x}}
&\leq \frac{2\alpha}{2\alpha-1}C\tau^{1-\frac{1}{2\alpha}} \|p-\bar{p}\|_{L^\infty_{t,x}}\|u^2\|_{L^\infty_{t,x}}.
\end{split}
\end{equation}

Substituting (\ref{velA22}) into (\ref{A2}) yields
\begin{equation*}
\begin{split}
\| T_{u^2} p - T_{u^2} \bar{p} \|_{L^{\infty}_{t,x}} &\leq 2\frac{2\alpha}{2\alpha-1} C\tau^{1-\frac{1}{2\alpha}} \|p-\bar{p}\|_{L^\infty_{t,x}} \|u_0\|_{L^\infty_{x}}.
\end{split}
\end{equation*}
Again, applying the size condition on $\tau$ in (\ref{Tcondition}), we conclude that
\begin{equation} \label{ln:estimate-on-contraction-with-u-2}
\begin{split}
 \| T_{u^2} p - T_{u^2} \bar{p} \|_{L^{\infty}_{t,x}} &\leq \frac{1}{4} \| p-\bar{p}\|_{L^\infty_{t,x}}.
 \end{split}
 \end{equation}
 As before, we now show $T_{u^2}$ maps $B_R$ into $B_R$. Let $f\in B_R$ and utilize (\ref{ln:estimate-on-contraction-with-u-2}) and Young's inequality to write
 \begin{equation*}
 \begin{split}
     &\|T_{u^2} f\|_{L_{t,x}^\infty} \leq \|T_{u^2} f - T_{u^2} 0\|_{L_{t,x}^\infty} + \|T_{u^2} 0 \|_{L_{t,x}^\infty} \leq \frac{1}{4} \|f-0\|_{L_{t,x}^\infty} + \|G_\alpha(t)\theta_0\|_{L_{t,x}^\infty} \\
     &\qquad \leq \frac{1}{4} \|f\|_{L_{t,x}^\infty} + \sup_{t\in [0,\tau]} \|g_\alpha(t)\|_{L^1_x} \|\theta_0\|_{L_x^\infty} \leq R.
     \end{split}
 \end{equation*}
Thus $T_{u^2}$ has a fixed point, call it $\theta^{2+1} = \theta^3$.

\medskip

\noindent\textbf{Step 3 (General case)}: For the inductive step, fix $n\in\N$ and suppose $\|\theta^m\|_{L^\infty_{t,x}} \leq 2\|\theta_0\|_{L^\infty_x}$ for every $m\leq n$. We compute $\theta^{n+1}$ by showing the existence of a fixed point of the map defined on $B_R$ given by
\begin{equation}
    T_{u^n}p = G_\alpha(t)\theta_0 - \int_0^t \nabla G_\alpha(t-s)\cdot(p u^n )(s) \, ds.
\end{equation}

The dependence of $T_{u^n}$ on $u^n$ prompts us to first estimate the size of $u^m$. To that effect, observe that for $m\leq n$,
\begin{equation*}
\begin{split}
\| u^m \|_{L^{\infty}_{t,x}} &= 
\left\|G_\alpha(t)u_0 - \int_0^t (K\ast \nabla G_\alpha(t-s)) \cdot (\theta^m u^{m-1}) (s) \, ds \right\|_{L_{t,x}^\infty} \\
&\leq 
\|G_\alpha(t)u_0\|_{L_{t,x}^\infty} + \left \|\int_0^t (K\ast \nabla G_\alpha(t-s)) \cdot (\theta^m u^{m-1}) (s)\,  ds \right\|_{L_{t,x}^\infty} \\
&\leq 
\|u_0\|_{L_{x}^\infty} + \frac{2\alpha}{2\alpha-1}C\tau^{1-\frac{1}{2\alpha}} \|\theta^m\|_{L^\infty_{t,x}} \|u^{m-1}\|_{L^\infty_{t,x}}, \\
\end{split}
\end{equation*}
where we have used our estimate on the integral of $K\ast \nabla G$ from (\ref{ln:size-of-integral-K-nabla-G}) in the second inequality. We can then apply the induction hypothesis $\| \theta^m \|_{L^{\infty}_{t,x}}\leq 2\| \theta_0 \|_{L^{\infty}_{x}}$ coupled with the size condition on $\tau$ in (\ref{Tcondition}) to conclude that 
\begin{equation}\label{velounifbound}
\begin{split}
\| u^m \|_{L^{\infty}_{t,x}}&\leq \| u_0 \|_{L^{\infty}_x} + \frac{2\alpha}{2\alpha-1}C\tau^{1-\frac{1}{2\alpha}} \| \theta^m \|_{L^{\infty}_{t,x}}\| u^{m-1} \|_{L^{\infty}_{t,x}}\\
&\leq \| u_0 \|_{L^{\infty}_x} + 2\frac{2\alpha}{2\alpha-1}C\tau^{1-\frac{1}{2\alpha}} \|\theta_0\|_{L^\infty_{x}}\| u^{m-1} \|_{L^{\infty}_{t,x}}  \\
&\leq \| u_0 \|_{L^{\infty}_x} + \frac{1}{4}\| u^{m-1} \|_{L^{\infty}_{t,x}}.
\end{split}
\end{equation}
Thus, for each $m\leq n$,
\begin{equation*}
\begin{split}
&\| u^m \|_{L^{\infty}_{t,x}} \leq \| u_0 \|_{L^{\infty}_x} + \frac{1}{4}\| u^{m-1} \|_{L^{\infty}_{t,x}}\\
&\qquad \leq \| u_0 \|_{L^{\infty}_x} + \frac{1}{4}\left(\| u_0 \|_{L^{\infty}_x}+\frac{1}{4}\| u^{m-2} \|_{L^{\infty}_{t,x}}\right)\\
&\qquad \leq \| u_0 \|_{L^{\infty}_x} + \frac{1}{4}\left(\| u_0 \|_{L^{\infty}_x}+\frac{1}{4}\left(\| u_0 \|_{L^{\infty}_x} +\frac{1}{4}\| u^{m-3} \|_{L^{\infty}_{t,x}}\right)\right)\\
&\qquad = \left(1+\frac{1}{4} +\left(\frac{1}{4}\right)^2\right)\| u_0 \|_{L^{\infty}_x} + \left(\frac{1}{4}\right)^{3} \| u^{m-3} \|_{L^{\infty}_{t,x}}.
\end{split}
\end{equation*}
Continuing in this way, we find that for every $m\leq n$,
\begin{equation} \label{ln:size-of-u-m}
\begin{split}
&\| u^m \|_{L^{\infty}_{t,x}}  \leq \sum_{k=0}^{m-2}\left(\frac{1}{4}\right)^k \| u^{0} \|_{L^{\infty}_{x}} + \left(\frac{1}{4}\right)^{m-1} \| u^{1} \|_{L^{\infty}_{t,x}} \\
&= \sum_{k=0}^{m-1}\left(\frac{1}{4}\right)^k  \| u^{1} \|_{L^{\infty}_{t,x}} \leq \frac{4}{3}\| u^{1} \|_{L^{\infty}_{t,x}} = \frac{4}{3}\|u_0\|_{L^\infty_x} \leq 2\|u_0\|_{L^\infty_{x}}.
\end{split}
\end{equation}

We now show that $T_{u^n}$ is a contraction map. Suppose $p,\bar{p}\in B_R$.  Invoking the integral estimate in (\ref{ln:size-of-integral-nabla-G}), the estimate on $u^n$ in (\ref{ln:size-of-u-m}), and the condition on $\tau$ in (\ref{Tcondition}) successively,
\begin{equation*}
\begin{split}
     &\|T_{u^n} p - T_{u^n} \bar{p}\|_{L^\infty_{t,x}} = \left\|\int_{0}^t \nabla G_\alpha(t-s)\cdot((p-\bar{p})u^n ) ds \right\|_{L^\infty_{t,x}} \\ 
     &\leq \frac{2\alpha}{2\alpha-1}C \tau^{1-\frac{1}{2\alpha}}  \|p - \bar{p}\|_{L^\infty_{t,x}}\|u^n\|_{L^\infty_{t,x}} \leq 2\frac{2\alpha}{2\alpha-1}C \tau^{1-\frac{1}{2\alpha}} \|p - \bar{p}\|_{L^\infty_{t,x}}\|u_0\|_{L^\infty_{x}}\\ 
     &\qquad \leq \frac{1}{4}\|p-\bar{p}\|_{L^\infty_{t,x}}.
\end{split}
\end{equation*} 

Moreover, observe that $T_{u^n} p\in B_R$ whenever $p\in B_R$, as
\begin{equation} 
\begin{split}
    \|T_{u^n} p\|_{L^\infty_{t,x}} &\leq \|T_{u^n} p - T_{u^n} 0\|_{L^\infty_{t,x}} + \|T_{u^n} 0\|_{L^\infty_{t,x}} \\
    & \leq \frac{1}{4}\|p\|_{L^\infty_{t,x}} + \|\theta_0\|_\infty \leq 2\|\theta_0\|_{L^\infty_{x}} = R.
    \end{split}
\end{equation}
  Thus, $T_{u^n}$ has a fixed point $\theta^{n+1}$ and we conclude by induction that 
  \begin{equation}\label{scalarunif}
  \|\theta^n\|_{L^\infty_{t,x}} \leq 2\|\theta_0\|_{L^\infty_x}   
  \end{equation}
  and \begin{equation} \label{ln:size-of-u-n}
        \|u^n\|_{L^\infty_{t,x}} \leq 2\|u_0\|_{L^\infty_x}  
  \end{equation}
  for all $n>1$.
\medskip

\textbf{Passing to the limit.} We show that the sequences $\{\theta^n\}$ and $\{u^n\}$ are Cauchy. Indeed, we have
\begin{align*}
    \theta^{n + 1} - \theta^n
        &= \int_0^t \grad G_\al(t - s)
            \cdot (\theta^n u^{n - 1} - \theta^{n + 1} u^n)
            \, ds \\
        &= \int_0^t \grad G_\al(t - s)
            \cdot \theta^n (u^{n - 1} - u^n)
            \, ds
            +  \int_0^t \grad G_\al(t - s)
            \cdot (\theta^n - \theta^{n + 1}) u^n
            \, ds, \\
    u^{n + 1} - u^n
        &= \int_0^t (K * \grad G_\al(t - s))
            \cdot (\theta^n u^{n - 1} - \theta^{n + 1} u^n)
            \, ds \\
        &= \int_0^t (K * \grad G_\al(t - s))
            \cdot \theta^n (u^{n - 1} - u^n)
            \, ds \\
            &\qquad
            +  \int_0^t (K * \grad G_\al(t - s))
            \cdot (\theta^n - \theta^{n + 1}) u^n
            \, ds.
\end{align*}
Using the integral bounds in \eqref{ln:size-of-integral-nabla-G} and \eqref{ln:size-of-integral-K-nabla-G}, one has the estimate,
\begin{align*}
    &\norm{\theta^{n + 1} - \theta^n}_{L^\iny_{t, x}}
            + \norm{u^{n + 1} - u^n}_{L^\iny_{t, x}} \\
        &\qquad
        \le
        2 \left( \norm{u_0}_{L^\iny} + \norm{\theta_0}_{L^\iny} \right)
            C\frac{2\alpha}{2\alpha-1} \tau^{1-\frac{1}{2\alpha}} 
           \left( \norm{\theta^{n + 1} - \theta^n}_{L^\iny_{t, x}}
            + \norm{u^n - u^{n - 1}}_{L^\iny_{t, x}}\right) .
\end{align*}
By invoking the size condition on $\tau$ in \eqref{Tcondition}, we can write
\begin{align}\label{unifconv}
    &\norm{\theta^{n + 1} - \theta^n}_{L^\iny_{t, x}}
            + \norm{u^{n + 1} - u^n}_{L^\iny_{t, x}}
        \le
        \frac{1}{4}
           \left(\norm{\theta^{n + 1} - \theta^n}_{L^\iny_{t, x}}
            + \norm{u^n - u^{n - 1}}_{L^\iny_{t, x}}\right) .
\end{align}
Thus
\begin{align*}
    &\norm{u^{n + 1} - u^n}_{L^\iny_{t, x}}
        \le
        \frac{1}{4}
           \norm{u^n - u^{n - 1}}_{L^\iny_{t, x}},
\end{align*}
from which it follows that $\{u^n\}$ is Cauchy and converges to $u$ in $L^\iny((0, \tau) \times \R^2)$. 

It also follows from (\ref{unifconv}) that
\begin{align*}
    \frac{3}{4} \norm{\theta^{n + 1} - \theta^n}_{L^\iny_{t, x}}
        &\le \frac{1}{4} \norm{u^n - u^{n - 1}}_{L^\iny_{t, x}},
\end{align*}
so $\{\theta^n\}$ is Cauchy and converges to $\theta$ in $L^\iny((0, \tau) \times \R^2)$.

Finally, the above observations imply that
\begin{equation}\label{max-principle-on-theta}
\|\theta\|_{L^\infty_{t,x}} \leq 2\|\theta_0\|_{L^\infty_x}
\end{equation}
and
\begin{equation}\label{max-principle-on-u}
\|u\|_{L^\infty_{t,x}} \leq 2\|u_0\|_{L^\infty_x}.
\end{equation}
\noindent \textbf{The limit $(u, \theta)$ is a mild solution}. 
We have
\begin{align*}
     u(t,x)
        - &G_\alpha(t)u_0(x)  + \int_0^t (K \ast\nabla G_\alpha(t-s)) \cdot (u\theta)(s, x) \, ds \\
        &=  u(t, x) - u^{n+1}(t,x)
            + \int_0^t (K \ast\nabla G_\alpha(t-s)) \cdot
                ((u \theta)(s) - (u^n\theta^{n+1})(s, x) \, ds.
\end{align*}
Because $u_n \to u$ and $\theta_n \to \theta$ in $L^\iny((0, \tau) \times \R^2)$ and are bounded in that same space, we have,
\begin{align*}
    \abs{((u \theta)(s) - (u^n\theta^{n+1})(s, x)}
        &\le C \left( \norm{u - u^n}_{L^\iny_{t, x}}
            + \norm{\theta - \theta^{n+1}}_{L^\iny_{t, x}}\right)
        \to 0 \text{ as } n \to \iny.
\end{align*}
It follows that
\begin{align*}
    u(t,x)
        - &G_\alpha(t)u_0(x)  + \int_0^t (K \ast\nabla G_\alpha(t-s)) \cdot (u\theta)(s, x) \, ds
        = 0.
\end{align*}
With the parallel argument for $\theta$, we see that $(u, \theta)$ is a mild solution to \cref{e:SQG} as in \cref{D:mildsolution}.

\noindent\textbf{Uniqueness}.
For uniqueness, suppose $(\theta, u)$ and $(\tilde{\theta},\tilde{u})$ are two mild solutions as in \eqref{D:mildsolution}.
We can then write
\begin{equation}\label{ln:uniqueness-argument}
\begin{split}
    \| \theta - \tilde{\theta}\|_{L^\infty_{t,x}} &+ \| u- \tilde{u}\|_{L^\infty_{t,x}} \\ 
    &= \left\| \int_0^t \nabla G_\alpha(t-s)\cdot (\theta u - \tilde{\theta}\tilde{u}) ds \right\|_{L^\infty_{t,x}} +\left\| \int_0^t K\ast \nabla G_\alpha(t-s)\cdot (\theta u - \tilde{\theta}\tilde{u}) ds \right\|_{L^\infty_{t,x}} \\
    &\leq \frac{4\alpha}{2\alpha-1}C\tau^{1-\frac{1}{2\alpha}} \| \theta u - \tilde{\theta}\tilde{u}\|_{L^
    \infty_{t,x}} = \frac{4\alpha}{2\alpha-1}C\tau^{1-\frac{1}{2\alpha}} \| \theta u -\tilde{\theta}u +\tilde{\theta}u - \tilde{\theta}\tilde{u}\|_{L^
    \infty_{t,x}} \\
    &\leq \frac{4\alpha}{2\alpha-1}C\tau^{1-\frac{1}{2\alpha}} \left( \|u\|_{L^\infty_{t,x}} \|\theta-\tilde{\theta}\|_{L^\infty_{t,x}} + \|\tilde{\theta}\|_{L^\infty_{t,x}} \|u-\tilde{u}\|_{L^\infty_{t,x}} \right).
\end{split}
\end{equation}
We apply the uniform bounds on $\tilde{\theta}$ in \eqref{max-principle-on-theta} and $u$ in \eqref{max-principle-on-u}, and the constraint on $\tau$ in \eqref{Tcondition}.  We conclude that 
\begin{equation} \label{ln:smallness-of-theta-u}
\begin{split}
     \|\theta - \tilde{\theta}\|_{L^\infty_{t,x}} &+ \| u- \tilde{u}\|_{L^\infty_{t,x}}\\
    &\leq  \frac{8\alpha}{2\alpha-1}C\tau^{1-\frac{1}{2\alpha}} \left( \|u_0\|_{L^\infty_{t,x}} \|\theta-\tilde{\theta}\|_{L^\infty_{t,x}} + \|\theta_0\|_{L^\infty_{t,x}} \|u-\tilde{u}\|_{L^\infty_{t,x}} \right) \\
    &\leq \frac{1}{2}\left(  \|\theta-\tilde{\theta}\|_{L^\infty_{t,x}} +   \|u-\tilde{u}\|_{L^\infty_{t,x}} \right).
\end{split}
\end{equation}
Thus, $\theta = \tilde{\theta}$ and $u=\tilde{u}$.  We conclude that $(\theta,u)$ is the unique mild solution on $[0,\tau]$. 
With a simple application of \cref{prop:constitutive-law-holds}, for all $t\in [0,T]$, we have $u(t) = \PV K \ast \theta(t)$. Lastly, the stated continuity properties are a consequence of \cref{P:time-regularity}.


\section{Spatial Regularity of the Solution}\label{sec:regularity}

\noindent In this section we establish the spatial regularity of short-time solutions as is stated in \cref{thm:regularity-of-theta-and-u}. 
 Going forward, we set $D^\gamma = \frac{\partial^{\gamma_1}}{\partial x^{1}}\frac{\partial^{\gamma_2}}{\partial x^{2}}$ for $\gamma \in \N^2$, and we let $D$ denote the partial derivative with respect to either $x_1$ or $x_2$.
\begin{proof}[\textbf{Proof of \cref{thm:regularity-of-theta-and-u}}] \label{proof:regularity-of-theta-and-u}
The argument is similar to that in \cite{Wu}. We will  manipulate (\ref{SQGintegral}) formally by taking spatial derivatives to obtain a map in terms of the derivatives of $u$ and $\theta$. We can then apply a Banach fixed point argument to produce a solution. Our argument will use induction on the number of derivatives of $u$ and $\theta$. 
\subsection{Existence of first derivatives of $\theta$ and $u$}\label{first}
We start by noting that if $D\theta$ and $Du$ exist, then they must satisfy
\begin{equation} \label{ln:first-derivative}
\begin{split}
        D\theta &= G_\alpha(t) D\theta_0 - \int_0^t \nabla G_\alpha(t-s) \cdot\left( (D\theta) u + \theta Du\right) ds, \\
    D u &=  G_\alpha(t) Du_0 - \int_0^t (K\ast \nabla G_\alpha(t-s))\cdot \left( (D\theta) u + \theta Du\right) ds.
\end{split}
\end{equation}
Our strategy is to show that the operator defined by the right hand side of (\ref{ln:first-derivative}) has a fixed point. By uniqueness, the fixed point will correspond to our derivatives $D\theta$ and $Du$. To show that (\ref{ln:first-derivative}) has a fixed point, we apply an argument similar to that used to show (\ref{SQGintegral}) has a fixed point. 

Let $\theta_0 \in C^1_b(\R^2)$ and $u_0 \in (C^1_b(\R^2))^2$.  We will show that the sequence $(\theta_x^n,u_x^n)$ generated by
\begin{equation}
\begin{split}
    \theta_x^1(t,x) = D\theta_0(x), \\
    u_x^1(t,x) = Du_0(x),
\end{split}
\end{equation}
and 
\begin{equation}\label{ln:first-derivative-scheme}
\begin{split}
&{\theta}_x^{n+1}(t,x)  =G_\alpha(t)D\theta_0(x)-\int_0^t \nabla G_\alpha(t-s) \cdot (\theta u_x^{n} + \theta^{n+1}_x u)(s) \, ds,\\
&u_x^{n+1}(t,x) = G_\alpha(t)Du_0(x)  - \int_0^t K \ast\nabla G_\alpha(t-s) \cdot (\theta u^n_x + \theta_x^{n+1} u)(s) \, ds 
\end{split}
\end{equation}
converges to the desired fixed point $(D\theta , Du)$. 
(That is, we will find that the sequence $\theta^{n}_{x}$ converges to a limit
as $n\rightarrow\infty.$  We have previously shown that $\theta$ exists, and
we will call this limit $\theta_{x}.$ However, that will leave a step still to 
establish, which is to show that this $\theta_{x}$ is actually the derivative of
our $\theta$ with respect to $x.$  And, of course, we establish the corresponding
results for $u$ as well.)
We first construct solutions $\theta_x$ and $u_x$ on $[0,\tau]$ with initial data ($\theta_0,u_0$), where $\tau$ is the same as in \eqref{Tcondition}, satisfying
\begin{equation}\label{Tcondition-2}
\frac{2\alpha}{2\alpha-1} C\tau^{1-\frac{1}{2\alpha}} (\| u_0 \|_{L^{\infty}_x} + \| {\theta}_0 \|_{L^{\infty}_x}) \leq 1/8.
\end{equation}
 Set $R= 2\max\{\|\theta_0\|_{C^1_b}, \|u_0\|_{C^1_b}\}$ and $B_R = \{f\in L^\infty([0,\tau]\times \R^2) : \|f\|_{L^\infty_{t,x}} \leq R\}$. We will use an inductive argument to posit the existence and boundedness of the sequence ($\theta^n_x, u^n_x$). For the base case, it is obvious that $\|\theta_x^1\|_{L^\infty_{t,x}} \leq R$ and $\|u_x^1\|_{L^\infty_{t,x}} \leq R$. For the inductive step, suppose that for all $0<t<\tau$, $\theta_x^{n}$ and $ u_x^{n}$ satisfy (\ref{ln:first-derivative-scheme}) with the bounds 
\begin{equation} \label{ln:first-derivative-induction-hypothesis}
    \|\theta_x^n\|_{L^\infty_{t,x}} \leq R \text{   and 
  } \|u_x^n\|_{L^\infty_{t,x}} \leq R.
\end{equation}
We aim to show the existence of $\theta_x^{n+1}$ and $u_x^{n+1}$ satisfying (\ref{ln:first-derivative-scheme}). Using a Banach fixed point argument, assign the maps
\begin{equation*}
T'_\omega p  = DG_\alpha(t)\theta_0-\int_0^t \nabla G_\alpha(t-s) \cdot (\theta\omega + pu) \, ds,    
\end{equation*}
and 
\begin{equation*}
    U'_p \omega= DG_\alpha(t)u_0  - \int_0^t K \ast\nabla G_\alpha(t-s) \cdot (\theta\omega + pu) \, ds,
\end{equation*}
and let $p,\bar{p} \in B_R$.
We apply our uniform bound on $u$ in (\ref{max-principle-on-u}) and the size condition on $\tau$ in (\ref{Tcondition-2}) to yield
\begin{equation} \label{ln:contraction-map-of-derivative}
\begin{split}
    &\|T'_{u_x^{n}} (p-\bar{p})\|_{L^\infty_{t,x}} = \left\| \int_0^t \nabla G_\alpha(t-s) \cdot (u(p-\bar{p})) \, ds \right\|_{L^\infty_{t,x}} \\
    &\qquad \leq 2\frac{2\alpha}{2\alpha-1}  C \tau^{1-\frac{1}{2\alpha}} \|u_0\|_{L^\infty_x} \|p-\bar{p}\|_{L^\infty_{t,x}}\leq \frac{1}{4} \|p-\bar{p}\|_{L^\infty_{t,x}}.
\end{split}
\end{equation}
To see that $T'_{u_x^{n}}$ maps $B_R$ into $B_R$, we use (\ref{ln:contraction-map-of-derivative}), the integral estimate (\ref{ln:size-of-integral-nabla-G}), and the boundedness of $\theta_0$ in $C^1_b$ to write
\begin{equation*}
\begin{split} 
      \|T'_{u^n_x}f\|_{L^\infty_{t,x}} &\leq   \|T'_{u^n_x}f- T'_{u^n_x}0\|_{L^\infty_{t,x}} + \|T'_{u^n_x} 0 \|_{L^\infty_{t,x}} \\
        &\leq \frac{1}{4}  \|p\|_{L^\infty_{t,x}} + \left\|G_\alpha(t)D\theta_0 - \int_0^t \nabla G_\alpha(t-s) (u^n_x \theta) ds  \right\|_{L^\infty_{t,x}} \\
        & \leq \frac{1}{4} \|p\|_{L^\infty_{t,x}} + \|G_\alpha(t)D\theta_0\|_{L^\infty_{t,x}} + \left \| \int_0^t \nabla G_\alpha(t-s) (u^n_x \theta) ds  \right\|_{L^\infty_{t,x}} \\
         &\leq \frac{1}{4} \|p\|_{L^\infty_{t,x}} + \|\theta_0\|_{C^1_b} + \frac{2\alpha}{2\alpha-1}C\tau^{1-\frac{1}{2\alpha}} \|u_x^n \|_{L^\infty_{t,x}} \|\theta\|_{L^\infty_{t,x}}.
\end{split}
\end{equation*}
The uniform bound on $\theta$ in (\ref{max-principle-on-theta}) and (\ref{Tcondition-2}) gives 
\begin{equation*}
\begin{split}                 \|T'_{u^n_x}f\|_{L^\infty_{t,x}} 
        &\leq \frac{1}{4} \|f\|_{L^\infty_{t,x}} + \|\theta_0\|_{C^1_b} + 2\frac{2\alpha}{2\alpha-1} C\tau^{1-\frac{1}{2\alpha}} \|u_x^n \|_{L^\infty_{t,x}} \|\theta_0\|_{L^\infty_{t,x}} \\
        &\leq  \frac{1}{4} \|f\|_{L^\infty_{t,x}} + \|\theta_0\|_{C^1_b} + \frac{1}{4}\|u_x^n \|_{L^\infty_{t,x}}.
\end{split}
\end{equation*}
By our induction hypothesis (\ref{ln:first-derivative-induction-hypothesis}),
\begin{equation}\label{ln:boundedness-of-map-of-derivative-1}
\begin{split}
    \|T'_{u^n_x}f\|_{L^\infty_{t,x}} &\leq \frac{1}{4}R + \|\theta_0\|_{C^1_b} + \frac{1}{4} R \leq R.
\end{split}
\end{equation}
We therefore have shown the existence of $\theta_x^{n+1}$ satisfying $(\ref{ln:first-derivative-scheme})_1$ and (\ref{ln:first-derivative-induction-hypothesis}). We apply a similar argument on $U'_{\theta_x^{n+1}}$ to yield the existence of $u_x^{n+1}$ satisfying $(\ref{ln:first-derivative-scheme})_2$. Note also that, thanks to the integral estimate in (\ref{ln:size-of-integral-K-nabla-G}), the condition on $\tau$ given by (\ref{Tcondition-2}), the uniform bounds on $\theta$ and $u$ given by (\ref{max-principle-on-theta}) and (\ref{max-principle-on-u}) respectively, and our induction hypothesis in (\ref{ln:first-derivative-induction-hypothesis}), we derive the following inequalities.

\begin{equation}\label{ln:boundedness-of-map-of-derivative-2} 
\begin{split}
    \|u_x^{n+1}\|_{L^\infty_{t,x}} &= \left\|  G_\alpha(t)Du_0(x)  - \int_0^t K \ast\nabla G_\alpha(t-s) \cdot (u^n_x\theta + \theta u_x^{n+1})(s) \, ds  \right \|_{L^\infty_{t,x}} \\
    &\leq \|G_\alpha(t)Du_0(x) \|_{L^\infty_{t,x}} + \left\|\int_0^t K \ast\nabla G_\alpha(t-s) \cdot (u^n_x\theta + \theta u_x^{n+1})(s) \, ds  \right\|_{L^\infty_{t,x}} \\
    &\leq \|u_0 \|_{C^1_b} + \frac{2\alpha}{2\alpha-1}C \tau^{1-\frac{1}{2\alpha}} \left( \|u_x^n\|_{L^\infty_{t,x}} \|\theta\|_{L^\infty_{t,x}} 
 + \|u\|_{L^\infty_{t,x}} \|\theta_x^{n+1}\|_{L^\infty_{t,x}} \right) \\
  &\leq \|u_0 \|_{C^1_b}  + \frac{1}{4}\|u_x^n\|_{L^\infty_{t,x}} + \frac{1}{4} \|\theta_x^{n+1}\|_{L^\infty_{t,x}}  \leq R.
\end{split}
\end{equation}

Hence, we can generate the sequences $\{\theta_x^n\}$ and $\{u_x^n\}$, which by construction possess weak-$\star$ limits $\theta_x$ and $u_x$, respectively, in $L^{\infty}_{t,x}$. 

To finalize the proof of the existence of $D\theta$ and $Du$, we show that $\theta_x$ and $u_x$ satisfy (\ref{ln:first-derivative}). To this end, note that by an argument identical to that leading to (\ref{unifconv}), we have 
\begin{equation}\label{Dunifconv}
\| \theta^{n+1}_x-\theta^{n}_x\|_{L^{\infty}_{t,x}} +\| u^{n+1}_x-u^{n}_x\|_{L^{\infty}_{t,x}} \leq \frac{1}{4}(\| \theta^{n+1}_x-\theta^{n}_x\|_{L^{\infty}_{t,x}} +\| u^{n}_x-u^{n-1}_x\|_{L^{\infty}_{t,x}} ),
\end{equation}
from which we conclude, as in Section \ref{existence}, that
\begin{align*}
    &\norm{u_x^{n + 1} - u_x^n}_{L^\iny_{t, x}}
        \le
        \frac{1}{4}
           \norm{u_x^n - u_x^{n - 1}}_{L^\iny_{t, x}}.
\end{align*}
Thus, $\{u_x^n\}$ is Cauchy and converges to $u_x$ in $L^\iny((0, \tau) \times \R^2)$. 

It also follows from (\ref{Dunifconv}) that
\begin{align*}
    \frac{3}{4} \norm{\theta_x^{n + 1} - \theta_x^n}_{L^\iny_{t, x}}
        &\le \frac{1}{4} \norm{u_x^n - u_x^{n - 1}}_{L^\iny_{t, x}},
\end{align*}
so $\{\theta_x^n\}$ is Cauchy and converges to $\theta_x$ in $L^\iny((0, \tau) \times \R^2)$.

Let $Du$ satisfy $Du = U'_{\theta_x} Du$, and let $D\theta$ satisfy $T'_{Du} D\theta = D\theta$. We omit the proof of the existence of $Du$ and $D\theta$, but their existence and the bounds
\begin{equation} \label{ln:Dtheta-Du-maximum-principle}
\|D\theta\|_{L^\infty_{t,x}} \leq R \text{ 
  and   }\|Du\|_{L^\infty_{t,x}} \leq R ,    
\end{equation}
are readily checked with an analogous argument given by the previous considerations in (\ref{ln:contraction-map-of-derivative}), (\ref{ln:boundedness-of-map-of-derivative-1}) and (\ref{ln:boundedness-of-map-of-derivative-2}). 

To see that $\theta_x=D\theta$, we show that $\theta_x$ also satisfies $T'_{Du} \theta_x = \theta_x$. We write
\begin{equation*}
\begin{split}
&\theta_x - DG(t)\theta_0 + \int_0^t \nabla G_{\alpha}(t-s)\cdot(\theta_x u + \theta u_x) \, ds  \\
&\qquad = \theta_x - \theta^n_x + \int_0^t \nabla G_{\alpha}(t-s)\cdot((\theta_x u + \theta u_x) - (\theta^n_xu + \theta u^{n-1}_x)) \, ds.  
\end{split}
\end{equation*}
As $u_x^n \to u_x$ and $\theta_x^n \to \theta_x$ in $L^\iny((0, T) \times \R^2)$ and are bounded in that same space, we have,
\begin{align*}
    \abs{(\theta_x u + \theta u_x) - (\theta^n_xu + \theta u^{n-1}_x)}
        &\le C \left( \norm{\theta_x - \theta_x^n}_{L^\iny_{t, x}}
            + \norm{u_x - u_x^{n-1}}_{L^\iny_{t, x}}\right)
        \to 0 \text{ as } n \to \iny.
\end{align*}
It follows that
\begin{align*}
    \theta_x(t,x)
        - & DG_\alpha(t)\theta_0(x)  + \int_0^t (K \ast\nabla G_\alpha(t-s)) \cdot (\theta_xu +\theta u_x)(s, x) \, ds
        = 0,
\end{align*}
and we have $T'_{Du} \theta_x = \theta_x$.

A similar argument shows that $U'_{D\theta} u_x = u_x$.

We conclude by the uniqueness of the fixed point in the Banach contraction mapping theorem that $\theta_x = D\theta$ and $u_x=Du$. Thus the first order derivatives of $u$ and $\theta$ exist and satisfy (\ref{ln:first-derivative}).

\subsection{Existence of higher derivatives of $\theta$ and $u$}
To establish existence of higher derivatives, we once more use induction. Fix $k\in\N$ and set $R = 2\max\{\|\theta_0\|_{C^{k}_b},\|u_0\|_{C^{k}_b}\}$. The base case is shown in the previous section. For the inductive step, suppose that $D^\beta \theta$ and $D^\beta u$ exist for all $\beta\in \N^2$ satisfying $|\beta| <k$ with the bound
\begin{equation*}
    \|\theta\|_{C^{k-1}_b} \leq R \text{  and  }\|u\|_{C^{k-1}_b} \leq R. 
\end{equation*}

Our objective is to show the existence of $D^\gamma \theta$ and $D^\gamma u$ for $\gamma \in \N^2$ with $|\gamma|=k$. First note that if such derivatives were to exist, then, by the Leibniz rule, they would satisfy 
 \begin{equation} \label{eq:deriviatives-of-thetan}
     D^\gamma \theta = G_\alpha(t)D^\gamma \theta_0 - \int_0^t \nabla G_\alpha(t-s)\cdot\left[\sum_{0\leq\beta \leq\gamma} \binom{\gamma}{\beta}D^\beta \theta D^{\gamma-\beta} u\right]  ds,
 \end{equation}
 and
 \begin{equation}\label{eq:derivatives-of-un}
     D^\gamma u  = G_\alpha(t)D^\gamma u_0  - \int_0^t (K\ast \nabla G_\alpha(t-s))\cdot\left[\sum_{0\leq\beta\leq\gamma}  \binom{\gamma}{\beta}D^\beta \theta D^{\gamma-\beta} u\right] ds.
 \end{equation}

As in the previous subsection, we construct sequences $\{u^n_{\gamma}\}$ and $\{{\theta}^n_{\gamma}\}$ via an iterative scheme and show that the sequences converge to $D^\gamma u$ and $D^\gamma \theta$, respectively. First let
\begin{equation*}
    \theta_\gamma^1 (t,x) \equiv D^\gamma \theta_0(x) \text{  and  } u_\gamma^1 (t,x) \equiv D^\gamma u_0,
\end{equation*}
and, using a similar iterative scheme as in (\ref{ln:first-derivative-scheme}), set
\begin{equation} \label{ln: derivatives-iterative-scheme}
\begin{split}
    &\theta_{\gamma}^{n+1} =  G_\alpha(t)D^\gamma \theta_0 - \int_0^t \nabla G_\alpha(t-s)\cdot\left[\theta_{\gamma}^{n+1}u + \theta u_\gamma^n +\sum_{0<\beta <\gamma} \binom{\gamma}{\beta}D^\beta \theta D^{\gamma-\beta} u\right]  ds, \\
    &\text{ and } \\
    &u_\gamma^{n+1} =  G_\alpha(t)D^\gamma u_0 - \int_0^t (K\ast \nabla G_\alpha(t-s))\cdot\left[\theta_\gamma^{n+1}u + \theta u_\gamma^n +\sum_{0<\beta <\gamma} \binom{\gamma}{\beta}D^\beta \theta D^{\gamma-\beta} u\right]  ds.
\end{split}
\end{equation}
The argument that $\{\theta^{n}_{\gamma}\}$ converges uniformly to $D^{\gamma}\theta$ and that $\{u^n_{\gamma}\}$ converges uniformly to $D^{\gamma}u$ is virtually identical to the analogous argument for the first derivatives of $\theta$ and $u$ as given in Subsection \ref{first}, so we omit the details. Remarkably, the higher derivatives satisfy that following bounds.
\begin{equation}\label{ln:estimate-on-multiderivatives}
    \|D^\gamma\theta\|_{L^\infty_{t,x}} \leq R \text{  and  } \|D^\gamma u \|_{L^\infty_{t,x}} \leq R.    
\end{equation} 
This completes the proof of \cref{thm:regularity-of-theta-and-u}.
\end{proof}

We can also show that the first spatial derivatives of $\theta$ are continuous
in time. The proof follows as a direct variation of \cref{P:time-regularity}. 
\begin{lemma}\label{lem:continuity-in-time-of-derivative-of-theta}
        Let $\alpha>\frac{1}{2}$ and select $\theta_0 \in C_b^1(\R^2)$, $u_0 \in   C_b^1(\R^2))^2$ satisfying $u_0= \PV K\ast \theta_0$. Let $(\theta, u)$ be the mild solution given by \cref{thm:existence-of-solutions} which exists up to time $T$. For all $t < T$, we have $\partial_{x_i}\theta(t, x)$ is continuous in $t$ for all $t>0$ for $i=1,2$.
 \end{lemma}

\begin{proof}
    We first aim to bound the size of $\partial_{x_i} \theta(b,x) - \partial_{x_i} \theta(a,x)$ uniformly in $x$ for $a,b>0$.
    We can write
    \begin{equation}\label{ln:continuity-in-time-theta-difference1}
    \begin{split}
        &\|\partial_{x_i} \theta(b,x)-\partial_{x_i} \theta(a,x)\|_{L^\infty_x} \leq \left \| \left(G_\alpha(b)-G_\alpha(a)\right) \partial_{x_i} \theta_0\right \|_{L^\infty_x}\\
        & + \left\| \int_0^b \nabla G_\alpha(b-s)\cdot (\partial_{x_i} \theta u + \theta \partial_{x_i} u)(s,x) \, ds -  \int_0^a \nabla G_\alpha(a-s)\cdot (\partial_{x_i} \theta u + \theta \partial_{x_i} u))(s,x) \, ds\right\|_{L^\infty_x}\\
        &\leq \left \| \left(G_\alpha(b)-G_\alpha(a)\right) \partial_{x_i} \theta_0(x)\right \|_{L^\infty_x}+\left\| \int_a^b \nabla G_\alpha(b-s)\cdot(\partial_{x_i} \theta u + \theta \partial_{x_i} u) (s,x) \, ds\right\|_{L^\infty_x} \\
        &\qquad\qquad
        + \left\| \int_0^a ( \nabla G_\alpha(b-s)- \nabla G_\alpha(a-s))\cdot (\partial_{x_i} \theta u + \theta \partial_{x_i} u)\,ds\right\|_{L^\infty_x}.
    \end{split}
    \end{equation}
    For the second term on the right hand side of (\ref{ln:continuity-in-time-theta-difference1}), applying Young's inequality and \cref{lem:operator-L1-estimates} for $k=1$ gives the bound,
    \begin{equation*}
    \begin{split}
                &\left\| \int_a^b \nabla G_\alpha(b-s)\cdot(\partial_{x_i} \theta u + \theta \partial_{x_i} u) (s,x) \,ds\right\|_{L^\infty_x}
                \leq C\int_a^b (b-s)^{-\frac{1}{2\alpha}} \|(\partial_{x_i} \theta u + \theta \partial_{x_i} u)(s)\|_{L^\infty_{x}}\, ds\\
                &\qquad\qquad \leq  \frac{2\alpha}{2\alpha-1}C\|\partial_{x_i} \theta u + \theta \partial_{x_i} u\|_{L^\infty_{t,x}} (b-a)^{1-\frac{1}{2\alpha}}.
    \end{split}
    \end{equation*}
    Observe that $\|\partial_{x_i} \theta u + \theta \partial_{x_i} u\|_{L^\infty_{t,x}}$ is bounded as a consequence of \cref{thm:regularity-of-theta-and-u}. 
    Moreover, applying Young's inequality to the third term on the right hand side of (\ref{ln:continuity-in-time-theta-difference1}), we have
    \begin{equation*}
    \begin{split}
        \bigg\| \int_0^a &(\nabla G_\alpha(b-s)- \nabla G_\alpha(a-s))\cdot (\partial_{x_i} \theta u + \theta \partial_{x_i} u)(s,x) \, ds\bigg\|_{L^\infty_x} \\
        &\leq  \|\partial_{x_i} \theta u + \theta \partial_{x_i} u\|_{L^\infty_{t,x}}  \int_0^a \|\nabla g_\alpha(b-s) - \nabla g_\alpha(a-s)\|_{L^1_{x}} \, ds.
    \end{split}
    \end{equation*}
    The argument then follows similarly to that after line \eqref{ln:continuity-in-time-theta-uniformly-cont}.
\end{proof}

Lastly, we conclude with a proof of \cref{cor:solution-to-sqg}.

\section{Extending the Solution}\label{sec:globalSection}

\subsection{Improved bounds on $(\theta,u)$}
Having established the short time existence and regularity of $(\theta,u)$, we now extend the solution to all $t\in [0,\infty)$. In order to achieve an extension, we must establish the appropriate $L^{\infty}$ bounds for $(\theta,u)$; in particular, the current bounds $\|\theta\|_{L^\infty_{t,x}}\leq 2\|\theta_0\|_{L^\infty_{x}}$ and $\|u\|_{L^\infty_{t,x}} \leq 2 \|u_0\|_{L^\infty_{x}} $ only allow us to extend the solution up to some finite time. 

To begin, we restate a convolution-type Gr\"onwall inequality from \cite{webb2019weakly}.

\begin{lemma}[Volterra-Gr\"onwall Inequality (Theorem 3.2 of \cite{webb2019weakly})]\label{lem:volterra-gronwall}
Suppose $v(t)\in L^\infty([0,T])$ with $v(t)\geq 0$ for all $t\in [0,T]$, and let $a\geq 0, b>0$, and $0<\gamma<1$ be constants. If $v(t)$ satisfies the inequality
\begin{equation*}
    v(t) \leq a+b \int_0^t(t-s)^{-\gamma} v(s)\, ds \text{ for a.e } t \in [0,T],
\end{equation*}
then
\begin{equation*}
    v (t) \leq \frac{a}{1-\gamma} \exp\left(\frac{b}{1-\gamma}\left(\frac{\gamma}{bB_0}\right)^{-\frac{\gamma}{1-\gamma}} t\right) \text{ for a.e } t \in [0,T].
\end{equation*}
Here $B_0 = B(1-\gamma,1)$ where $B(x,y)$ is the Beta function.
\end{lemma}
From here, we are able to improve the $L^{\infty}$ bound on $u$.
\begin{prop}\label{betterubound}
    Let $T>0$. Suppose $(\theta,u)$ is a mild solution to \eqref{e:SQG} on $[0,T]$ with initial data $(\theta_0,u_0)\in L^\infty(\R^2)$. Then $u$ satisfies the bound:
  \begin{equation*} 
        \| u (t )\|_{L^\infty_x} \leq \mu \|u_0\|_{L^\infty_x} \exp\left(C_\alpha\|\theta(t)\|_{L^\infty_{x}}^{\mu} t\right),
    \end{equation*}
  where 
   \begin{equation}\label{def:mu-and-C-mu}
       \mu = \frac{2\alpha}{2\alpha-1}, \quad B_0 = B(\mu^{-1}, 1), \quad \text{and}\quad  C_\alpha = \frac{2\alpha}{2\alpha-1}C(2\alpha B_0)^{\frac{1}{2\alpha-1}}.
   \end{equation} 
    \begin{proof}
Applying the $L^\infty$ norm to $(\ref{SQGintegral})_2$ and using the fact that $G_\alpha$ is a probability measure, we obtain,
\begin{equation}
    \| u(t) \|_{L^\infty_x} \leq \|u_0\|_{L^\infty_x} + \left\| \int_0^t K\ast\nabla G_\alpha(t-s)\cdot (\theta u)(s,x) \, ds \right\|_{L^\infty_x}.
\end{equation}

Next, bringing the $L^\infty$ norm inside the integral, invoking Young’s convolution inequality, and applying the bound on $K\ast \nabla G_\alpha(t)$ in \cref{lem:operator-L1-estimates} yields
\begin{equation*}
\begin{split}
  &  \| u (t) \|_{L^\infty_x} \leq \|u_0\|_{L^\infty_x} + \int_0^t C(t-s)^{-\frac{1}{2\alpha}} \| \theta u\|_{L^\infty_x} \, ds\\
  &\qquad \leq \|u_0\|_{L^\infty_x} + \| \theta\|_{L^\infty([0,t];\R^2)}  \int_0^t C(t-s)^{-\frac{1}{2\alpha}} \| u\|_{L^\infty_x} \, ds.
\end{split}    
\end{equation*}
 To reach the conclusion, we employ \cref{lem:volterra-gronwall} with $$a=\|u_0\|_{L^\infty_x},\quad b=C\| \theta\|_{L^\infty([0,t];\R^2)}, \text{    and } \gamma = \frac{1}{2\alpha}$$ to produce the desired inequality,
     \begin{equation*}
        \| u (t )\|_{L^\infty_x} \leq \frac{2\alpha}{2\alpha-1} \|u_0\|_{L^\infty_x} \exp\left(\frac{2\alpha}{2\alpha-1}C \| \theta\|_{L^\infty([0,t];\R^2)}\left(2\alpha B_0 \| \theta\|_{L^\infty([0,t];\R^2)}\right)^{\frac{1}{2\alpha-1}} t\right).
    \end{equation*}  
    \end{proof}

\end{prop}

We now establish a maximum principle for a solution $\theta$ of \eqref{e:SQG} with sufficient regularity.

\begin{prop} \label{prop:theta-max-principle}  
    Let $\alpha\in\left(\frac{1}{2},1\right].$
    Suppose that $(\theta,u)\in (L^\infty([0,T];C^2_b(\R^2)))^3$ is a mild solution to \eqref{e:SQG} on $[0,T]$ for some $T$ with initial data $(\theta_0,u_0)$. Then $\theta$ obeys the maximum principle 
    \begin{equation*}
    \|\theta\|_{L^\infty([0,T]\times \R^2)} \leq \|\theta_0\|_{L^\infty_x} .
    \end{equation*} 
   
\begin{proof}

 With the hypotheses of the proposition, 
 by \cref{P:SQGMotivation}, ($\theta,u$) is a $C^2$-solution to 
\begin{equation}\label{SQG}
\partial_t \theta + u\cdot\nabla\theta = -\nu\Lambda^{2\alpha}\theta
\end{equation}
on $[0,T] \times \mathbb{R}^2$.  

Let $\phi$ be a smooth, compactly supported bump function, with $\phi$ identically one on $B_1(0)$ and  $\supp \phi$ contained in $B_2(0)$.  For each $R>0$ and $x\in\mathbb{R}^2$, set $\phi_R(x) = \phi(x/R)$.

We multiply (\ref{SQG}) by $\phi_R$, which gives
\begin{equation} \label{ln:sqg-bump-version}
    \frac{\partial}{\partial t} (\phi_R \theta) + \phi_R u\cdot \nabla \theta = - \nu\phi_R \Lambda^{2\alpha} \theta.
\end{equation}
Using the product rule, we have
\begin{equation*}
    u \cdot \nabla (\phi_R \theta) = \theta u \cdot\nabla \phi_R +  \phi_R u \cdot\nabla\theta.
\end{equation*}
Making the above substitution and adding and subtracting $\nu\Lambda^{2\alpha}(\phi_R\theta)$ , one has 
\begin{equation}\label{SQGloc}
\begin{split}
    &\frac{\partial}{\partial t} (\phi_R \theta) + u \cdot \nabla(\phi_R \theta)  = \theta u \cdot \nabla\phi_R- \nu\phi_R \Lambda^{2\alpha} \theta +\nu\Lambda^{2\alpha}(\phi_R\theta) - \nu\Lambda^{2\alpha}(\phi_R\theta)\\
    &\qquad  = -\nu\Lambda^{2\alpha}(\phi_R\theta) + I + II,
\end{split}
\end{equation}
where $$I = u\theta\cdot\nabla\phi_R$$ and $$II = \nu [\Lambda^{2\alpha}, \phi_R]\theta.$$
Now, for $p\geq 2$, multiply (\ref{SQGloc}) by $p|\phi_R\theta|^{p-2}\phi_R\theta$ and integrate over $\mathbb{R}^2$.  This gives
\begin{equation*}
\begin{split}
    &\int_{\R^2} p|\phi_R\theta|^{p-2}\phi_R\theta \frac{\partial}{\partial t}(\phi_R \theta)\, dx  = - \int_{\R^2} 
 p|\phi_R\theta|^{p-2}\phi_R\theta u\cdot\nabla(\phi_R\theta)\, dx \\ 
 &\qquad -\int_{\R^2}  \nu p|\phi_R\theta|^{p-2}\phi_R\theta\Lambda^{2\alpha}(\phi_R\theta)\, dx  + \int_{\R^2} p|\phi_R\theta|^{p-2}\phi_R\theta(I + II)\, dx .    
\end{split}
\end{equation*}
Using a weak derivative formulation, we apply the identity $\frac{d}{dt} |z(t)|^p = p |z(t)|^{p-2} z(t) \frac{dz}{dt}$ and the Liebniz integral rule to conclude that
\begin{equation*}
\begin{split}
         \int_{\R^2} p|\phi_R\theta|^{p-2}\phi_R\theta \frac{\partial}{\partial t}(\phi_R \theta)dx &= \int_{\R^2} \frac{\partial}{\partial t}  |\phi_R \theta|^p dx = \frac{d}{dt} \int_{\R^2}  |\phi_R \theta|^p dx.  
\end{split}  
\end{equation*}
Therefore, we have 
\begin{equation} \label{ln:divergence-of-u-term}
\begin{split}
&\frac{d}{dt} \|\phi_R\theta\|_{L^p}^{p} = -p\int_{\R^2} |\phi_R\theta|^{p-2}\phi_R\theta u\cdot\nabla(\phi_R\theta) \, dx\\
&\qquad -\nu p\int_{\R^2} |\phi_R\theta|^{p-2}\phi_R\theta\Lambda^{2\alpha}(\phi_R\theta) \, dx + p\int_{\R^2} |\phi_R\theta|^{p-2}\phi_R\theta(I + II) \, dx. 
\end{split}
\end{equation}
With the divergence free condition on $u_0$, by \cref{L:divuZero}, $u$ is divergence free for all time, so we can recast the first term on the right hand side of \eqref{ln:divergence-of-u-term} as
\begin{equation} \label{ln:divergence-of-u-zero}
\begin{split}
    p \int_{\R^2} |\phi_R \theta|^{p-2} \phi_R \theta \, u \cdot \nabla (\phi_R \theta) \, dx  
    &= \int_{\R^2} u \cdot \nabla (\phi_R \theta)^p \, dx \\
    &= -\int_{\R^2} (\phi_R \theta)^p \operatorname{div} u \, dx = 0.
\end{split}
\end{equation}
Moreover, by Lemma 2.5 of \cite{cordoba2004maximum}, $$-p\nu \int_{\R^2} |\phi_R\theta|^{p-2}(\phi_R\theta)\Lambda^{2\alpha}(\phi_R\theta) \, dx \leq 0.$$  Thus,
\begin{equation}\label{ln:positivity-lemma}
p\| \phi_R\theta \|_{L^p_x}^{p-1} \frac{d}{dt} \| \phi_R\theta \|_{L^p_x} \leq p\int_{\R^2} |\phi_R\theta|^{p-2}\phi_R\theta(I + II) \, dx.
\end{equation}
We divide both sides of (\ref{ln:positivity-lemma}) by $p\| \phi_R\theta \|_{L^p_x}^{p-1}$.  We conclude that
\begin{equation}\label{postdivide}
\begin{split}
&\frac{d}{dt} \| \phi_R\theta \|_{L^p_x} \leq p\left(\frac{1}{p\| \phi_R\theta \|_{L^p_x}^{p-1}}\right)\int_{\R^2} |\phi_R\theta|^{p-2}\phi_R\theta(I + II) \, dx  \\
&\qquad \leq \| I + II\|_{L^{\infty}_x} \frac{\| \phi_R\theta \|_{L^{p-1}_x}^{p-1}}{\| \phi_R\theta \|_{L^p_x}^{p-1}} = \| I + II\|_{L^{\infty}_x} \left(\frac{\| \phi_R\theta \|_{L^{p-1}_x}}{\| \phi_R\theta \|_{L^p_x}}
\right)^{p-1}.
\end{split}
\end{equation}
We now bound the second term in the product. By the generalized H\"older inequality and the compactness of the support of $\phi$,
$$\| \phi_R\theta \|_{L^{p-1}_x} = \| \phi_{2R}\phi_R\theta \|_{L^{p-1}_x} \leq \| \phi_{2R} \|_{L^{p(p-1)}_x}\| \phi_R\theta \|_{L^{p}_x} \leq (2\pi R^2)^{\frac{1}{p(p-1)}}\| \phi_R\theta \|_{L^{p}_x}.$$
This implies 
$$\left(\frac{\| \phi_R\theta \|_{L^{p-1}_x}}{\| \phi_R\theta \|_{L^p_x}}
\right)^{p-1} \leq  (2\pi R^2)^{\frac{1}{p}}.$$

Substituting the above bound into (\ref{postdivide}), we find that for any $p\in [1,\infty)$ and any fixed $R<\infty$,
\begin{equation*}
    \frac{d}{dt} \| \phi_R\theta \|_{L^p_x} \leq (2\pi R^2)^{\frac{1}{p}} \| I(s) + II(s)\|_{L^{\infty}_x}.
\end{equation*}
Integrating the above expression over time, we get
\begin{equation*} 
    \|\phi_R \theta(t) \|_{L^p_x} \leq \|\phi_R \theta_0 \|_{L^p_x} + (2\pi R^2)^{\frac{2}{p}}\int_0^t\|I(s) + II(s)\|_{L^\infty_x} \, ds.
\end{equation*}
We take the limit as $p \to \infty$ to produce the bound
\begin{equation} \label{ln:L-infty-cutoff-bound}
    \|\phi_R \theta(t) \|_{L^\infty_x} \leq \|\phi_R \theta_0 \|_{L^\infty_x} + \int_0^t \|I(s) + II(s)\|_{L^\infty_x} \, ds.
\end{equation}

We claim that $\|I(t)+II(t)\|_{L^\infty_{x}}<\infty$ for all $t$ and moreover, $$\lim_{R\rightarrow \infty}\| I(t) + II(t)\|_{L^{\infty}_x}= 0.$$

Clearly we have
\begin{equation}\label{ln:L-infty-I-term}
\|I\|_{L^{\infty}_x} \leq \frac{1}{R} \left\|(\theta u) (s,x) \right\|_{L^{\infty}_x} \|\nabla\phi\|_{L^{\infty}_x}.    
\end{equation}

We now estimate the commutator term $II$. We consider two cases separately: $\alpha\in (1/2, 1)$ and $\alpha=1$. First assume $\alpha\in (1/2, 1)$. We expand the fractional Laplacian using the singular integral definition \eqref{def:frac-laplacian-sio}, since $\theta\in C^2_b(\R^2)$. We then simplify the commutator as follows:

\begin{equation*}
\begin{split}
[\Lambda^{2\alpha}, \phi_R]\theta(x) &= \phi_R(x)\int_{\R^2} \frac{\theta(x) - \theta(y)}{|x-y|^{2+2\alpha}} \, dy - \int_{\R^2} \frac{\phi_R(x)\theta(x) - \phi_R(y)\theta(y)}{|x-y|^{2+2\alpha}} \, dy\\
&= \int_{\R^2} \frac{\phi_R(y)\theta(y) - \phi_R(x)\theta(y)}{|x-y|^{2+2\alpha}} \, dy.
\end{split}
\end{equation*}
We now split up the above integral into two parts,
\begin{equation*}
\begin{split}
[\Lambda^{2\alpha}, \phi_R]\theta(x) 
&= III + IV,
\end{split}
\end{equation*}
where
\begin{equation*}
\begin{split}
&III = \int_{\mathbb{R}^2} \phi(x - y) \frac{\phi_R(y)\theta(y) - \phi_R(x)\theta(y)}{|x - y|^{2 + 2\alpha}} \, dy,\\
&IV = \int_{\mathbb{R}^2} (1 - \phi)(x - y) \frac{\phi_R(y)\theta(y) - \phi_R(x)\theta(y)}{|x - y|^{2 + 2\alpha}} \, dy.
\end{split}
\end{equation*}
We estimate $III$ and $IV$ separately. Starting with $III$, write 
\begin{equation*}
\begin{split}
&III  = \int_{\R^2} \phi(x-y)\frac{\phi_R(y)\theta(y) - \phi_R(x)\theta(y)}{|x-y|^{2+2\alpha}} \, dy\\
& = \int_{\R^2} \phi(x-y)\frac{(\phi_R(y) - \phi_R(x))(\theta(y)-\theta(x)) + \phi_R(y)\theta(x) - \phi_R(x)\theta(x)}{|x-y|^{2+2\alpha}} \, dy\\
&= \int_{\R^2} \phi(x-y)\frac{(\phi_R(y) - \phi_R(x))(\theta(y)-\theta(x))}{|x-y|^{2+2\alpha}} \, dy 
+ \theta(x)\int_{\R^2}\phi(x-y)\frac{(\phi_R(y) - \phi_R(x))}{|x-y|^{2+2\alpha}} \, dy.
\end{split}
\end{equation*}
We invoke the Lipschitz bounds on $\phi_R$ and $\theta$, as well as the singular integral definition of $\Lambda^{2 \alpha}$, and find that
\begin{equation} \label{ln:expression-for-iii}
\begin{split}
&III  \leq \|\nabla\phi_R\|_{L^{\infty}_x}\|\nabla\theta\|_{L^{\infty}_x}\int_{\R^2} \frac{\phi(x-y)}{|x-y|^{2\alpha}}  + \abs{\theta(x)\int_{\R^2}\phi(x-y)\frac{(\phi_R(y) - \phi_R(x))}{|x-y|^{2+2\alpha}} \, dy }\\
&\qquad \leq  C\|\nabla\phi_R\|_{L^{\infty}_x}\|\nabla\theta\|_{L^{\infty}_x} + \abs{\theta(x) \Lambda^{2\alpha} \phi_R(x)}\\
&\qquad\qquad\qquad + \abs{\theta(x)\int_{\R^2}(1-\phi)(x-y)\frac{(\phi_R(y) - \phi_R(x))}{|x-y|^{2+2\alpha}} \, dy }\\
&\qquad \leq C\|\nabla\phi_R\|_{L^{\infty}_x}\|\nabla\theta\|_{L^{\infty}_x} + \|\theta\|_{L^{\infty}_x} \|\Lambda^{2\alpha} \phi_R\|_{L^{\infty}_x} + C\|\theta\|_{L^{\infty}_x}\|\nabla\phi_R\|_{L^{\infty}_x}.
\end{split}
\end{equation}

Moving on to $IV$, we use the Lipschitz bound of $\phi_R$ to write
\begin{equation}\label{IVbound}
\begin{split}
&\qquad IV \leq \int_{\R^2}(1-\phi)(x-y)\frac{(\phi_R(y) - \phi_R(x))}{|x-y|^{2+2\alpha}}|\theta(y)| \, dy\\
&\leq  \|\theta\|_{L^{\infty}_x}\|\nabla\phi_R\|_{L^{\infty}_x}\int_{\R^2}\frac{(1-\phi)(x-y)}{|x-y|^{1+2\alpha}} \, dy\leq C\|\theta\|_{L^{\infty}_x}\|\nabla\phi_R\|_{L^{\infty}_x}. 
\end{split}
\end{equation}
Combining \eqref{ln:expression-for-iii} and (\ref{IVbound}), we conclude that 
$$\|  [\Lambda^{2\alpha}, \phi_R]\theta(x) \|_{L^{\infty}_x} =\|III+IV\|_{L^{\infty}_x} \leq \tilde{C}\| \theta\|_{C^1_b}( \|\nabla\phi_R\|_{L^{\infty}_x} + \|\Lambda^{2\alpha} \phi_R\|_{L^{\infty}_x}).$$
It remains to show that as $R\to\infty$, the above quantities vanish in the limit.  Clearly we have
\begin{equation} \label{ln:decay-of-phi-R}
    \|\nabla \phi_R\|_{L^\infty_x}  
    \leq \frac{1}{R} \left\| (\nabla \phi)\left(\frac{x}{R}\right) \right\|_{L^\infty_x} 
    = \frac{1}{R} \|\nabla \phi\|_{L^\infty_x}.
\end{equation}

Lastly, we have to prove $\|\Lambda^{2\al} \phi_R\|_{L^{\infty}_x} \to 0$ as $R$ approaches infinity.  To see why this is the case, write
\begin{equation}
\begin{split}
\Lambda^{2\al} \phi_R(x) &= \int_{\R^2}
            \frac{\phi_R(y) - \phi_R(x)}
                {\abs{x - y}^{2+2\alpha}} \, dy = \frac{1}{R^{2+2\al}}\int_{\R^2}
            \frac{\phi(y/R) - \phi(x/R)}
                {\abs{(x/R) - (y/R)}^{2+2\al}} \, dy \\
&\qquad = \frac{1}{R^{2\al}}\int_{\R^2}
            \frac{\phi(z) - \phi(x/R)}
                {\abs{(x/R) - z}^{2+2\al}} \, dz= \frac{1}{R^{2\al}} (\Lambda^{2\al} \phi)(x/R).
\end{split}
\end{equation}
Therefore, 
\begin{equation} \label{ln:frac-laplace-to-0}
\begin{split}
    \| \Lambda^{2\al} \phi_R\|_{L^{\infty}_x} = \frac{1}{R^{2\al}}\| (\Lambda^{2\al} \phi)(\cdot/R)\|_{L^{\infty}_x} = \frac{1}{R^{2\al}}\| \Lambda^{2\al} \phi\|_{L^{\infty}_x}.
\end{split}
\end{equation}
Gathering \eqref{ln:decay-of-phi-R} and \eqref{ln:frac-laplace-to-0}, we obtain the following commutator estimate:
\begin{equation}\label{ln:commutator-estimate}
    \|  [\Lambda^{2\alpha}, \phi_R]\theta(x) \|_{L^{\infty}_x} \leq \tilde{C}\| \theta\|_{C^1_b} \left( \frac{1}{R} \|\nabla \phi\|_{L^\infty_x} + \frac{1}{R^{2\al}}\| \Lambda^{2\al} \phi\|_{L^{\infty}_x} \right).
\end{equation}
Substituting \eqref{ln:commutator-estimate} and the estimate for $I$ in \eqref{ln:L-infty-I-term} into the $L^\infty$ bound for the cutoff of $\theta$ in \eqref{ln:L-infty-cutoff-bound}, we conclude that
\begin{equation*}
\begin{aligned}
    \|\phi_R \theta(t)\|_{L^\infty_x}  
    &\leq \|\phi_R \theta_0 \|_{L^\infty_x} +\int_0^t \Bigg( 
        \frac{1}{R} \|(u\theta)(s)\|_{L^\infty_x} \|\nabla \phi\|_{L^\infty_x} \\
    &\quad + \tilde{C}\nu \| \theta(s)\|_{L^\infty_t C^1_b} 
        \left( \frac{1}{R} \|\nabla \phi\|_{L^\infty_x} 
        + \frac{1}{R^{2\alpha}} \|\Lambda^{2\alpha} \phi\|_{L^\infty_x} 
        \right) 
    \Bigg) \, ds.
\end{aligned}
\end{equation*}
Thus, given our hypotheses,
$\theta\in L^\infty([0,T]; C^1_b(\R^2)),$ we can rewrite the above expression as
\begin{equation}\label{ln:L-infty-cutoff-theta-prelimit-step}
\begin{aligned}
    \|\phi_R \theta(t)\|_{L^\infty_x}  
    &\leq \|\phi_R \theta_0 \|_{L^\infty_x} +t \Bigg( 
        \frac{1}{R} \|\theta u\|_{L^\infty_{t,x}} \|\nabla \phi\|_{L^\infty_x} \\
    &\quad + \tilde{C} \| \theta\|_{L^\infty_t C^1_b} 
        \left( \frac{1}{R} \|\nabla \phi\|_{L^\infty_x} 
        + \frac{1}{R^{2\alpha}} \|\Lambda^{2\alpha} \phi\|_{L^\infty_{x}} 
        \right) 
    \Bigg) .
\end{aligned} 
\end{equation}
Taking the limit $R\to\infty$ and invoking the weak-$\star$ convergence of $\phi_R \theta \to \theta$ and the uniform boundedness principle yields the desired conclusion,
\begin{equation}
    \|\theta(t)\|_{L^\infty_x}  
    \leq \|\theta_0\|_{L^\infty_x},
\end{equation}
for all $t\in [0,\tau)$, when $\alpha\in (1/2, 1)$.

Now suppose that $\alpha = 1$. We again fix $R>0$ and multiply (\ref{SQG}) by the radial function $\phi_R$, which gives
\begin{align*}
    \frac{\partial}{\partial t} (\phi_R \theta) + \phi_R u\cdot \nabla \theta
    = \nu \phi_R \Delta \theta.
\end{align*}
We expand the Laplacian term as
\begin{align*}
    \phi_R \Delta \theta
        &= \Delta (\phi_R \theta)
            - \theta \Delta \phi_R
            - 2 \grad \theta \cdot \grad \phi_R.
\end{align*}
Proceeding as before, we multiply by $p \abs{\phi_R \theta}^{p - 2} \phi_R \theta$ and integrate over $\R^2$, leading to
\begin{align*}
    \diff{}{t} \norm{\phi_R \theta}_{L^p}^p
        &\le p \int_{\R^2}
            \abs{\phi_R \theta}^{p - 2} \phi_R \theta
            (I + II') \, dx,
\end{align*}
where
\begin{align*}
    II'
        &= - \nu\theta \Delta \phi_R
            - 2 \nu\grad \theta \cdot \grad \phi_R.
\end{align*}

For the new term, we have
\begin{equation} \label{ln:II-bound-for-alpha-1}
\begin{split}
    \norm{II'}_{L^\iny}
        &\le \nu\norm{\Delta \phi_R}_{L^\iny}
            \norm{\theta}_{L^\iny}
        + 2 \nu\norm{\grad \phi_R}_{L^\iny}
            \norm{\grad \theta}_{L^\iny} \\
        &\le \frac{\tilde{C}'\nu}{R^2}
            \norm{\theta}_{L^\iny}
        + \frac{\tilde{C}'\nu}{R}
            \norm{\grad \theta}_{L^\iny},    
\end{split}
\end{equation}
Repeating the argument up to \eqref{ln:L-infty-cutoff-bound}, we get
\begin{equation*}
    \|\phi_R \theta(t) \|_{L^\infty_x} \leq \|\phi_R \theta_0 \|_{L^\infty_x} + \int_0^t \|I(s) + II'(s) \|_{L^\infty_x}\, ds.
\end{equation*}
By repeating an analogous argument as above, we arrive at the same conclusion. That is, for $\alpha = 1$,
\begin{equation}
    \|\theta(t)\|_{L^\infty_x}  
    \leq \|\theta_0\|_{L^\infty_x}
\end{equation}
for all $t\in[0,T]$.
\end{proof}
\end{prop}

\subsection{Extending the Solution}\label{proof-of-global-in-time-solution}

In this section we prove \cref{thm:global-in-time-solution,thm:L-infty-global-solution}.

\begin{proof}[\textbf{Proof of \cref{thm:global-in-time-solution}}]
By \cref{thm:existence-of-solutions}, we can produce a short time solution $(\theta,u)$ on $[0,\tau]$ for some $\tau >0$. By \cref{thm:regularity-of-theta-and-u}, $(\theta(t),u(t)) \in (C^2_b(\R^2))^3$ for all $t\in [0,\tau]$. This implies that $(\theta,u)$ is a classical solution to \eqref{e:SQG} on $[0,\tau]$ by \cref{prop:constitutive-law-holds} and \cref{P:SQGMotivation}. Moreover, from the construction of the solution, we have $\|\theta(t)\|_{L^\infty_x} \leq 2\|\theta_0\|_{L^\infty_x}$ and $\|u(t)\|_{L^\infty_x} \leq 2\|u_0\|_{L^\infty_x}$. In fact, we have better bounds for the solution: 
\begin{equation} \label{eq:extension-estimates}
    \begin{aligned}
        &\text{(i) By \cref{prop:theta-max-principle}, } \|\theta (t)\|_{L^\infty_x}\leq \|\theta_0\|_{L^\infty_x}, \quad \forall t\in [0,\tau]. \\
        &\text{(ii) By \cref{betterubound}, } \|u(t)\|_{L^\infty_x}\leq \frac{2\alpha}{2\alpha-1} \|u_0\|_{L^\infty_x}\exp(C_\alpha \|\theta_0\|_{L^\infty_x}^\mu t), \quad \forall t\in [0,\tau],
    \end{aligned}
\end{equation}

where $$ \mu = \frac{2\alpha}{2\alpha-1} \quad \text{ and }\quad  C_\alpha = \frac{2\alpha}{2\alpha-1}C(2\alpha B_0)^{\frac{1}{2\alpha-1}}.$$

Let $\tau_1= \tau$. Inductively, we will create a sequence $\{\tau_n\}_{n=1}^\infty$ of finite times for which we can extend the solution. To do so, for each additional extension time $\tau_n$, we verify that the conditions \eqref{eq:extension-estimates} hold. 

For each integer $k\geq 1$, define $$S_{k} := \sum_{j=1}^{k} \tau_j.$$ Suppose that we have constructed the solution up to time $S_{n-1}$ and that \eqref{eq:extension-estimates} holds for all $t\in [0,S_{n-1}]$. Then, define $\tau_n$ as in \eqref{Tcondition} by
\begin{equation}\label{Tcondition-i}
\frac{2\alpha}{2\alpha-1}C\tau_n^{1-\frac{1}{2\alpha}} \left( \|\theta_0\|_{L^\infty_{x}} +\frac{2\alpha}{2\alpha-1} \|u_0\|_{L^\infty_x} \exp\left(C_\alpha \|\theta_0\|_{L^\infty_x}^\mu S_{n-1}\right)\right) \leq \tfrac{1}{8}.
\end{equation}
For this $\tau_n$, setting $S_n := S_{n-1} + \tau_n$, we use \cref{thm:existence-of-solutions} to generate a mild solution to \eqref{e:SQG} on $[S_{n-1}, S_n]$ whose initial data is given by $(\theta(S_{n-1},x),u(S_{n-1},x))$. With \cref{thm:regularity-of-theta-and-u}, the extended solution satisfies $(\theta,u)\in (L^\infty([0,S_n];C^2_b(\R^2)))^3$. 
Moreover, on the interval $[0,S_n]$, we have the a priori bounds on $\theta$ and $u$ given by
\begin{equation*}
\begin{split}
    &\|\theta \|_{L^\infty([0, S_n]\times \R^2)} \leq \max \left\{\|\theta \|_{L^\infty([S_{n-1}, S_n]\times \R^2)}, \|\theta \|_{L^\infty([0, S_{n-1}]\times \R^2)} \right\} \\
    &\qquad\qquad \leq \max \left\{ 2\|\theta(S_{n-1},\cdot) \|_{L^\infty_x}, \|\theta_0\|_{L^\infty_x} \right\} \leq 2\|\theta_0\|_{L^\infty_x},
    \end{split}
\end{equation*}
and similarly,
\begin{equation*}
\begin{split}
    &\|u\|_{L^\infty([0, S_n]\times \R^2)} \leq \max \left\{\|u \|_{L^\infty([S_{n-1}, S_n]\times \R^2)}, \|u \|_{L^\infty([0,S_{n-1}]\times \R^2)} \right\} \\
    &\qquad\qquad \leq \max \left\{2\|u (S_{n-1},\cdot) \|_{L^\infty_x}, \frac{2\alpha}{2\alpha-1}\|u_0\|_{L^\infty_x}\exp(C_\alpha \|\theta_0\|_{L^\infty_x}^\mu t) \right\}\\
    &\qquad\qquad\leq  \frac{4\alpha}{2\alpha-1}\|u_0\|_{L^\infty_x}\exp(C_\alpha \|\theta_0\|_{L^\infty_x}^\mu t).
    \end{split}
\end{equation*}
Thus, we improve these estimates to match our induction hypothesis.
\begin{enumerate}
   \item We now have that a classical solution exists on $[0,S_{n}]$. Thus, with \cref{prop:theta-max-principle}, we deduce $\|\theta (t)\|_{L^\infty_x}\leq \|\theta_0\|_{L^\infty_x}$ for all $t\in [0,S_{n}]$.
    \item By \cref{betterubound}, $\|u(t)\|_{L^\infty_x}\lesssim \|u_0\|_{L^\infty_x}\exp(C_{\alpha}\|\theta_0\|_{L^\infty_x}^\mu t)$ for some exponent $\mu>0$ and all $t\in [0,S_{n}].$

\end{enumerate}

Lastly, we verify that this extension process will cover all time $t\in [0,\infty)$.
Rearranging \eqref{Tcondition-i} and solving for \(\tau_n\) gives
\[
\tau_n \le \left[\frac{1}{8}\,\frac{2\alpha-1}{2\alpha}\,\frac{1}{C}\,\frac{1}{\displaystyle \|\theta_0\|_{L^\infty_{x}} +\frac{2\alpha}{2\alpha-1}\,\|u_0\|_{L^\infty_x}\,\exp\!\left(C_\alpha \|\theta_0\|_{L^\infty_x} ^\mu S_{n-1} \right)}\right]^{\frac{2\alpha}{2\alpha-1}}.
\]
For large \(S_{n-1}\) the exponential term dominates so that
\[
\tau_n \gtrsim \exp\!\left(-\lambda\, S_{n-1}\right),
\]
for some \(\lambda>0\). It is clear that $S_n$ satisfies the discrete recurrence,
\[
S_n = S_{n-1} + \tau_n \quad \text{with} \quad \tau_n \gtrsim \exp(-\lambda\,S_{n-1}).
\]
For that reason, consider the continuous ODE analogue:
\[
\frac{dS}{dn} = \exp(-\lambda S),\quad S(0)=S_0.
\]
Hence,
\[
S(n) = \frac{1}{\lambda}\ln\Bigl(\lambda\, n+\exp(\lambda S_0)\Bigr).
\]
Since \(\ln(\lambda\, n+\exp(\lambda S_0))\to\infty\) as \(n\to\infty\), we conclude that \(S(n)\to\infty\). By a discrete-continuous comparison argument, it follows that
\[
\sum_{n=1}^\infty \tau_n = \lim_{n\to\infty} S_n = \infty.
\]

Therefore, we conclude that the solution can be extended for arbitrary time. Moreover, for any $T>0$, we have the bounds,
\begin{enumerate}
    \item $\|\theta\|_{L^\infty([0,T]\times \R^2)}\leq \|\theta_0\|_{L^\infty(\R^2)}$ and
    \item $\|u\|_{L^\infty([0,T]\times \R^2)} \leq \|u_0\|_{L^\infty_x} \exp\left( C_\alpha \|\theta_0\|_{L^\infty_x} ^\mu T\right) $.
\end{enumerate}
We conclude from \cref{thm:regularity-of-theta-and-u} that if $(\theta_0, u_0) \in C^k_b(\R^2)$, then the higher derivatives of $(\theta,u)$ also exist on the interval $[0,\tau_n]$ for $n$ arbitrarily large.  Moreover, we have the simple estimate  $\|\nabla^{k} \theta\|_{L^\infty([0,\tau_n]\times\R^2)} \leq 2^n \|\nabla^k \theta_0\|_{L^\infty_{x}}  $ and $\|\nabla^k u\|_{L^\infty([0,\tau_n]\times\R^2)}  \leq 2^{n}\|\nabla u_0\|_{L^\infty_{x}}$.
\end{proof}

Having extended the $C^k$ solution for $k\geq 2$ to be global in time, we now 
use this result to extend solutions whose initial data is in $L^\infty(\R^2)$.
We present the proof of \cref{thm:L-infty-global-solution}
\begin{proof}[\textbf{Proof of \cref{thm:L-infty-global-solution}}] \label{proof:L-infty-global-solution}
    For the mild solutions $(\theta,u)$, if $(\theta_0,u_0)$ are not $C^\gamma$ continuous for some $\gamma$, set $(\theta_0,u_0):= (\theta(t),u(t))$ for some arbitrarily small $t\in [0,T]$. As a consequence of \cref{P:gamma-holder-regularity}, we can assume that there exists $\gamma>0$ such that $u_0$ and $\theta_0$ belong to $C^{\gamma}(\R^2)$. Let $\varphi:\R^2\to \R$ be a smooth bump function and consider $\varphi_n (x):= n^2\varphi \left(nx\right).$ Consider the sequence $\theta_0^n := \varphi_n \ast \theta_0$. Clearly, $\theta_0^n \in L^\infty(\R^2)$ and moreover $u_0^n:= \PV K\ast \theta_0^n$ exists and is an element of $L^\infty(\R^2)$ as
    \begin{equation*}
    \begin{split}
        \|\PV K \ast (\varphi_n \ast \theta)\|_{\infty} \leq \|\varphi_n\|_{L^1} \|\PV K \ast \theta \|_{L^\infty}. 
    \end{split}
    \end{equation*}
    Because of the convolution, we have that $(\theta_0^n, u_0^n)$ is sufficiently
    regular that \cref{thm:global-in-time-solution} applies.
    We can therefore let $(\theta^n, u^n)$ be the classical solution on $[0,T]$ with initial data $(\theta_0^n, u_0^n)$. We show that $\{(\theta^n, u^n)\}_{n=1}^\infty$ is Cauchy in $L^{\infty}_{t,x}$. For $n>m\geq 0$,
    \begin{equation*}
    \begin{split}
        \|\theta^n - \theta^m\|_{L^\infty_{x}} &\leq \left \| G_\alpha(t)(\theta_0^n - \theta_0^m) \right \|_{L^\infty_{x}} + \int_0^t \left\| \nabla G_\alpha(t-s)\cdot ((\theta^n u^n- \theta^m u^m)(s,x))   \right\|_{L^\infty_{x}} \, ds
    \end{split}
    \end{equation*}
and similarly for $u$,
    \begin{equation*}
    \begin{split}
    \|u^n - u^m\|_{L^\infty_{x}} &\leq \left \| G_\alpha(t)(u_0^n - u_0^m) \right\|_{L^\infty_{x}}+ \int_0^t \left\|(K\ast \nabla G_\alpha(t-s))\cdot ((\theta^n u^n- \theta^m u^m)(s,x)) \right\|_{L^\infty_{x}}\, ds.
    \end{split}
    \end{equation*}

Adding the above lines together we  invoke the boundedness of $G_\alpha(t)$,
\begin{equation*}
\begin{split}
   \|\theta^n - \theta^m\|_{L^\infty_{x}}+ &\|u^n - u^m\|_{L^\infty_{x}} 
  \leq \|\theta_0^n - \theta_0^m\|_{L^\infty_{t,x}}+ \|u_0^n - u_0^m\|_{L^\infty_{t,x}}\\
  &+ \int_0^t \|  \nabla G_\alpha(t-s)\cdot ((\theta^n u^n- \theta^m u^m)(s,x)) \|_{L^\infty_x} \,ds \\
  &+ \int_0^t \|  K\ast \nabla G_\alpha(t-s)\cdot ((\theta^n u^n- \theta^m u^m)(s,x)) \|_{L^\infty_x} \,ds.
\end{split}
\end{equation*}
Now apply the kernel estimates of \cref{lem:operator-L1-estimates},
\begin{equation*}
\begin{split}
  \|\theta^n - \theta^m\|_{L^\infty_{x}}+ &\|u^n - u^m\|_{L^\infty_{x}} 
  \leq \|\theta_0^n - \theta_0^m\|_{L^\infty_{t,x}}+ \|u_0^n - u_0^m\|_{L^\infty_{t,x}}\\
  &+ 2C\int_0^t (t-s)^{-\frac{1}{2\alpha}}\| \theta^n u^n- \theta^m u^m\|_{L^\infty_x} ds.
\end{split}
\end{equation*}
Further, by adding and subtracting $\theta^n u^m$, one can derive the bound
\begin{equation}
\begin{split}
    &\|\theta^n - \theta^m\|_{L^\infty_{x}}+ \|u^n - u^m\|_{L^\infty_{x}} 
  \leq \|\theta_0^n - \theta_0^m\|_{L^\infty_{t,x}}+ \|u_0^n - u_0^m\|_{L^\infty_{t,x}}\\
  &\qquad + 2C\left( \|\theta^n\|_{L^\infty_{t,x}} + \|u^m\|_{L^\infty_{t,x}} \right)\int_0^t (t-s)^{-\frac{1}{2\alpha}} \left(   \|\theta^n - \theta^m\|_{L^\infty_{x}}+ \|u^n - u^m\|_{L^\infty_{x}} \right) \, ds.
\end{split}
\end{equation}
Thus, by Volterra-Gr\"onwall inequality as in \cref{lem:volterra-gronwall}, one has
\begin{equation*}
\begin{split}
    \|\theta^n - \theta^m\|_{L^\infty_{x}}+& \|u^n - u^m\|_{L^\infty_{x}} \\
    \leq \mu \left(\|\theta_0^n - \theta_0^m\|_{L^\infty_{t,x}}+ \|u_0^n - u_0^m\|_{L^\infty_{t,x}}\right)&\exp\left(C_\alpha \left(2C \left(\|\theta^n\|_{L^\infty_{t,x}} + \|u^m\|_{L^\infty_{t,x}} \right)^\mu\right) t\right), 
\end{split}
\end{equation*}
where $\mu$ and $C_\alpha$ are given by \eqref{def:mu-and-C-mu}.  
Next, invoking the $C^\gamma$ continuity of the initial data, by an approximation to the identity argument (see Lemma 8.14 of \cite{Folland}) we have, 
    $$\theta^\epsilon_0 \to \theta_0 \text{ and } u^\epsilon_0 \to u_0 \text{ as } \epsilon \to 0.$$ For fixed $T$, one can conclude $\{(\theta^n,u^n)\}_{n=1}^\infty$ forms a Cauchy sequence.
    
    Since the sequence converges uniformly and each element of the sequence satisfies a mild formulation, it is clear that the limit $(\theta,u)$ will also satisfy the mild solution definition. Therefore, for arbitrary $T>0$, there exists a mild solution $(\theta,u)$ to \eqref{e:SQG} on $[0,T]$ with initial data $(\theta_0,u_0)\in (L^\infty(\R^2))^3$, and with $\|\theta\|_{L^\infty_{t,x}} \leq \|\theta_0\|_{L^\infty_x}$.
\end{proof}
Finally, we show \cref{cor:solution-to-sqg}.
\begin{proof}[\textbf{Proof of \cref{cor:solution-to-sqg}}] \label{proof:solution-ssqg}
    The first conclusion holds from \cref{thm:existence-of-solutions}. The second conclusion follows from \cref{thm:regularity-of-theta-and-u}. The global existence of solutions follows from \cref{thm:global-in-time-solution}. The fact that $C^2$ solutions to \eqref{ssqg} satisfy the equation in a classical sense follows from \cref{P:lp-SQGMotivation}. 
\end{proof}

\appendix
\section{}\label{sec:appendix}

We conclude by showing the equivalence of the definitions of $\Lambda$ and $\Lambda_I$ on $C^2_b(\R^2)$.

\begin{lemma}\label{lem:fractional-laplacian-agreeing}
    Suppose that $f\in C^2_b(\R^2)$, then for $\alpha \in (1/2,1),$
    \[
\Lambda^{2\alpha} f \equiv \Lambda^{2\alpha}_I f.
\]
\begin{proof}
    As a consequence of Lemma 3.10 of \cite{kwasnicki2017ten} or Lemma 1 Section $V$ of \cite{stein1970singular}, it suffices to show that $\Lambda^{2\alpha}_I f$ is well defined as $\text{Dom}(\Lambda^{2\alpha}_I; L^\infty(\R^2)
    )\subset \text{Dom}(\Lambda^{2\alpha}; L^\infty(\R^2))$.  To that end, we first break the integral into two pieces,
    \begin{equation*}
    \begin{split}
         \Lambda^{2\alpha}_I f  &= \lim_{r\to 0^+} \int_{r<|h|<1} \underbrace{\frac{f(x+h)-f(x)}{|h|^{2+2\alpha}} dh}_{=:I(x)}  +  \underbrace{\int_{|h|>1} \frac{f(x+h)-f(x)}{|h|^{2+2\alpha}} dh. }_{=:II(x)}
    \end{split}
    \end{equation*}
We first work with $I(x)$ by using the differentiability of $f$. Rewrite $I(x)$ as
\begin{equation*}
\begin{split}
    I(x) = \lim_{r\to 0^+}\frac{1}{2}\int_{r<|h|<1} \frac{f(x+h)-2f(x)+f(x-h)}{|h|^{2+2\alpha}} dh.
\end{split}
\end{equation*}
Since \( f\in C^2_b(\mathbb{R}^2) \), by the Mean Value Theorem there exists some \( \xi \) on the line segment from \( x \) to \( x+h \) such that
\[
f(x+h)-f(x) = \nabla f(x+\xi h)\cdot h.
\]
Similarly, there exists some \( \zeta \) on the segment from \( x-h \) to \( x \) such that
\[
f(x)-f(x-h) = \nabla f(x-\zeta h)\cdot h.
\]
Subtracting the second equality from the first, we obtain
\[
f(x+h)-2f(x)+f(x-h) = \nabla f(x+\xi h)\cdot h - \nabla f(x-\zeta h)\cdot h.
\]
Now, we apply the Mean Value Theorem once more to the difference \(\nabla f(x+\xi h)-\nabla f(x-\zeta h)\). Since \( f\in C^2_b(\mathbb{R}^2) \) implies that \( \nabla f \) is Lipschitz, there exists some point \( \eta \) (lying on the line segment connecting \( x+\xi h \) and \( x-\zeta h \)) such that
\[
\nabla f(x+\xi h)-\nabla f(x-\zeta h) = D^2f(\eta)\bigl[(x+\xi h)-(x-\zeta h) = \bigr] =D^2f(\eta)\bigl[(\xi+\zeta)h \bigr].
\]
Noting that \( \xi,\zeta\in[0,1] \) so that \( \xi+\zeta\le 2 \), we deduce
\[
\left| f(x+h)-2f(x)+f(x-h) \right| \le \|D^2f\|_{L^\infty} \, (\xi+\zeta) \, |h|^2 
\le 2\,\|D^2f\|_{L^\infty}\, |h|^2.
\]
Thus, for every \( x\in\mathbb{R}^2 \) and \( h\neq 0 \) we have
\[
\frac{|f(x+h)-2f(x)+f(x-h)|}{|h|^{2+2\alpha}}
\le 2\,\|D^2f\|_{L^\infty}\,\frac{|h|^2}{|h|^{2+2\alpha}}
= 2\,\|D^2f\|_{L^\infty}\,\frac{1}{|h|^{2\alpha}}.
\]
Thus, since \( \alpha \in (1/2,1) \),
\[
 |I(x)| \leq 2\|D^2f\|_{L^{}\infty}\int_{|h|\le 1} \frac{dh}{|h|^{2\alpha}} <\infty.
\]

Turning to estimate $II$, we have 
\begin{equation*}
\begin{split}
    |II(x)| \leq \int_{|h|>1} \frac{| f(x+h)-f(x)|}{|h|^{2+2\alpha}} &\leq  \int_{|h|>1} \frac{| f(x+h)-f(x)|}{|h|^{2+2\alpha}} \\
    &\leq \|f\|_{L^\infty} \int_{|h|>1} \frac{1}{|h|^{2+2\alpha}} dh <\infty.
\end{split}
\end{equation*}

Hence, the principal-value integral
\[
I(x)=\frac{1}{2}\lim_{r\to 0^+} \int_{|h|>r} \frac{f(x+h)-2f(x)+f(x-h)}{|h|^{2+2\alpha}}\,dh
\]
converges absolutely for each \( x\in\mathbb{R}^2 \), and so the singular integral definition of the fractional Laplacian is well defined for \( f\in C^2_b(\mathbb{R}^2) \).

\end{proof}

\end{lemma}

\section*{Acknowledgments} DMA is grateful to the National Science Foundation
for support through grant DMS-2307638. EC is grateful to the Simons Foundation
for support through grant 429578.

\bibliographystyle{plain}
\bibliography{Refs}

\end{document}